\documentclass[11pt]{amsart}
\usepackage{amssymb, amscd, amsmath, amsthm, epsf, epsfig, latexsym,color} 
\usepackage{hyperref}
\usepackage{comment}
\usepackage{cancel}   
\usepackage{pb-diagram}

\newcommand{\ad}[1]{{\operatorname{ad}#1}} 

\newcommand{\TryPackage}[3]{\IfFileExists{#1.sty}{\usepackage{#1}#2}{#3}
}
\TryPackage{mathrsfs}{\renewcommand{\mathcal}{\mathscr}}{%
        \TryPackage{eucal}{}{}}

\newcommand{\ep}{\epsilon}

\newcommand{\bbi}{{{\bf i}}}
\newcommand{\bbj}{{{\bf j}}}
\newcommand{\bbk}{{{\bf k}}}

\newcommand{\ZZ}{{\mathbb Z}}
\newcommand{\RR}{{\mathbb R}}
\newcommand{\CC}{{\mathbb C}}

\newcommand{\RN}[1]{%
  \textup{\uppercase\expandafter{\romannumeral#1}}%
}

\newcommand{\Hom}{\operatorname{Hom}}

\newcommand{\nat}{\natural}

\newcommand{\Int}{\operatorname{Int}}
\newcommand{\Stab}{\operatorname{Stab}}

\newcommand{\Real}{\operatorname{Re}}

\graphicspath{ {figures/} }

\newcommand{\punctures}{{\{a_i,b_i\}_{i=1}^n}}

\theoremstyle{definition}

\newtheorem{df}{Definition}[section]

\theoremstyle{plain}

\newtheorem{thm}[df]{Theorem}

\newtheorem{cor}[df]{Corollary}
\newtheorem{lem}[df]{Lemma}
\newtheorem{prop}[df]{Proposition}

\date{}

\thanks{ }

\author{Christopher M. Herald}
\address{Department of Mathematics and Statistics,   University of Nevada, Reno, Reno, NV 89557} 
\email{herald@unr.edu}

 \author{Paul Kirk}
\address{Department of Mathematics, Indiana University, Bloomington, IN 47405} 
\email{pkirk@indiana.edu}


\begin{document}

\title[Holonomy perturbations in a cylinder and regularity]{Holonomy perturbations in a cylinder, and regularity for traceless SU(2) character varieties of tangles}

\thanks{The second author thanks the Simons Foundation for its support through the collaboration grant 278714 and the Max-Planck-Institut f\"ur Mathematik for its support during the Fall 2015 semester.}

\begin{abstract}The traceless $SU(2)$ character variety $R(S^2,\{a_i,b_i\}_{i=1}^n)$ of a $2n$-punctured 2-sphere
 is the symplectic reduction of a Hamiltonian $n$-torus action on the $SU(2)$ character variety of a closed surface of genus $n$. It is stratified with a finite singular stratum and a top smooth symplectic stratum of dimension $4n-6$.

For generic holonomy perturbations $\pi$, the traceless $SU(2)$ character variety $R_\pi(Y,L)$ of an $n$-stranded tangle $L$ in a homology 3-ball $Y$ is stratified with a finite singular  stratum and top stratum a smooth manifold. The restriction to  $R(S^2, 2n)$ is a Lagrangian immersion which preserves the   cone neighborhood   structure near the singular stratum.

For generic holonomy perturbations $\pi$, the variant $R_\pi^\natural(Y,L)$, obtained by taking the connected sum of $L$ with a Hopf link and considering $SO(3)$ representations with $w_2$ supported near the extra component, is a smooth compact manifold without boundary of dimension $2n-3$, which Lagrangian immerses into the smooth stratum of $R(S^2,\{a_i,b_i\}_{i=1}^n)$.

The proofs of these assertions consist of stratified transversality  arguments to eliminate non-generic strata in the character variety and to insure that the restriction map to the boundary character variety is also generic.   

The main tool   introduced  to establish abundance of holonomy perturbations is the use of holonomy perturbations along curves $C$ in a cylinder $F\times I$, where $F$ is a closed surface. When $C$ is obtained by pushing an embedded curve on $F$ into the cylinder, we prove that the corresponding holonomy perturbation induces  one of Goldman's generalized Hamiltonian twist flows    on the $SU(2)$ character variety $\mathcal{M}(F)$ associated to the curve $C$. 

\end{abstract}

 \maketitle
 
 \section{Introduction}
   The symplectic properties of character varieties of surfaces has been studied extensively, beginning with the work of  Atiyah-Bott \cite{AB} and Goldman \cite{goldman1}. Moreover, when a 3-dimensional manifold $Y$ has boundary surface $F$,   the character variety of $X$ gives rise to a Lagrangian submanifold in the character variety of $F$, although there are some issues with singularities complicating this picture.  In this paper we   establish the analogous symplectic properties  in the more challenging setting  of the  $SU(2)$   traceless character varieties of   a codimension two pair $(M,L)$, where $M$ is a homology 3-ball and $L$ is an  $n$-tangle, and its boundary codimension two pair $(\partial M, \partial L)=(S^2,\punctures)$. 
   
 We remind the reader that the traceless character variety $R(M,L)$ of a codimension 2 pair $(M,L)$ is the real-algebraic variety of conjugacy classes of $SU(2)$ representations of $\pi_1(M\setminus L)$ which send every meridian of $L$ to a traceless matrix (see Section \ref{background}). When $M$ is a 3-manifold,  a  variant $R^\nat(M,L)$ is defined  roughly by replacing   $L$ by its connected sum with a Hopf link (see Section \ref{piercedear}). The construction of $R^\nat$ was introduced in \cite{KM1} as a means to   ensure  that critical set of the Chern-Simons function is disjoint from the set of points with non-trivial stabilizer under the action of the group of gauge transformations. Finally, $R_\pi(M,L)$  and  $R_\pi^\nat(M,L)$ 
 denote the holonomy perturbed versions of these traceless character varieties (see Section \ref{background}).

 \medskip

The main results of this paper are Theorems B and C, stated below.   To explain the statements of these theorems, we first describe   another important  result of this article, one which only involves  character varieties of surfaces, and which should be of independent interest.      It is well known (\cite{goldman1}) that the character variety of a closed surface  is a stratified symplectic space.   Removing the $n$ handles from a genus $n$ closed surface $F$ produces a $2n$-punctured sphere.  We prove that  the traceless character variety of the $2n$-punctured 2-sphere 
is a symplectic reduction of the character variety of $F$.   
 We state the theorem somewhat imprecisely here and refer to Theorem \ref{surj} for a more careful statement.

\medskip

 \noindent{\bf Theorem A. }{\em 
 Let $F$ be a closed, oriented surface of genus $n$, and $S_0\subset F$ be the $2n$-punctured sphere obtained by removing tubular neighborhoods of $n$ disjoint essential curves in $F$.   
 Let $\mathcal{M}(F)$ denote the variety of conjugacy classes of $SU(2)$ representations of $\pi_1(F)$ and $R(S^2,\{a_i,b_i\}_{i=1}^n)$ the variety of conjugacy classes of $SU(2)$ representations of $S_0$ which send the boundary circles to traceless matrices. 
 
 Then there is a Hamiltonian $n$-torus action on   (an open subset of)  $\mathcal{M}(F)$ with moment  map $\mu:\mathcal{M}(F)\to \RR^n$, for which the  symplectic quotient  is $R(S^2,\{a_i,b_i\}_{i=1}^n)$. 
 }

 \medskip
 
We now can state our first main result, Theorem \ref{thm1},  in slightly simplified form as follows.

\medskip

 \noindent{\bf Theorem B. }{\em Assume $Y$ is a $\ZZ$-homology ball containing an $n$-strand tangle $L$, with $n\geq 2$. Then $R_\pi(Y,L)$ is compact for any perturbation $\pi$.    There
 exist arbitrarily small perturbations so that 
  $R_\pi(Y,L)$ is the union of two strata
$$  R_\pi(Y,L)= 
  R_\pi(Y,L)^{\ZZ/2,\ZZ/2}\sqcup R_\pi(Y,L)^{U(1),U(1)}, $$ 
with the following properties:  
 $  R_\pi(Y,L)^{\ZZ/2,\ZZ/2}$  a smooth manifold of dimension $2n-3$, and $R_\pi(Y,L)^{U(1),U(1)}$ a finite set. Each point in $R_\pi(Y,L)^{U(1),U(1)}$ has a neighborhood in $R_\pi(Y,L)$ homeomorphic to a cone on $\CC P^{n-2}$.
 
The restriction map  $  R_\pi(Y,L) \to R(S^2,\punctures)$  takes  the 0-manifold $R_\pi(Y,L)^{U(1),U(1)}$  into the 0-manifold $R(S^2,\punctures)^{U(1)}$,  and Lagrangian immerses the $(2n-3)$-manifold $R_\pi(Y,L)^{\ZZ/2,\ZZ/2}$ in the symplectic $(4n-6)$-manifold $R(S^2,\punctures)^{\ZZ/2}$.}

\medskip
Our second main result is  Theorem \ref{earthm}, the  analogue of Theorem B for $R^\nat(Y,L)$.   Theorem \ref{earthm} states the following.

\medskip

  \noindent{\bf Theorem C. }{\em Assume $Y$ is a $\ZZ$-homology ball containing an $n$-strand tangle $L$, with $n\geq 2$. Then $R^\nat_\pi(Y,L)$ is compact for any perturbation $\pi$.    There
 exist arbitrarily small perturbations so that 
  $R^\nat_\pi(Y,L)$ is   a smooth manifold of dimension $2n-3$, and the restriction map  $  R_\pi^\nat(Y,L) \to R(S^2,\punctures)$   is a Lagrangian immersion into the smooth stratum.}

 \medskip
 
Theorems B and C are  satisfying   results, in that they shows that although $   R(S^2,\punctures)$  and $  R_\pi(Y,L)  $ are not smooth symplectic (resp. Lagrangian immersed) manifolds,  they are the next best thing; namely, their singular strata are  finite  sets,   and the restriction map is stratum preserving.      In the case of the variant $R^\nat_\pi(Y,L)$, the situation is even nicer.  Theorem C says that $R^\nat_\pi(Y,L)$ is generically a smooth manifold which immerses into the top (smooth) stratum of $R(S^2\punctures)$.

\medskip

 When $n=2$, the results are simpler to state but still of considerable interest.  In \cite{HHK1,HHK2} the authors developed a Lagrangian-Floer theory for certain immersed curves in   the 2-dimensional variety   $R(S^2,\{a_1,b_1,a_2,b_2\})$, a space also known as the {\em pillowcase}. It is  a 2-sphere with four orbifold points obtained as the quotient of the 2-torus by the hyperelliptic involution.  
 
As a corollary of Theorems B and C when $n=2$, one has the following.

\medskip

 \noindent{\bf Corollary D. }{\em Given a 2-tangle in a homology 3-ball, there exist arbitrarily small holonomy perturbations $\pi$ so that $R_\pi(Y,L)$ is a compact 1-manifold with 2 boundary components and the restriction map $R_\pi(Y,L)\to R(S^2,\{a_1,b_1,a_2,b_2\})$ is an immersion taking the boundary points to the orbifold points of the pillowcase, and immersing the interior into the complement of the four orbifold points. 

Similarly, there exist arbitrarily small holonomy perturbations $\pi$ so that $R_\pi^\nat (Y,L)$ is a compact 1-manifold without boundary, and the restriction map $R_\pi(Y,L)\to R(S^2,\{a_1,b_1,a_2,b_2\})$ is an immersion which misses the four orbifold points. 
 }
 
  \medskip

 As explained in \cite{HHK2}, Corollary D is nearly sufficient in the case of $n=2$  to define a Lagrangian-Floer theory associated to 2-tangle decompositions of links.    The only remaining requirement is that the immersions be unobstructed, in the sense of \cite{FOOO}, but we defer the unobstructedness property to a future article.  
 
  In \cite{HHK2}, many examples were found where   $R_\pi(Y_1,L_1)$  and $R_\pi^\nat(Y_2, L_2)$  are transversely immersed  unobstructed  1-manifolds in the pillowcase, for which the Lagrangian-Floer theory applies to construct what we call the {\em pillowcase homology} of the tangle decomposition. In all  those  calculations, the pillowcase homology  agrees with known or conjectured calculations of reduced singular instanton homology.

 \medskip
 
 We expect that a similar Lagrangian-Floer theory can be constructed for all $n\ge 2$. 
 More precisely, our goal is to produce a tangle-theoretic counterpart to the (reduced) instanton knot Floer homology defined by Kronheimer-Mrowka \cite{KM1, KM2} for knots and links in 3-manifolds in terms of traceless character varieties.  We propose to take the Lagrangian-Floer homology of the Lagrangians $R_\pi(Y_1,L_1)$ and $R^\nat_\pi(Y_2, L_2)$ in the symplectic variety $R(S^2,\punctures)$.     Theorems B and C provide a guarantee   that traceless character varieties do indeed give rise to an Lagrangian intersection picture, after generic small perturbations.

\medskip

We briefly  outline of the proofs of  Theorems B and C.  
Consider a pair $(Y,L)$, where $L$ is an $n$-strand tangle in a 3-manifold $Y$ with 2-sphere boundary.  Denote by $X$ the complement of a tubular neighborhood of $L$ in $Y$ and by $F$ the boundary of $X$, a closed genus $n$ surface.  Set $S_0=\partial Y \setminus nbd(L)$.   Then the differential of the restriction map on $SU(2)$ character varieties from  $\mathcal{M}(X)$  to  $\mathcal{M}(F)$ has  Lagrangian image at each point.

 It is a general property of symplectic reduction that if $M$ is a  symplectic manifold with Hamiltonian $G$ action and moment map $\mu:M\to g^*$, and $i:L\to M$ is a Lagrangian immersion   which is transverse to the level set $\mu^{-1}(0)$, then $L\cap \mu^{-1}(0)$ Lagrangian immerses to the symplectic quotient. 
Therefore, if   the restriction $j:\mathcal M(X) \to \mathcal M(F)$ is a Lagrangian immersion transverse to $\mu^{-1}(0)$, for $\mu$  the    moment map of Theorem A, then $R(Y,L)=(\mu\circ j)^{-1}(0)$ Lagrangian immerses into the symplectic quotient $R(S^2,\punctures)$.

In general,  $j:\mathcal M(X) \to \mathcal M(F)$ need not be a Lagrangian immersion, even on its top stratum.  For example,  it is well known that the presence of incompressible surfaces in $X$ increases the dimension of $\mathcal{M}(X)$.  Moreover, even when  $j:\mathcal M(X) \to \mathcal M(F)$ is an immersion, it  need not be transverse to $\mu^{-1}(0)$.  

Fixing up $\mathcal{M}(X)$ and its restriction to $\mathcal{M}(F)$ in a manner consistent with the symplectic structure and compatible with the  perturbations of the Chern-Simons functional used to construct instanton Floer theory is accomplished by means of holonomy perturbations $\pi$. We first appeal to the results of \cite{herald1} to fix $\mathcal{M}(X)$; this prepares us for  the delicate part of  the argument, namely establishing the existence of arbitrarily small holonomy perturbations  $\pi$  making the restriction map $\mathcal{M}_\pi(X)\to \mathcal{M}(F)$ transverse to $\mu^{-1}(0)$ in a stratum preserving sense.

   Let $A_1,\dots, A_n$ be simple closed curves in the $2n$-punctured 2-sphere $S_0$ which form meridians to the $n$ components of $L$, that is, they are boundary curves to half the punctures. These form a half symplectic basis for $H_1(F)$, where $F$ is the closed surface obtained by adding $n$ handles to the boundary circles of $S_0$.   
The curves $A_1,\dots, A_n$ determine a function 
 $\mu:\mathcal{M}(F)\to \RR^n$ on the $SU(2)$ character variety of $F$ (Definition \ref{mumap}), essentially by taking a character around  each $A_i$.     Results of \cite{goldman2}  and \cite{Jeffrey-Weitsman}  are used to show
 that $\mu$ is (essentially) the moment map for a Hamiltonian $(S^1)^n$ action on $\mathcal{M}(F)$ with symplectic quotient $\mu^{-1}(0)/(S^1)^n$
 (essentially) the traceless character variety $R(S^2,\punctures)$.
 
 The results of \cite{herald1} show that, after appropriate holonomy perturbations $\pi$, the restriction map $j:\mathcal{M}_\pi(X)\to \mathcal{M}(F)$ is a Lagrangian immersion. We show that  (with further  perturbation) $j$ can be made transverse to $\mu^{-1}(0)$, and 
 $R_\pi(Y,L)$ is identified with $\mathcal{M}_\pi(X)\cap j^{-1} \left( \mu^{-1}(0)\right) $. Symplectic reduction then implies that the composite $R_\pi(Y,L)\to 
\mu^{-1}(0)/(S^1)^n=R(S^2,\punctures)$ is again a Lagrangian immersion.

  The proof of Theorem C is similar, but role of the map called $\mu$ above, the $n$-tuple of traces, is modified slightly to also include the anticommutativity condition between the earring meridian and the meridian of the strand about which the earring has been added.

\medskip
The difficulties in carrying out this outline arise in dealing with the parenthetical comments in the previous paragraphs.  The character varieties are not manifolds, but rather are stratified spaces, and one must work stratum-by-stratum to ensure the entire perturbed character variety has the appropriate structure after suitable holonomy perturbations. One  needs to perturb so that $j$ is transverse to $\mu^{-1}(0)$.  Hence much of the technical work consists of establishing that holonomy perturbations are sufficiently abundant to ensure that transversality holds in a stratified sense.

 In order to achieve transversality with $\mu^{-1}(0)$, we use perturbations supported in  a cylinder $F\times [0,1]$, so we examine  the homeomorphisms of the character variety $\Phi_{\pi}:\mathcal{M}(F)\to \mathcal{M}(F)$ induced by  holonomy perturbations $\pi$ along embedded curves pushed in from $F\times \{0\}$.   Perhaps surprisingly, these perturbations give rise to well known Hamiltonian isotopies of $\mathcal{M}(F)$. Namely, Theorem \ref{hamiltonian}  identifies these isotopies with the twist flows on flat moduli spaces of surfaces discovered by Goldman \cite{goldman1}. The statement is as follows. We refer the reader to Section \ref{sec2} for the construction of the maps $\Phi_{\pi}$.

\medskip
\noindent{\bf Theorem E.} {\em  Let $\pi_{C,t}$ be the 1-parameter family of holonomy perturbations $\pi_{C,t}=(N_C, t \phi)$ where $C\subset F$ is an embedded curve  and $\phi:\RR\to \RR $ a perturbation function.  
Let   $\Phi_{\pi_{C,t}}:\mathcal{M}(F)\to \mathcal{M}(F)$ be the corresponding isotopy. Then $\Phi_{\pi_{C,t}}$ restricts to a Hamiltonian isotopy on the smooth  stratum  $\mathcal{M}(F)^{\ZZ/2}$. In fact, $\Phi_{\pi_{C,t}}$ is equal to Goldman's Hamiltonian twist flow associated to $C$,  generated by the function 
$$f_C:\mathcal{M}(F)\to R, ~f_C([\rho])=\psi(\cos^{-1}\Real(\rho(C))),$$ for $\psi$ an antiderivative of $\phi$.
}

  \medskip
  
 The authors thank Lisa Jeffrey for discussions critical to the proof of Theorem A. They also thank Matthew Hedden, Henry Horton and Dan Ramras for illuminating discussions. 
 
\section{Character varieties,   perturbations, and stabilizers}\label{background}

Identify $SU(2)$ with the group of unit quaternions, and the Lie algebra $su(2)$ with the span of $\{\bbi,\bbj,\bbk\}$.  Every unit quaternion can be written in the form $e^{\alpha P}=\cos\alpha+\sin \alpha P$ for $P$ a purely imaginary unit quaternion $P$;  this description is unique, for unit quaternions different than $\pm 1$,  if we choose $0<\alpha < \pi$. Here, unit vectors in the Lie algebra correspond to purely imaginary quaternions of length one with respect to the positive definite inner product  
$$\langle v,w\rangle=-\Real (vw).$$   The function $\Real:SU(2)\to \RR$ on  unit quaternions 
corresponds to one half the trace on $SU(2)$ matrices.  Its point preimages are precisely the conjugacy classes in $SU(2)$.

\bigskip

Given a compact  2- or 3-manifold $M$, we    use the notation $\mathcal{M}(M)$ for  the space  of conjugacy classes of $SU(2)$ representations of $\pi_1(M)$,
$$ \mathcal{M}(M) =\Hom(\pi_1(M),SU(2))/_{\text{\rm conj}}  $$
and call $\mathcal{M}(M)$ the {\em character variety} of $M$.
Given a  properly embedded codimension two submanifold $L\subset M$, we call an element of $\pi_1(M\setminus L)$ a {\em meridian} if it is freely homotopic in $M\setminus L$ to the boundary of a 2-disk hitting $L$ transversely once.  We define a {\em traceless representation of $\pi_1(M\setminus L)$} to be an $SU(2)$ representation  which satisfies the following condition:  
\begin{equation}
\label{traceless}
\text{For each meridian } m\in \pi_1(M\setminus L), ~~~\Real(\rho(m))=0 \end{equation}
 We denote by $R(M,L)$  the space of 
conjugacy classes of traceless  representations of $\pi_1(M\setminus L)$: 
 \begin{equation}\label{RML}
 R(M,L) = \{\rho\in\Hom(\pi_1(M\setminus L),SU(2))~|~\rho \text{~satisfies~} (\ref{traceless}) \}/_{\text{\rm conj}}  
\end{equation}
and call $R(M,L)$ the {\em traceless character variety} of $(M,L)$.  Note that the condition (\ref{traceless}) is conjugation invariant, so it is not important how these meridians are connected to the chosen base point (in order to view them as representing elements of $\pi_1$).  

\medskip

We will need to use holonomy-perturbed versions of these varieties when $M$ is an oriented  3-manifold.     Fix $k>3\operatorname{genus} (\partial M )$.  Denote by $\mathcal{X}$ the Banach space of {\em perturbation functions}
\begin{equation}
\label{pertfun}\mathcal{X}=\{f:\RR\to \RR~|~\text{ $f$ is   $C^k$, odd, $2\pi$-periodic}\}
\end{equation}
Each $f$ defines a conjugation equivariant function 
$F:SU(2) \to SU(2)$ by 
\begin{equation}
\label{conjinv}F(e^{\alpha Q}) = e^{f(\alpha) Q}.
\end{equation}

Given a 3-manifold $M$, define {\em perturbation data, $\pi=\{( N_i , f_i)\}_{i=1}^p $,  for $M$} to be  a finite collection of disjoint  orientation preserving embeddings $N_i:S^1\times D^2\subset \Int(M)$, and for each embedding, a choice $f_i\in \mathcal{X}$.   We call the collection of solid tori 
$\sqcup_i(N_i(S^1\times D^2))$ the {\em support} of the perturbation $\pi$, and abbreviate it to $\sqcup_i N_i$.

 Define a {\em $\pi$-perturbed representation of $M$} to be  a representation 
 $\rho:\pi_1( M\setminus  (\sqcup_i N_i))\to SU(2)$ which satisfies the {\em perturbation condition:} 
 \begin{equation}\label{pert}
  \rho(\mu_i)=F_i(\rho(\lambda_i)), ~~i=1, \dots, p,    
\end{equation}
where $\mu_i=N_i( \{1\}\times \partial D^2)$ and $\lambda_i=N_i(S^1\times \{1\})$, and $F_i$ is associated to $f_i$ as in Equation (\ref{conjinv}).
Like condition (\ref{traceless}), condition (\ref{pert}) is conjugation independent and hence is independent of the choice of path from $N_i(1,1)$ to the base point used to define $\pi_1( M\setminus  (\sqcup_i N_i))$.  

We denote by $\mathcal{M}_\pi(M)$ the perturbed character variety:
$$\mathcal{M}_\pi(M)=\{\rho\in \Hom(\pi_1(M\setminus(\sqcup_iN_i)),SU(2))~|~ \rho \text{ satisfies } (\ref{pert})\}/_{\text{\rm conj}}.
$$
Similarly, if $M$ contains a properly embedded codimension two submanifold $L$ and the embeddings $N_i$ miss $L$, then we denote by $R_\pi(M,L)$ the space of conjugacy classes of $\pi$-perturbed traceless representations.
$$R_\pi(M,L)=
\{\rho\in \Hom(\pi_1\big(M\setminus(L\cup(\sqcup_iN_i))\big),SU(2))~|~ \rho \text{ satisfies } (\ref{traceless}) \text{~and~}(\ref{pert})\}/_{\text{\rm conj}}.
$$    
If $f_i=0$ for all $i$, then $\mathcal{M}_\pi(M)$ and   $R_\pi(M,L)$ are naturally identified with $  \mathcal{M}(M)$ and  $R(M,L)$, respectively.

\medskip

For an illustration of the effect perturbations have on traceless character varieties,  we offer the reader the following instructive examples. In \cite[Section 11.6]{HHK2} the space $R(Y,L)$ for a certain 2-tangle in a 3-ball associated to the $(3,4)$ torus knot is identified, it is a singular real algebraic variety homeomorphic to the letter $\phi$. A 1-parameter family of perturbations $\pi_t$ is described, so that $R_{\pi_t}(Y,L)$ is a homeomorphic to the disjoint union of an interval and a circle when $t\ne 0$.   The reader should keep this example in mind when trying to understand the statement of Theorem \ref{thm1}.   In particular, neighborhoods of the endpoints of the interval can be viewed as cones on $\CC P^0$.    The second example concerns the case when $T$ is the trivial 2-tangle in a 3-ball $B$. Then $R^\nat(B,T)$ is homeomorphic to a 2-sphere (see Section \ref{piercedear} for the definition of $R^\nat$, a variant of $R$), and Theorem 7.1 of \cite{HHK1}   shows that there exists a 1-parameter family of perturbations $\pi_t$ so that $R^\nat_{\pi_t}(B,T)$ is a smooth circle whenever $t\ne 0$.  This second example illustrates the content of Theorem \ref{earthm}.  In both cases, perturbations serve to break a symmetry on the unperturbed varieties: in the first case a $\ZZ/2$ symmetry on the letter $\phi$ with an arc of fixed points, and in the second case the rotational $S^1$ symmetry on $S^2$.

 \medskip
 
 If $i:Z\subset M$ is a connected subspace, then the map $i_*:\pi_1(Z)\to \pi_1(M)$ induces a restriction map $i^*:\mathcal{M} (M) \to \mathcal{M}(Z)$.    We will also call the analogous maps  $R(Y,L)\to R(\partial Y, \partial L)$ and $R_\pi(Y,L)\to R(\partial Y, \partial L)$ in the traceless context restriction maps.
 
 \medskip
 
Under the conjugation action, a representation $\rho$ has stabilizer  either isomorphic to  $\ZZ/2$,   a maximal torus $U(1)\subset SU(2)$, or the entire group $SU(2)$; we call  $\rho$ {\em irreducible},  {\em abelian},  and {\em central} in these respective cases. By this convention, a central representation is not called abelian.   

Denote by $\mathcal{M}_\pi(M)^G$ the subspace of conjugacy classes of representations with stabilizer $G$. Stabilizers determine a partition    $$\mathcal{M}(M)=\mathcal{M}(M)^{\ZZ/2}\sqcup \mathcal{M}(M)^{U(1)}
\sqcup\mathcal{M}(M)^{SU(2)}.$$
 When $M$ is a 3-manifold and $\pi$  perturbation data, one obtains a similar partition of $\mathcal{M}_\pi(M)$. 
If $(M,L)$ is a codimension two proper pair with $L$ nonempty,  a traceless representation cannot be central. Hence 
$$R(M,L)=R(M,L )^{\ZZ/2}\sqcup R(M,L )^{U(1)}.$$

Restricting  to a subspace need not preserve stabilizers.  The stabilizer of the restriction of a representation $\rho$ to a subspace may be larger than  the stabilizer of $\rho$. 
For a  3-manifold $M$ with nonempty boundary we may therefore refine the partition of $\mathcal{M}(M)$ to a partition indexed by two subgroups, namely the stabilizer and the stabilizer of the restriction to the boundary.  For example, $\mathcal{M}_\pi(M)^{ \ZZ/2,SU(2) }$ denotes the subspace of conjugacy classes of $\pi$ perturbed representations which are irreducible and restrict to central representations on the boundary.    We use similar decorations on $R_\pi (M,L)$, e.g.,  $R_\pi (M,L) ^{ \ZZ/2,U(1)}$ denotes the set of traceless perturbed representations that are irreducible on $\pi_1(M\setminus L)$ but which have abelian restriction to $\pi_1(\partial M \setminus \partial L)$.

\section{Tangent spaces}\label{tangentspaces}

We remind the reader  of the relationship   between tangent spaces of character varieties and cohomology.     For any space $M$, the Zariski tangent space of  $\mathcal{M}(M)$ at a representation $\rho$ is identified with the cohomology group $H^1(M;su(2)_{\ad{\rho}})$ (see  \cite{weil}). We outline some aspects of this identification and indicate how to generalize this in the context of perturbations.    

A presentation $\langle x_1,\dots, x_n~|~ w_1,\dots w_r\rangle$ of $\pi_1(M)$ determines a {\em relation} map 
$$R:SU(2)^n\to SU(2)^r, ~~ (X_1, \dots, X_n) \mapsto (w_1(X_1, \dots, X_n), \dots, w_r(X_1, \dots, X_n))$$
so that  the assignment of an element $X_i\in SU(2)$ to each generator $x_i$ determines a representation if and only if $$R(X_1\dots, X_n)=(1,1,\dots,1)={\bf 1}.$$   This gives identifications 
$$\Hom(\pi_1(M),SU(2))=R^{-1}({\bf 1}) \text{ and } \mathcal{M}(M)=R^{-1}({\bf 1})/_{\text{\rm conj}}.$$
 
The presentation for $\pi_1(M)$ determines a 2-complex $K_M$ with one 0-cell, $n$ 1-cells, and $r$ 2-cells. 
A representation $\rho:\pi_1(M)\to SU(2)$ can be composed with the adjoint representation $\ad :SU(2)\to {\rm Aut}(su(2))$ to determine a local coefficient system    on the 2-complex $K_M$ which we denote by $\ad{\rho}$.
This data determines the cellular cochain complex for $K_M$ (see e.g. \cite{DK})  whose chain groups are given by:
$$C^0(K_M;su(2)_{\ad{\rho}})=su(2),~C^1(K_M;su(2)_{\ad{\rho}})=\operatorname{Funct} (\{x_1,\dots, x_n\},su(2))=su(2)^n, $$
and
$$C^2(K_M;su(2)_{\ad{\rho}})=\operatorname{Funct} (\{w_1,\dots, w_r\},su(2))=su(2)^r.$$
   Here, the vector space $su(2)$ is viewed a $\pi_1(M)$ module via the adjoint action.    The differential on $0$-cochains is given by 
$$d^0:C^0(K_M;su(2)_{\ad{\rho}})\to C^1(K_M;su(2)_{\ad{\rho}}),~ d^0v=((\ad{\rho}(x_1)-1)v,\dots,(\ad{\rho}(x_n)-1)v).$$

Right translation by the $n$-tuple $(X_1^{-1},\dots, X_n^{-1})\in SU(2)^n$ identifies 
$T_{(X_1, \dots, X_n)} SU(2)^n$ 
with $T_{\bf 1}SU(2)^n=su(2)^n$. The differential  $d^0$ can then be identified with the derivative at $\rho$ of the orbit map $o:SU(2)\to SU(2)^n, ~ o(g)=(\ad{\rho}(g)(X_1),\dots , \ad{\rho}(g)(X_1)).$ The tangent space to the orbit 
is therefore identified with the 1-coboundaries 
$B^1(K_M;su(2)_{\ad{\rho}}), $ and the tangent space to the stabilizer Stab$(\rho)$ is identified 
with $H^0(K_M;su(2)_{\ad{\rho}})=H^0(M;su(2)_{\ad{\rho}})$.

Similarly, the identifications of  the tangent spaces $SU(2)^n$ and $SU(2)^r$ with $su(2)^n$ and $su(2)^r$, respectively, allow us to 
identify the differential $d^1:C^1(K_M;su(2)_{\ad{\rho}})\to C^2(K_M;su(2)_{\ad{\rho}})$ with the derivative $dR_\rho$.

Hence  at smooth points of $\mathcal{M}(M)$, 
\begin{equation}
\label{andreweil}
T_\rho\mathcal{M}(M)=H^1(K_M;su(2)_{\ad{\rho}})=H^1(M;su(2)_{\ad{\rho}}).
\end{equation}
 We take Equation (\ref{andreweil})  as a definition of $T_\rho\mathcal{M}(M)$ at singular points.

\medskip

If the 2-complex determined by the presentation is aspherical,  then $$H^2(\pi_1(M);su(2)_{\ad \rho})=H^2(K_M;su(2)_{\ad{\rho}})= \operatorname{coker} dR_\rho.$$ 
If in addition $M$ is an aspherical manifold, these groups equal $H^2(M;su(2)_{\ad{\rho}})$.    For example, when $M$ is an oriented  2-manifold (other than $S^2$) and the presentation is given by a cell structure of $M$ with one 0-cell, then $H^i(M;su(2)_{\ad{\rho}})=H^i(\pi_1(M);su(2)_{
\ad{\rho}})=H^i(K_M;su(2)_{\ad{\rho}})$ for all $i$.
 In general, however, the $i$th cohomology of $M$,  $\pi_1M$ and $K_M$ for $i>1$ need not agree.

\medskip

We next indicate how to introduce perturbations into this perspective.
Suppose that $M$ is a 3-manifold with given   perturbation data $\pi=\{(N_i, f_i)\}_{i=1}^p$. Given a presentation 
$$\pi_1(M\setminus (\sqcup_i N_i))=\langle x_1,\dots, x_q~|~ w_1,\dots w_s\rangle ,$$ express the meridians $\mu_i$ and longitudes $\lambda_i$  of $N_i$ as words in the generators $x_\ell$.  Then
the relation map $R: SU(2)^q\to SU(2)^s$
can be augmented to   $$R_\pi=R\times(P_1,\dots, P_p):SU(2)^q\to SU(2)^s\times SU(2)^p$$
where  \begin{equation}
\label{pertrel}P_i(X_1,\dots, X_n) =F_i(\lambda_i(X_1,\dots, X_n)) \mu_i(X_1,\dots, X_n)^{-1}.
\end{equation}
It is easy to see that 
\begin{equation}
\label{compact}
\mathcal{M}_\pi(M)=R_\pi^{-1}({\bf 1})/_{\text{\rm conj}}.
\end{equation}
 Notice that this implies that $\mathcal{M}_\pi(M)$ is compact for any perturbation data $\pi$, since it is the quotient by $SU(2)$  of a closed subset of $SU(2)^q$.    

For $\rho\in\mathcal{M}_\pi(M)$, define 
 \begin{equation}
\label{zariski}H^1_\pi(M;su(2)_{\ad{\rho}})=\ker  (dR_\pi)_\rho /B^1(M;su(2)_{\ad{\rho}}).
\end{equation}
Then $$T_\rho\mathcal{M}_\pi(M)=H^1_\pi(M;su(2)_{\ad{\rho}}).$$ When all the perturbation functions $f_i$ are zero,  $H^1_\pi(M;su(2)_{\ad{\rho}})=H^1(M;su(2)_{\ad{\rho}})$.
One can also define the the perturbed $0$th cohomology as before: 
\begin{equation}
\label{h0pi}H^0_\pi(M;su(2)_{\ad{\rho}})=\{ v\in su(2)~|~
\ad{\rho}(x_i)(v)=v \text{ for all } x_i\}.
\end{equation}

The cellular chain complex  for the 2-complex $K_M$ associated to the presentation is not adequate to compute the second cohomology when $M$ is a 3-manifold, and the introduction of perturbations $\pi$ makes it difficult to present a clean definition of $H^2_\pi(M; su(2)_{\ad{\rho}})$ in terms of cellular chains. 
In light of these difficulties, we instead refer to \cite{herald1} for a definition of $H^i_\pi(M; su(2)_{\ad{\rho}})$ as the cohomology  of a Fredholm complex constructed by deforming the twisted de Rham complex.  To give a full definition would take us too far afield.  

We will use several facts about these groups.  First, the 0th and 1st cohomology are canonically isomorphic with the definitions (\ref{zariski}) and (\ref{h0pi}).  Second, if $\rho\in \mathcal{M}_\pi(M)$, then there is a Kuranishi map 
\begin{equation}
\label{kurpi}
K:H^1_\pi(M; su(2)_{\ad{\rho}})\to H^2_\pi(M; su(2)_{\ad{\rho}}),
\end{equation}
equivariant with respect to the action of the stabilizer ${\rm Stab}(\rho)$ of $\rho$, so that, locally near $\rho$,  
$$\mathcal{M}_\pi(M)\cong  K^{-1}(0)/{\rm Stab}(\rho).$$
 Third, if $\rho$ takes values in the diagonal circle subgroup of $SU(2)$, then the adjoint action on $su(2)$ splits equivariantly with respect to the splitting $su(2)=\RR\bbi \oplus \CC\bbj$; the action is trivial on the $\RR$ summand and weight two on the complex summand. The corresponding Fredholm complex splits accordingly, and 
\begin{equation}
\label{cohosplit}
H^i_\pi(M; su(2)_{\ad{\rho} })=H^i_\pi(M;\RR)\oplus H^1_\pi(M; \CC_{\ad{\rho}})\
\end{equation}

Finally, we will use a upper semicontinuity property of the dimensions of these cohomology groups: $\dim H^i_\pi(M; su(2)_{\ad{\rho}})\leq \dim H^i_{\pi_0}(M; su(2)_{\ad{\rho_0}})$ for all $\pi$ close enough to $\pi_0$ and for $\rho\in \mathcal{M}_\pi(M)$ close enough to $\rho_0 \in \mathcal{M}_{\pi_0}(M)$.  Again we refer to \cite{herald1} for a careful description of the topology on the space of perturbations. For our purposes it suffices to compare perturbations $\pi=\{(N_i,f_i)\}$ and $\pi'=\{(N_i',f_i')\}$ for which the embeddings coincide, i.e. $N_i=N_i'$, in which case we can measure their distance using the $C^k$ metric on the perturbation functions $f_i, f_i'\in \mathcal{X}$.  The distance between $\rho\in \mathcal{M}_\pi(M)$ and $\rho'\in \mathcal{M}_{\pi'}(M)$ can be taken to be the distance between the  $q$-tuples $(\rho(x_1),\dots, \rho(x_q))$ and $(\rho'(x_1),\dots, \rho'(x_q))$ in $SU(2)^q$ for a set of generators $x_i$ of $\pi_1(M\setminus \sqcup_i N_i)$.

 \section{Curves on a surface}\label{snakesonaplane}
 
The aim of this paper is to establish certain transversality results for the perturbed character variety for an $n$-strand tangle in a homology 3-ball. 
 The complement of a tubular neighborhood of  $n$-strand tangle in a homology 3-ball $Y$ is a 3-manifold $X$ with boundary a genus $n$ surface $F$, and $F$ is the union of $S_0=F\cap \partial Y$, a 2-sphere with $2n$ open disks removed, and $n$ cylinders.   The $2n$ boundary circles of $S_0$ are paired by the cylinders.

To keep careful track of curves on $F$ and $S_0$ and paths connecting them to a base point, we identify $S^2$ with $\RR^2$ together with a point at infinity.  
 Then $S^2$ can   be decomposed as a union of sectors $S_1,\dots, ,S_n$, ordered counterclockwise, i.e., in polar coordinates $S_\ell=\left\{(r,\theta)~\left|~ \theta\in \left[\frac{2\pi(\ell-1)}{n}, \frac{2\pi \ell} {n}\right] \right.   \right\}$.  Let $a_i,b_i$ denote a pair of points lying on the central ray of each sector, and remove a pair of small disjoint  disk neighborhoods of  each $a_i$ and $b_i$ to obtain $S_0$.  Attaching cylinders to  each pair of boundary circles  yields an oriented closed surface $F$ of genus $n$, containing the $2n$-punctured 2-sphere $S_0$.    The sector indexing should be viewed as a cyclic ordering. 
 \begin{figure}[h]
\begin{center}
\def\svgwidth{3.4in}
 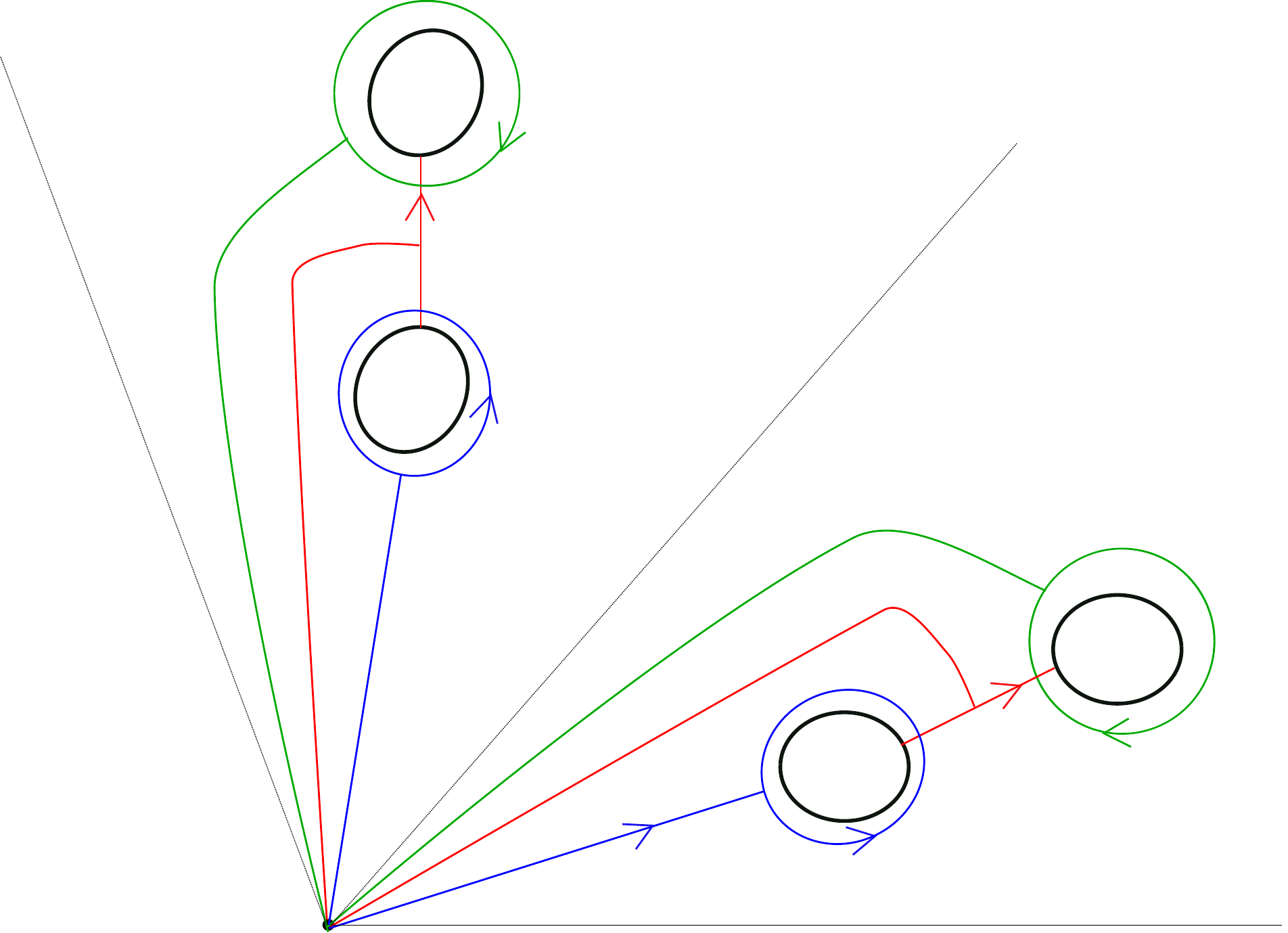
 \caption{ \label{drawingfig} }
\end{center}
\end{figure}

 Figure \ref{drawingfig}
 portrays, in the first two sectors,  embedded simple closed curves $A_i, B_i, D_i$, $i=1,2$, each with its own  arc $\gamma_{A_i}, \gamma_{B_i}, \gamma_{D_i}$ from  the central base point to the curve.   Make analogous choices in each sector.
 
 To keep the notation unencumbered, whenever we consider $A_i$ as an element of $\pi_1(S_0)$ or $\pi_1(F)$, we always mean the representative based loop $\gamma_{A_i}*A_i*\gamma_{A_i}^{-1}$, and similarly for $B_i$ and $D_i$. With these choices,
\begin{equation}
 \label{eq5}
\pi_1(S_0)=\langle A_i,B_i~|~ \textstyle \prod\limits_{i=1}^n A_iB_i^{-1}=1\rangle,
\end{equation}
\begin{equation*}
\pi_1(F)=\langle A_i, D_i~|~\textstyle \prod\limits_{i=1}^n[A_i,D_i]=1\rangle,
\end{equation*}
and the inclusion  $\pi_1(S_0)\to \pi_1(F)$ 
  is given by $$A_i\mapsto A_i, ~  B_i\mapsto D_i A_i D_i^{-1}.$$ 
 With the standard orientation of $S^2$ (represented by the standard orientation on $\RR^2 \setminus \{2n \mbox{ disks} \}$), $A_i\cdot D_i=-1$, $A_i\cdot D_j=0$ for $i\ne j$,   $A_i\cdot A_j=0=D_i\cdot D_j$.  The curves $A_i,B_i$ form meridians to the $a_i,b_i$.

\section{The structure of  the traceless character variety of the punctured sphere}
\label{pokeholesinit}

  \medskip

As in Section \ref{snakesonaplane}, let $S_0$  denote the  2-sphere with disk neighborhoods of  points $\{a_i,b_i\}_{i=1}^n$ removed, and $F$ be the closed, oriented surface  of genus $n$ obtained by attaching an annulus to  each pair of boundary curves. 
We will show that the traceless character variety $R(S^2,\punctures)$ has $2^{2n-2}$ singular points.  Moreover,  away from these singular points $R(S^2,\punctures)$ will be identified with the symplectic reduction of $\mathcal{M}(F)$
with respect to a Hamiltonian torus action.
  We remark that the results of this section are purely 2-dimensional; they do not  refer to any 3-manifold or to holonomy perturbations.

\medskip

It is well known that $\mathcal{M}(F)$ is a stratified real-algebraic variety.  In fact, the stabilizer decomposition 
\begin{equation}
\label{stabde}
 \mathcal{M}(F)=\mathcal{M}(F)^{\ZZ/2}\sqcup\mathcal{M}(F)^{U(1)}\sqcup\mathcal{M}(F)^{SU(2)}
\end{equation}
 is a decomposition into smooth symplectic  manifolds. The irreducible stratum $\mathcal{M}(F)^{\ZZ/2}$ has dimension $6n-6$ (see \cite{goldman1}),  the abelian stratum $\mathcal{M}(F)^{U(1)}$ has dimension $2n$, and the central stratum $\mathcal{M}(F)^{SU(2)}$ is a finite set containing $2^{2n}$ points (see \cite{goldman3}). Tangent spaces to these strata are identified with the  invariant subspace of the first cohomology,  
 $T_\rho\mathcal{M}(F)\cong H^1(F;su(2)_{\ad{\rho}})^{\Stab(\rho)} $,
and the 
symplectic structure on each stratum is given by the (restriction to this tangent space of the cup product composed with the inner product $\langle ~,~ \rangle:su(2)\times su(2)\to\RR $, that is,     
$$\omega :(\alpha, \beta) \in H^1(F;su(2)_{\ad{\rho}})\times H^1(F;su(2)_{\ad{\rho}})\mapsto - \Real (\alpha \cup \beta )\cap  [F] \in \RR.$$

 The traceless character variety $R(S^2,\punctures)$ has a decomposition 
\begin{equation}
\label{strat}
R(S^2,\punctures)= R(S^2,\punctures)^{\ZZ/2} \sqcup R(S^2,\punctures)^{U(1)}
\end{equation}
into the irreducible and abelian representations.

\begin{prop}\label{lin} The abelian stratum $R(S^2,\punctures)^{U(1)}$ consists of $2^{2n-2}$ points, and the irreducible stratum $R(S^2,\punctures)^{\ZZ/2}$ is a smooth manifold of dimension $4n-6$. \qed
\end{prop}

\begin{proof} 
 If $\rho\in R(S^2,\punctures)^{U(1)}$, then $\rho$ is conjugate to a unique representation which takes $A_1$ to $\bbi$ and each $A_i,~i=2,\dots, n$ and $B_i,~i=1,\dots ,n$ to $\pm \bbi$. Half of the resulting $2n-1$ signs satisfy the relation in  (\ref{eq5}).  This proves the first claim.
 The rest of the  proof of Proposition \ref{lin}  can be found in \cite{Lin} or \cite{Heusener-Kroll}.     It also follows from Theorem \ref{surj} below.  
 \end{proof}

The inclusion $S_0 \subset F$  
induces a restriction map
 \begin{equation}\label{rest}
r:\mathcal{M}(F)\to \mathcal{M}(S_0).
\end{equation}
The traceless character variety $R(S^2,\punctures)$ is identified with the subspace of $\mathcal{M}(S_0)$ consisting of  
representations $\rho$ satisfying the traceless condition (\ref{traceless}) around the punctures.

 The  restriction map (\ref{rest}) is neither surjective nor injective, but  the image of $r$ contains $R(S^2,\punctures)$.  To see this, note that any two traceless $SU(2)$ elements are conjugate.  Hence given any  $\rho\in R(S^2,\punctures))$, there exist $d_ 1, \dots, d_n\in SU(2), $  such that $\rho(B_i)=d_i\rho(A_i) d_i^{-1}$.  Setting $\tilde\rho(D_i)=d_i$, $\tilde \rho(A_i)=\rho (A_i)$  defines   $\tilde\rho\in\mathcal{M}(F)$ satisfying $r(\tilde\rho)=\rho$.

   \bigskip
   
 The rest of this section is devoted to the proof that
 $R(S^2,\punctures)$ is the symplectic quotient of $\mathcal{M}(F)$ by  a 
 torus action.  We define the moment map and torus action next.

\begin{df}\label{mumap} Let
  $\mu:\mathcal{M}(F)\to \RR^n$ be the map
 \begin{equation*} 
\mu(\rho)=(-\sin^{-1}(\Real(\rho(A_1))),\dots,-\sin^{-1}(\Real(\rho(A_n)))).
\end{equation*}
Notice that if $Q$ is a purely imaginary unit quaternion and $s\in [0,\pi]$, then $\Real(e^{sQ})= \cos s$, and hence $-\sin^{-1}(\Real(e^{sQ}))=s-\tfrac\pi 2$.\end{df}

\medskip
We next introduce a torus action on an open, dense subset of $\mathcal{M}(F)$. 

\begin{df}
Let $\mathcal{M}(F)_0\subset \mathcal{M}(F)$ denote the open subset of representations $\rho$ satisfying $\rho(A_i)\ne \pm1$ for each $i$. Notice that $\mu^{-1}(0)\subset \mathcal{M}(F)_0$.
Define 
\begin{equation}\label{Hmap}
H:SU(2)\setminus\{\pm 1\}\to su(2),~H(e^{s Q})=Q \text{ for } s\in(0,\pi) \text{ and } Q   \in su(2) \text{ with } \| Q \|=1.  \end{equation}
That is, $H(g)=(g-\Real(g))/\| g-\Real(g) \|$. 
Note that $H(hgh^{-1})=hH(g)h^{-1}$ for all $g,h\in SU(2)$.
We then define an   $\RR^n$ action
on  $\mathcal{M}(F)_0$ on the right by 
\begin{equation}
\label{actionjackson}
(\rho \cdot {\bf t})(A_i)= \rho (A_i)\text{ and }
(\rho \cdot {\bf t})(D_i)= \rho  (D_i)e^{t_i  H(\rho(A_i))} \text{~ for all~} i,
 \text{ where } {\bf t}=(t_1,\dots, t_n).
\end{equation}
 This action is periodic with period $2\pi$ in each factor, so it induces an action of the $n$-torus ${\bf T}^n=(S^1)^n$ on $\mathcal{M}(F)_0$.      Since  $$g e^{t_i  H(\rho(A_i))}g^{-1}=e^{t_ig  H(\rho(A_i))g^{-1}}=e^{t_i H( g\rho(A_i) g^{-1} )} ,$$   this induces an action on conjugacy classes. 
\end{df}
\bigskip

Following \cite{goldman1, Jeffrey-Weitsman}, we prove the following theorem. The authors thank Lisa Jeffrey for help  with the   argument.

\begin{thm} \label{surj}  The ${\bf T}^n$ action has the following properties:
\begin{enumerate} 
\item  \label{pa1} The restriction $r:\mathcal{M}(F)_0\to \mathcal{M}(S_0)$ induces a homeomorphism $$\mu^{-1}(0)/{\bf T}^n\cong R(S^2,\punctures).$$ 
\item  \label{pa2} 
The ${\bf T}^n$ action is free on the preimage  $r^{-1}(\mathcal{M}(S_0)^{\ZZ/2})$, and the stabilizer of points in  the preimage  $\mathcal{M}(F)^{\ZZ/2} \cap r^{-1}(\mathcal{M}(S_0)^{U(1)})$    is $S^1$.

\item  \label{pa3}  The ${\bf T}^n$ action on the smooth $(6n-6)$-dimensional irreducible stratum $\mathcal{M}(F)_0^{\ZZ/2}$ is Hamiltonian, and $\mu:\mathcal{M}(F)_0^{\ZZ/2}\to \RR^n$ is a moment map for this action.

\item  \label{pa4}  The restriction $\mu'=\mu|_{\mathcal{M}(F)^{\ZZ/2}_0\cap  r^{-1}(\mathcal{M}(S_0)^{\ZZ/2})}$ has $0\in \RR^n$ as a regular value.
 Under the identification in {\rm (i)}, the corresponding (Marsden-Weinstein)   symplectic quotient $(\mu')^{-1}(0)/{\bf T}^n$ is the smooth $(4n-6)$-dimensional symplectic manifold $R(S^2,\punctures)^{\ZZ/2}.$   \end{enumerate} 

\end{thm}
 \begin{proof}

Since $e^{t_iH(\rho(A_i))}$ commutes with $\rho(A_i)$,
$$(\rho \cdot  {\bf t})(B_i)=( \rho \cdot {\bf t})(D_iA_iD_i^{-1})=\rho(D_iA_iD_i^{-1})= \rho(B_i),$$ and so   $r:\mathcal{M}(F)_0\to \mathcal{M}(S_0)$ factors through the orbit map $\mathcal{M}(F)_0\to \mathcal{M}(F)_0/{\bf T}^n$. Conversely, if   $\rho_1, \rho_2\in \mathcal{M}(F)_0$ are  two representations whose restrictions  $r(\rho_1)$ and $r(\rho_2)$ to  $\pi_1(S_0)$ are conjugate, then there exists $g\in SU(2)$ so that
$$g\rho_1(A_i)g^{-1}=\rho_2(A_i)\text{ and }g\rho_1(D_iA_iD_i^{-1})g^{-1}=\rho_2(D_iA_iD_i^{-1}) \text{  for each $i$.}$$
  It follows that, for each $i$, $\rho_2(D_i)^{-1}g\rho_1(D_i)g^{-1} $ commutes with $ \rho_2(A_i)$ and hence  there exists  a $t_i$ so that  $\rho_2(D_i)^{-1}g\rho_1(D_i)g^{-1}  =e^{t_iH(\rho_2(A_i))}.$
Therefore, $  g\rho_1 g^{-1}=\rho_2\cdot{\bf t}$, where ${\bf t}=(t_1, \dots, t_n)$. 
We conclude that the restriction $r:\mathcal M(F)_0 \to \mathcal M (S_0)$ induces a homeomorphism  
 $\mathcal{M}(F)_0/{\bf T}^n  $ onto its image in $ \mathcal{M}(S_0)$.  
 
As was observed earlier, the image of $r:\mathcal M (F) \to \mathcal M(S_0)$ contains $R(S^2, \punctures)$.  If $\rho\in \mathcal{M}(F)$ satisfies $\mu(\rho)=0$, then $\Real(\rho(A_i))=0$ and so $\rho\in \mathcal{M}(F)_0$. Moreover, $\Real(\rho(B_i))=\Real(\rho(D_i A_iD_i^{-1}))=0$ so that $r(\rho)\in R(S^2,\punctures)$.  Hence $r$ sends $\mu^{-1}(0)$ onto $R(S^2,\punctures)$ and therefore $\mu^{-1}(0)/{\bf T}^n\cong R(S^2,\punctures)\subset \mathcal{M}(S_0).$ This proves claim (i).

\medskip
 We next show that the action is free on $r^{-1}(\mathcal{M}(S_0)^{\ZZ/2})$ and has $S^1$ stabilizer at each point in  $r^{-1}(\mathcal{M}(S_0)^{U(1)})$ $\cap \mathcal{M}(F)^{\ZZ/2}$.  
 Suppose that $\rho\in \mathcal{M}(F)_0$$\cap \mathcal{M}(F)^{\ZZ/2}$, $[{\bf t}]\in {\bf T}^n$ and $g\in SU(2)$ satisfy $\rho\cdot {\bf t}=g\rho g^{-1}$.  
If $g= \pm 1$, then
$$\rho(D_i)e^{t_iH(\rho(A_i))}=(\rho\cdot {\bf t})(D_i)=g\rho(D_i)g^{-1}=\rho(D_i)$$
 so that  each $t_i\equiv 0 \mod 2\pi $, and  hence $[{\bf t}]={\bf 1}\in {\bf T}^n$.   
 
 Assume, therefore, that $g
=e^{sP}$ with $s\in (0,\pi)$, with $P$ a unit purely imaginary quaternion. 
 Since $g \rho(A_i) g^{-1} = (\rho \cdot {\bf t}) (A_i) = A_i$,  we have $H(\rho(A_i))=\epsilon_i P$ for some 
$\ep_i\in\{\pm 1\}$.  In particular, the $\rho(A_i)$ all commute. 
In addition, for each $i$, 
 $\rho(D_i)e^{t_i\ep_iP}= e^{sP}\rho(D_i) e^{-sP}$ implies that
 $$e^{(t_i\ep_i+s)\rho(D_i)P\rho(D_i)^{-1}}=e^{sP}.$$  
Since $s\in(0,\pi)$, this is only possible (for each $i$) if either 
$$\rho(D_i)P\rho(D_i)^{-1}=P \text{ and } t_i\ep_i\equiv 0 \mod 2\pi, $$ 
 or else
  $$\rho(D_i)P\rho(D_i)^{-1}=-P \text{ and }  t_i\ep_i\equiv -2s \mod 2\pi.$$  
If the first case holds for all $i$, then $\rho$ is an abelian representation on $\pi_1(F)$,  which we have ruled out with the hypothesis that $\rho\in \mathcal{M}(F)^{\ZZ/2}$.
If for some $i$ the second case holds, then $\rho(D_i)$ must be a purely imaginary  unit quaternion orthogonal to $P$. 
This implies that $\rho(B_i)=\rho(D_i)\rho(A_i)\rho(D_i^{-1})= -\rho(A_i)$.

In either case we see that all  the $\rho(A_i)$ and $\rho(B_i)$ commute.  That is, we have shown that if there exists a $g\in SU(2)$ and ${\bf t}\in {\bf T}^n$ so that $\rho\cdot{\bf t}=g\rho g^{-1}$, then either $g=\pm 1$ and $[{\bf t}]={\bf 1} \in {\bf T}^n$ or $r(\rho)\in \mathcal{M}(S_0)^{U(1)}$.  
Hence the ${\bf T}^n$ action is free on $\mathcal{M}(F)_0\setminus r^{-1}(\mathcal{M}(S_0)^{U(1)})=r^{-1}(\mathcal{M}(S_0)^{\ZZ/2})$.

To see that the stabilizer is 1-dimensional if $r(\rho)\in \mathcal{M}(S_0)^{U(1)}$, observe that for each index $i$ so that the second case $t_i\ep_i\equiv -2s$ mod $2\pi$ holds (and there is at least one such index   if the stabilizer is nontrivial in ${\bf T}^n$), $e^{t_i\bbi}=e^{-2\ep_i s}$.    The $\ep_i$ are determined up to an overall sign by $\rho$, and hence the stabilizer is the 1-dimensional subgroup of ${\bf T}^n$ consisting of those $n$-tuples
$(e^{t_1\bbi}, \dots, e^{t_n\bbi})$ so that 
$$e^{t_i\bbi}=\begin{cases} 1& \text{ if }\rho(D_i)\rho(A_i)\rho(D_i)^{-1}=\rho(A_i)
\\ e^{\ep_it\bbi} & \text{ if } \rho(D_i)\rho(A_i)\rho(D_i)^{-1}=\rho(A_i)^{-1} .\end{cases}$$
This proves (ii).  

\medskip

We turn now to the symplectic properties.  The function $h:SU(2)\to \RR$ given by $h(g)=-\sin^{-1}(\Real(g))$, or equivalently, by $h(e^{sQ})=s-\frac\pi 2$ for a purely imaginary unit quaternion  $Q$ and $s\in [0,\pi]$, satisfies
$$\langle H(g), v\rangle =\left. \frac{d}{dt} h(ge^{tv}) \right|_{t=0} \text{ for all } g\in SU(2)\setminus\{\pm 1\} \text{ and } v\in su(2)$$
where $H$ is the function defined in Equation (\ref{Hmap}). Thus the functions $H$ and $h$ satisfy the relationship described in Section 1 of Goldman's article \cite{goldman2}  (see  Section \ref{sec2}  for more details).  For each $i=1,\dots, n$, define the function $h_{A_i}:\mathcal{M}(F)^{\ZZ/2}_0\to \RR$ 
 by $$h_{A_i}(\rho)=h(\rho(A_i)).$$ Since the $A_i$ are disjoint, \cite[Corollary 3.6]{goldman2} shows that the $h_{A_i}$ Poisson-commute.

Then \cite[Theorem 4.7]{goldman2}  shows that the Hamiltonian flow induced on $\mathcal{M}(F)^{\ZZ/2}_0$ by $h_{A_i}$ is given by 
$$t\cdot \rho (E)=\begin{cases} \rho(D_i)e^{t H(\rho(A_i))} & \text{ if } E=D_i\\ \rho(E)& \text{ if } E=A_i, \text{ or } E=A_j, B_j, j\ne i.\end{cases}$$
 This flow is $2\pi$-periodic and the corresponding  $S^1$ action is precisely that one obtained by restricting the ${\bf T}^n$ action to the $i$th factor.   Since the $h_{A_i}$ Poisson-commute, the entire ${\bf T}^n$ action is Hamiltonian. Moreover, since $\mu=(h_{A_1},\dots, h_{A_n})$, $\mu:\mathcal{M}(F)^{\ZZ/2}_0\to \RR^n$ is a moment map for the ${\bf T}^n$ action. This proves (iii).
 
 \medskip
 
To verify claim (iv), we must  check that $0$ is a regular value for the restriction of $\mu$ to $\mathcal{M}(F)^{\ZZ/2}_0\cap  r^{-1}(\mathcal{M}(S_0)^{\ZZ/2})$. This is more or less well known, but we provide an  argument here for completeness.  
 
 Pick $\rho\in \mathcal{M}(F)^{\ZZ/2}_0\cap  r^{-1}(\mathcal{M}(S_0)^{\ZZ/2})\cap \mu^{-1}(0)$.
Then for each $i$, $H^1(A_i;su(2)_{\ad{\rho}})\cong \RR$ since $\mu(\rho)=0$. The differential of $\mu$ at 
 $\rho$, $d\mu_\rho:T_\rho( \mathcal{M}(F)^{\ZZ/2})\to \RR^n$ can be identified with 
the map
$$\RR^{6n-6}\cong H^1(F;su(2)_{\ad{\rho}})\to H^1(\sqcup_i A_i;su(2)_{\ad{\rho}})\cong\RR^n,$$
The long exact sequence of the pair $(F, \sqcup_iA_i)$ identifies the cokernel with a subspace of $H^2(F,\sqcup_i A_i;su(2)_{a\ad{\rho}})$, which, by replacing $\sqcup_i A_i$ by a small neighborhood and applying excision, is isomorphic to $H^2(S_0,\partial S_0;su(2)_{\ad{r(\rho)}})$. Poincar\'e duality then identifies this with $H_0(S_0;su(2)_{\ad{r(\rho)}}),$ which vanishes because $r(\rho)\in \mathcal{M}(S_0)^{\ZZ/2}$, i.e., because $r(\rho)$ is irreducible. Hence the differential is onto.

The proof is completed by recalling that the symplectic quotient \cite{MS} by the Hamiltonian free ${\bf T}^n$ action on $\mathcal{M}(F)^{\ZZ/2}_0\cap  r^{-1}(\mathcal{M}(S_0)^{\ZZ/2})$ 
is the manifold $(\mu')^{-1}(0)/{\bf T}^n$, which we have identified with $
R(S^2,\punctures)^{\ZZ/2}.$
\end{proof}
 
\begin{cor}\label{lagreg}
Suppose that $\ell:L\to \mathcal{M}(F)_0^{\ZZ/2}\cap r^{-1}(\mathcal{M}(S_0)^{\ZZ/2})$ is a Lagrangian immersion which is transverse to $(\mu')^{-1}(0)$.  Then its symplectic reduction $L':=(\mu\circ \ell)^{-1}(0)$  Lagrangian immerses  to $R(S^2,\punctures)$.
\end{cor}
 
\begin{proof}
 This is a basic property of symplectic reduction and moment maps.    If the Lagrangian immersion $\ell$  meets $(\mu')^{-1}(0)$ cleanly, then the restriction of $\ell$ to the preimage of $(\mu')^{-1}(0)$, composed with the quotient map, is Lagrangian (see, for example, \cite{MS}).  In our case, the stronger hypothesis that $\ell$ meets $(\mu')^{-1}(0)$ transversely implies that $\ell$ also meets the orbits in  $(\mu')^{-1}(0)$ transversely, and so we obtain a Lagrangian immersion $L \cap (\ell \circ \mu')^{-1}(0) \to  R(S^2,\punctures)$. \end{proof}

 \medskip

 It will be simpler in the following to work with the map
 \begin{equation} \label{Tmap}
T:\mathcal{M}(F)\to \RR^n, ~T(\rho)=(\Real(\rho(A_1)),\dots, \Real(\rho(A_n))).
\end{equation}
rather than the moment map $\mu$ of Definition \ref{mumap}.
Although $T$ is not a moment map for the ${\bf T}^n$ action, the level sets of $T$ and $\mu$ coincide.  Furthermore, since the function $\sin^{-1}(x)$ is a diffeomorphism near $0$,  the restriction\begin{equation}
\label{Tregular}
T:\mathcal{M}(F)^{\ZZ/2}_0\cap  r^{-1}(\mathcal{M}(S_0)^{\ZZ/2})\to \RR^n
\end{equation}
has $0$ as a regular value, and 
$$R(S^2,\punctures)=T^{-1}(0)/{\bf T}^n.$$

\section{perturbation curves}\label{snakesonaplane2}

Our desired transversality results will be established with the help of a carefully constructed collection of curves on the surface $F$.  This section will be devoted to specifying these curves and tabulating how they intersect with the standard fundamental group generators.  These intersections  will be important in the analysis of the effect of perturbing using these curves (pushed slightly into the 3-manifold from the boundary).
\begin{df}\label{longitude}
  Fix two embedded, oriented, unbased, transverse curves $C$ and $E$ in $F$ missing the base point, and equip $E$ with an embedded arc $\gamma_E$ starting at the base point and ending   on $E$.      Assume that either $C$ intersects $E$ transversely in a single point and misses the arc $\gamma_E$, or that $C$ misses $E$ but intersects the arc $\gamma_E$ transversely in a single point.

  Define {\em the longitude of $C$ with respect to $E$},  $\lambda_C(E)$,  as follows.
  \begin{enumerate} 
\item In the first case, $\lambda_C(E)$ travels from the base point along $\gamma_E$, then forward along $E$ to the intersection  with $C$, then around $C$ returning to the intersection point, then backward along the same portion of $E$, and finally back to the base point along $\gamma_E$. 
\item In the second case, $\lambda_C(E)$ travels  along $\gamma_E$  from the base point to the intersection  with $C$, then around $C$ returning to the intersection point, then backward  along $\gamma_E$ to the base point. 

\end{enumerate}
      
\end{df}

\begin{df} A {\em special perturbation curve} is an embedded,  unbased curve $C$   in $F$ satisfying the condition that, for each  $E\in\{A_i,B_i\}_{i=1}^n$,  if $C \cap \left( E\cup \gamma_E\right)$ is nonempty then either $C$ meets  $E$ transversely  one point and misses $\gamma_E$, or $C$ is disjoint from $E$ but  intersects $\gamma_E$ in transversely  one point.   \end{df}

 \bigskip
 
 We now tabulate a finite collection of special perturbation curves, together with the curves in the family $\{A_i, D_i\}_{i=1}^n$ which intersect them, in Table \ref{table1}.  The first column, labeled {\em Perturbation curve}, lists 11 families of special perturbation curves,  $C_{\RN 1}(i)$ through $C_{\RN{11}}(ij)$.  These are illustrated in 
 Figures \ref{drawing3afig}, \ref{drawing3a1fig}, \ref{drawing3a2fig},    \ref{drawing3a3fig}, and  \ref{drawing3a4fig}.  
    Recall that the sector indexing should be viewed as a cyclic ordering.  In Figures \ref{drawing3a1fig}, \ref{drawing3a2fig},  \ref{drawing3a3fig}, and  \ref{drawing3a4fig} illustrate the curves we have in mind if $1\leq i < j \leq n$, but the curves in Table \ref{table1} when  $j<i$ are intended to denote the analogous curves that cross  sectors  from $i$ to $j$ in the counterclockwise direction. 
    
 The special perturbation curves $C_{\RN 1}(i),C_{\RN 2}(i), $ and $C_{\RN 3}(i)$ lie in the interior of the $i$th sector.  In particular,  they miss $A_\ell, B_\ell$ and $D_\ell$ as well as $\gamma_{A_\ell}, \gamma_{B_\ell},$ and $\gamma_{D_\ell},$ for $\ell\ne i$. The curves 
 $C_{\RN 4}(ij),  C_{\RN 5}(ij),  C_{\RN 6}(ij),$ and $  C_{\RN 7}(ij)$ miss $A_\ell, B_\ell$, and $D_\ell$ when $\ell\ne i,j$, but they do meet 
 $\gamma_{A_\ell}, \gamma_{B_\ell},$ and $\gamma_{D_\ell},$ for $\ell\ne i,j$. The curves 
 $C_{\RN 8}(ij),  C_{\RN 9}(ij),  C_{\RN{10}}(ij),$ and $  C_{\RN{11}}(ij)$ miss $A_\ell, B_\ell$, and $D_\ell$  as well as
 $\gamma_{A_\ell}, \gamma_{B_\ell},$ and $\gamma_{D_\ell},$ for $\ell\ne i,j$.

     \begin{figure}[h]
\begin{center}
\def\svgwidth{5in}
 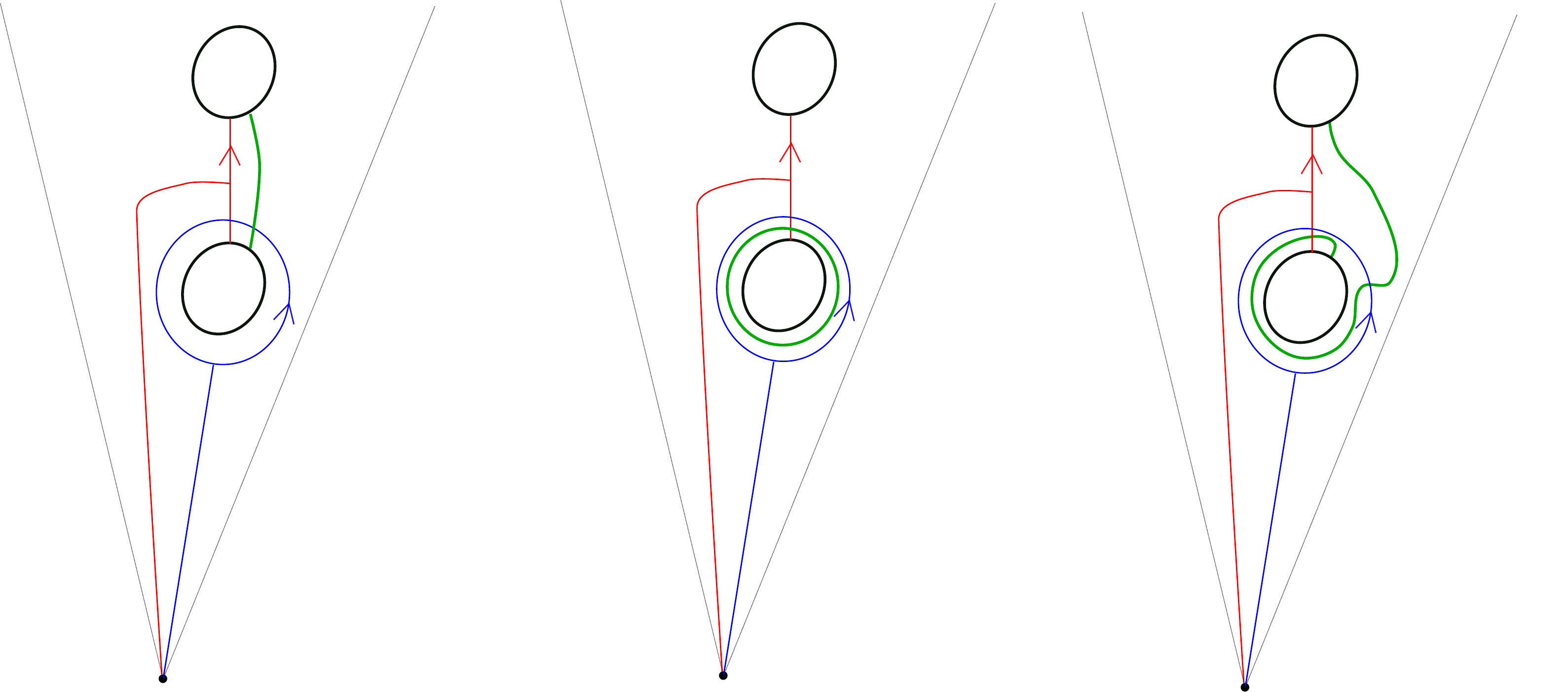
 \caption{The special perturbation curves $C_{\RN 1}(i), C_{\RN 2}(i),$ and $C_{\RN 3}(i)$ in the $i$th sector. \label{drawing3afig} }
\end{center}
\end{figure}

     \begin{figure}[h]
\begin{center}
\def\svgwidth{6in}
 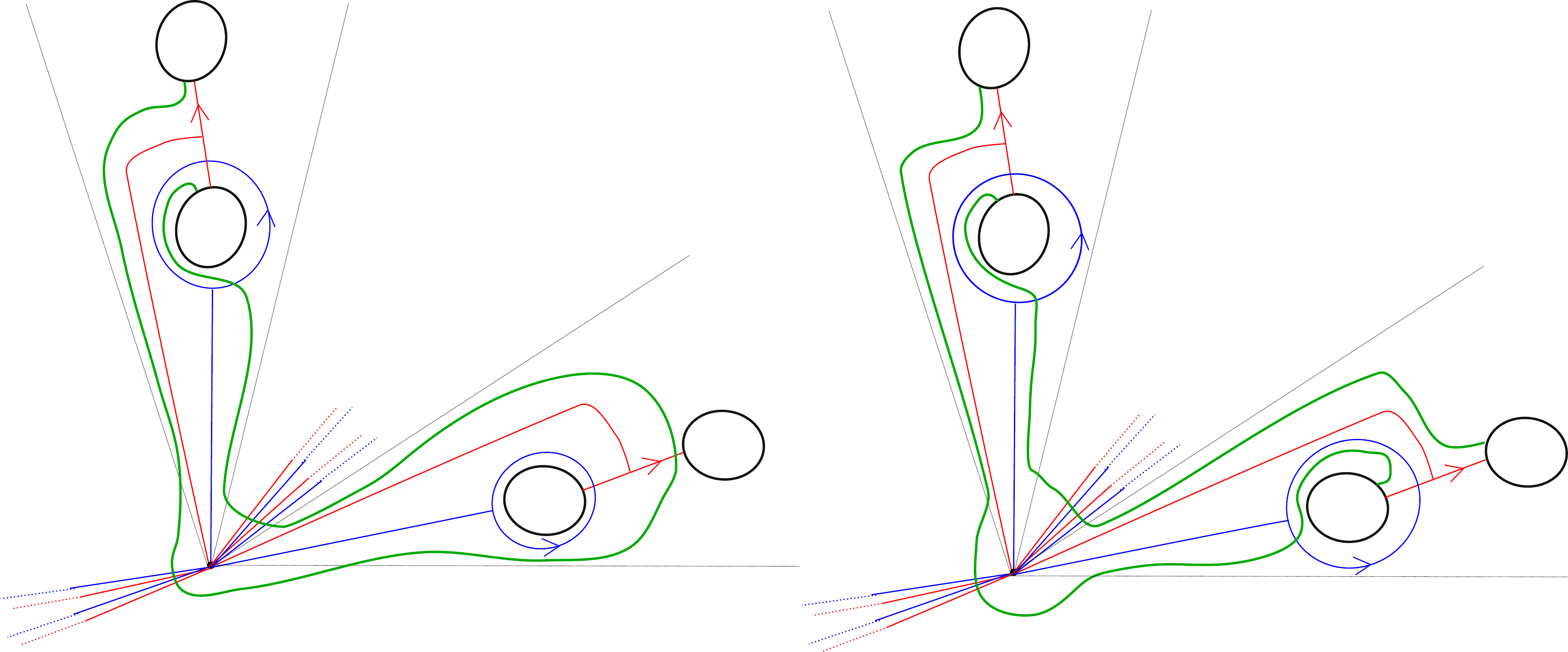
 \caption{The  special perturbation curves $C_{\RN 4}(ij)$ and $C_{\RN 5}(ij)$ in the $ith$ and $j$th sectors. \label{drawing3a1fig} }
\end{center}
\end{figure}

     \begin{figure}[h]
\begin{center}
\def\svgwidth{6in}
 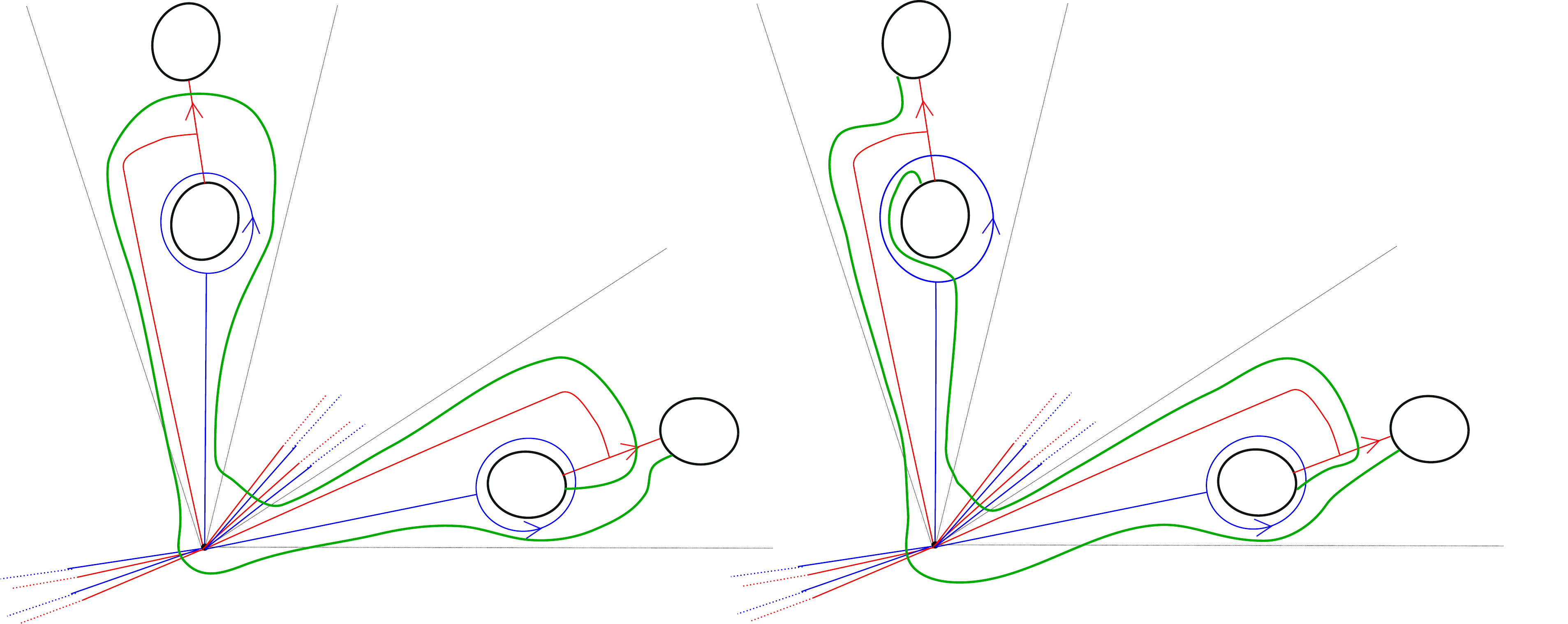
 \caption{The  special perturbation curves $C_{\RN 6}(ij)$ and $C_{\RN 7}(ij)$ in the $ith$ and $j$th sectors. \label{drawing3a2fig} }
\end{center}
\end{figure}

     \begin{figure}[h]
\begin{center}
\def\svgwidth{5.5in}
 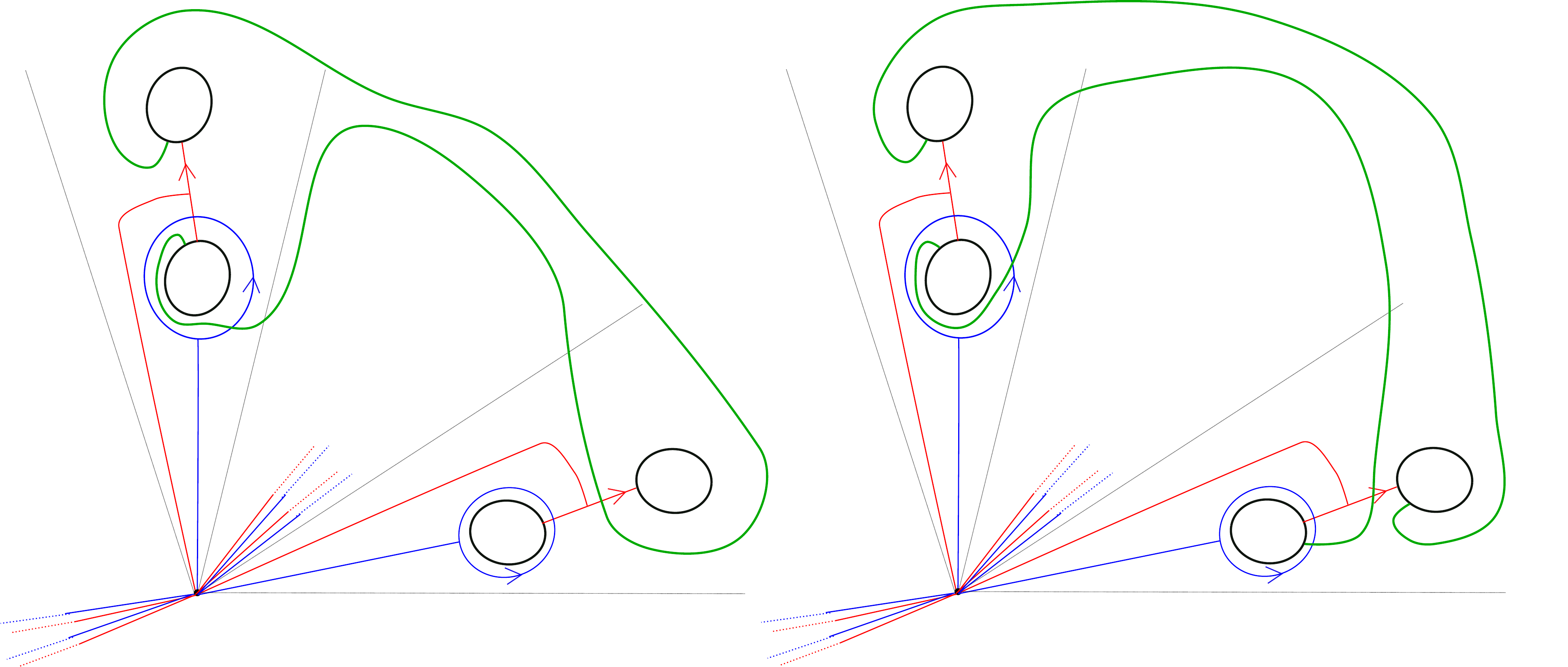
 \caption{The composite special perturbation curves $C_{\RN 8}(ij)$ and $C_{\RN 9}(ij)$ in the $ith$ and $j$th sectors. \label{drawing3a3fig} }
\end{center}
\end{figure} 

    \begin{figure}[h]
\begin{center}
\def\svgwidth{5.5in}
 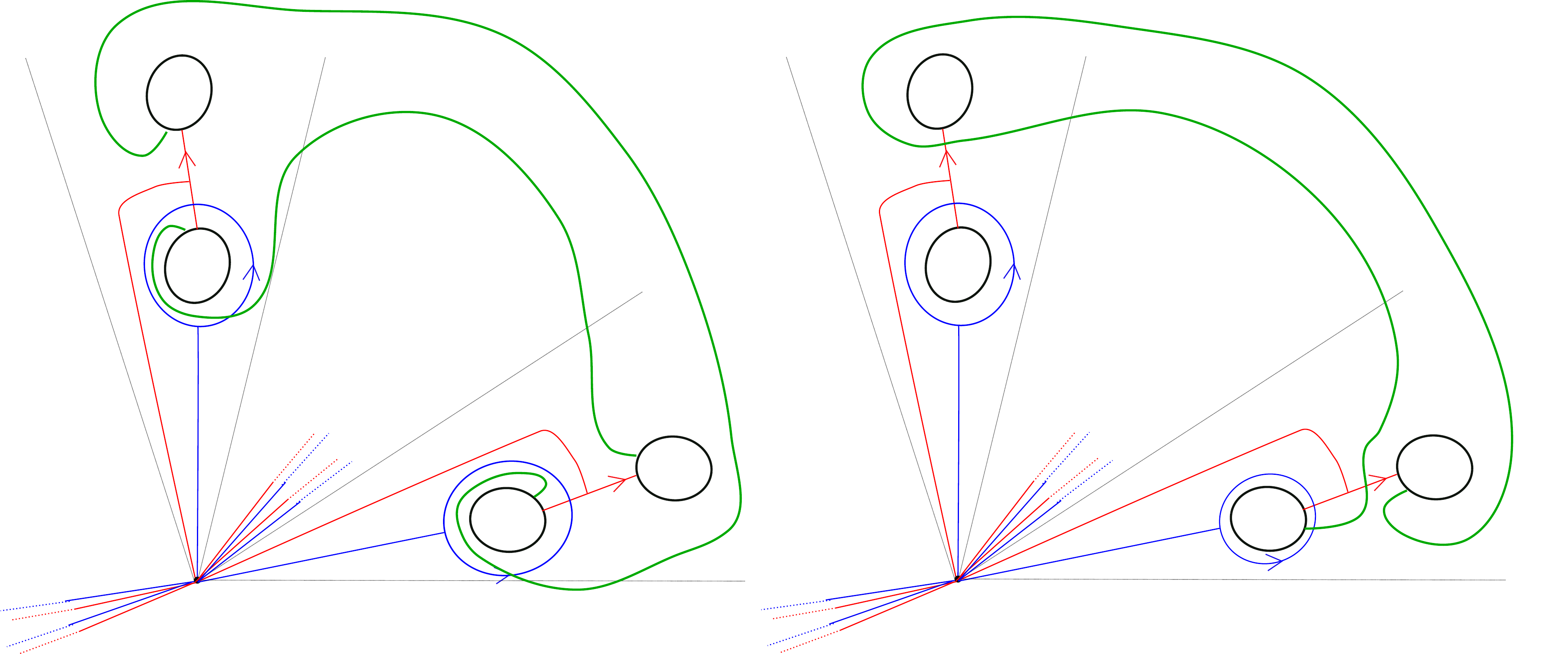
 \caption{The composite special perturbation curves $C_{\RN {10}}(ij)$ and $C_{\RN {11}}(ij)$ in the $ith$ and $j$th sectors.  \label{drawing3a4fig} }
\end{center}
\end{figure}

  For each special perturbation curve $C$ in the first column, the  {\em Intersecting curve} column lists all the  embedded  curves  $E$ in the set $\{A_i,D_i\}_{i=1}^n$ which intersect $C$.  In each case, $E$ meets $C$ transversely in one point.  
 Notice that we do not list any $E$ as an intersection curve if the perturbation curve $C$ only intersects the path $\gamma_E$.  The third column, labeled {\em Longitude},  expresses the longitude $\lambda_C(E)$ of $C$   with respect to $E$, as an   element of  $\pi_1(F)$.     The last column records the sign of the intersection  $E\cdot C$.

\begin{table}
\begin{center}
  \begin{tabular}{|c|c|c|r|}
  \hline
 Perturbation curve $C$ & Intersecting curve $E$  & Longitude  $\lambda_C(E)$ &sign\\
   \hline \hline
   $C_{\RN 1}(i)$& $A_i$&$A_iD_iA_i^{-1}$& $-1$\\ \hline
    $C_{\RN 2}(i)$& $D_i$&$ D_iA_iD_i^{-1}$& $1$\\ \hline
 $ C_{\RN 3}(i)$ & $D_i$ & $ D_iA_i$ &1 \\ \cline{2-4}
   & $A_i$ &$ A_iD_i$ & $-1$\\ \cline{2-4}
  \hline
$C_{\RN 4}(ij)$, $i\neq j$ & $D_i$ & $D_j^{-1}A_i$	&1 \\ \cline{2-4}
& $A_j$ &$D_j^{-1} A_i$ & $1$\\ \cline{2-4}\hline
$ C_{\RN 5}(ij)$, $i\neq j$ & $A_i$ & $D_jD_i$&$-1 $\\ \cline{2-4}
& $A_j$ &$D_iD_j$ & $-1$\\ \cline{2-4}\hline
$C_{\RN 6}(ij)$, $i\neq j$ & $A_i $ & $ A_iA_jA_iD_jA_i^{-1}$& $-1$ \\ \cline{2-4}
& $D_i $ &$ A_j A_iD_i$& $ 1$\\ \cline{2-4} 
& $D_j $ &$A_iD_iA_j$& $ 1$\\ \cline{2-4}\hline
$C_{\RN 7}(ij)$, $i\neq j$ & $ D_i $ &$D_j^{-1}A_i  D_i$& 1\\ \cline{2-4} 
& $ A_i$ & $A_iD_j^{-1}A_iD_iA_i^{-1} $& $-1 $ \\ \cline{2-4}
& $A_j $ &$D_j^{-1}A_iD_i$& $1  $\\ \cline{2-4}\hline

$C_{\RN 8}(ij)$, $i\neq j$ & $ D_i $ &$ ( \prod_{\ell=i+1}^{j-1}A_\ell B_\ell^{-1}) D_j^{-1}  (\prod_{\ell=i}^{j}A_\ell B_\ell^{-1})^{-1}A_i$& $1$  \\ \cline{2-4} 
& $A_j $ &$D_j^{-1}(\prod_{\ell=i}^{j}A_\ell B_\ell^{-1})^{-1} A_i (\prod_{\ell=i+1}^{j-1} A_\ell B_\ell^{-1})$& $ 1  $\\ \cline{2-4}\hline

$C_{\RN 9}(ij)$, $i\neq j$ & $ D_i $ &$ (\prod_{\ell=i+1}^{j-1}A_\ell B_\ell^{-1})D_j^{-1}(\prod_{\ell=i}^{j}A_\ell B_\ell^{-1})^{-1}A_iD_i$& $1$ \\ \cline{2-4} 
& $ A_j$ & $ D_j^{-1}(\prod_{\ell=i}^{j}A_\ell B_\ell^{-1})^{-1}A_iD_i(\prod_{\ell=i+1}^{j-1}A_\ell B_\ell^{-1})$& $ 1 $ \\ \cline{2-4}
& $A_i $ &$A_i(\prod_{\ell=i+1}^{j-1}A_\ell B_\ell^{-1})D_j^{-1}(\prod_{\ell=i}^{j}A_\ell B_\ell^{-1})^{-1}A_iD_iA_i^{-1} $& $  -1 $\\ \cline{2-4}\hline

$C_{\RN {10}}(ij)$, $i\neq j$ & $ A_i $ &$ (\prod_{\ell=i}^{j}A_\ell B_\ell^{-1})D_j (\prod_{\ell=i+1}^{j-1}A_\ell B_\ell^{-1})^{-1} D_i $&$-1 $\\ \cline{2-4} 
& $A_j $ &$ (\prod_{\ell=i+1}^{j-1}A_\ell B_\ell^{-1})^{-1}D_i (\prod_{\ell=i}^{j}A_\ell B_\ell^{-1}) D_j $& $-1   $\\ \cline{2-4}\hline

$C_{\RN {11}}(ij)$, $i\neq j$ & $ D_i $ &$  (\prod_{\ell=i+1}^{j-1}A_\ell B_\ell^{-1})A_j(\prod_{\ell=i}^{j}A_\ell B_\ell^{-1})^{-1}A_iD_i$& $-1$ \\ \cline{2-4} 
& $ A_i$ &$A_i(\prod_{\ell=i+1}^{j-1}A_\ell B_\ell^{-1})A_j(\prod_{\ell=i}^{j}A_\ell B_\ell^{-1})^{-1}A_iD_iA_i^{-1}$& $-1$\\ \cline{2-4}
& $D_j $ &$(\prod_{\ell=i}^{j}A_\ell B_\ell^{-1})^{-1}A_iD_i (\prod_{\ell=i+1}^{j-1}A_\ell B_\ell^{-1})A_j$& $1   $\\ \cline{2-4}\hline

\end{tabular}
\medskip
 \caption {\label{table1} Perturbation curves illustrated in   Figures \ref{drawing3afig}, \ref{drawing3a1fig}, \ref{drawing3a2fig},   \ref{drawing3a3fig}, and \ref{drawing3a4fig} with their longitudes for each intersecting curve $E\in \{ A_i, D_i\}_{i=1}^n$.  }
\end{center}
\end{table}
\medskip

 We leave it as a straightforward exercise
to verify most of the formulas in the third column of Table \ref{table1} for the longitudes with respect to the intersecting curves, but we illustrate the case of  the perturbation curve $C_{\RN 4}(12)$ and intersection curve $A_2$, in 
Figure \ref{drawing2fig}.   The longitude, also illustrated there, is  easily seen to represent the word $ D_2^{-1}A_1$.  The sign is given by $A_2\cdot C_{\RN 4}(12)=1$.

     \begin{figure}[h]
\begin{center}
\def\svgwidth{4.5in}
 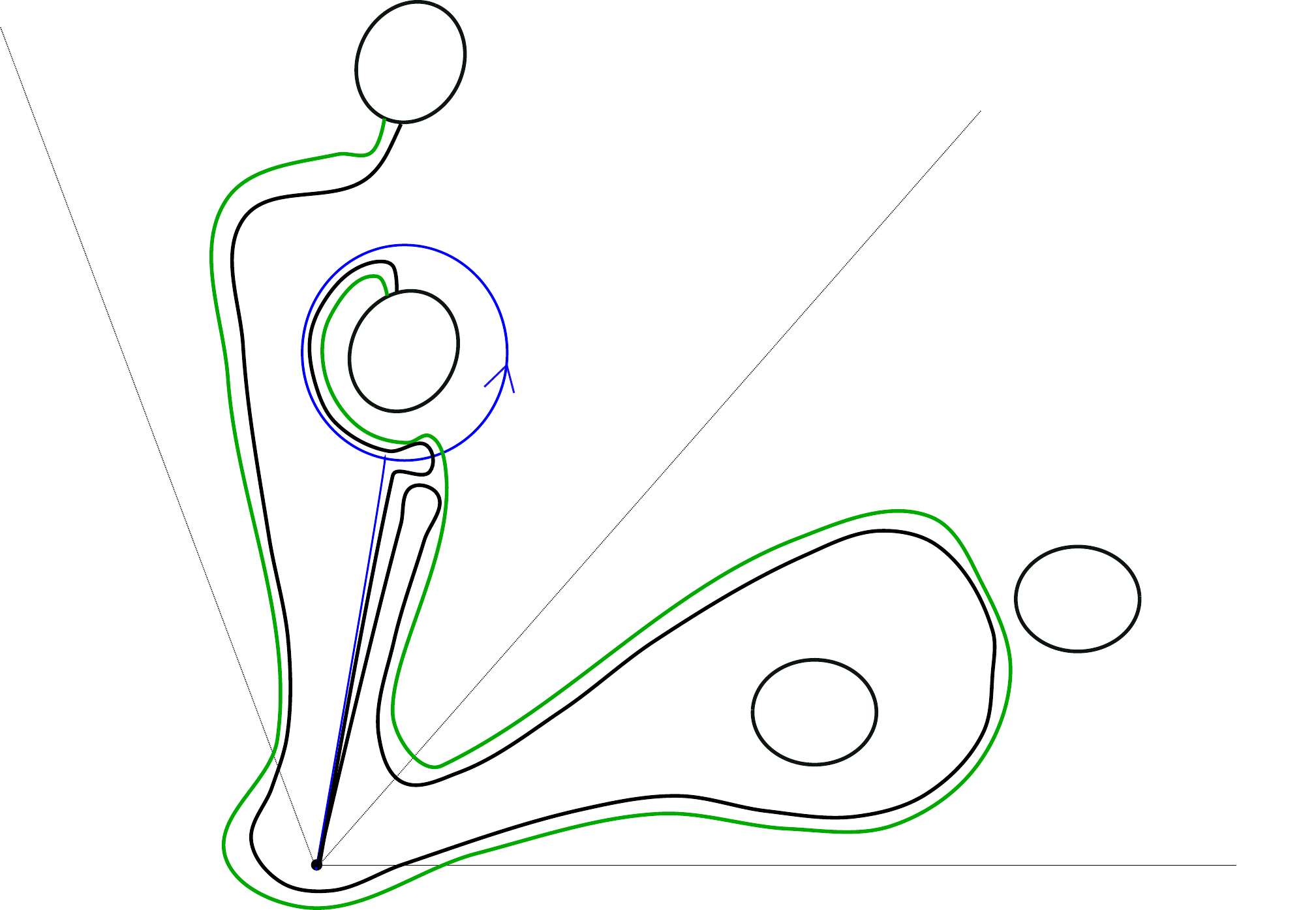
 \caption{The longitude for the perturbation curve $C_{\RN 4}(1,2)$ and intersecting curve $A_2$ equals $D_2^{-1}A_1$. \label{drawing2fig} }
\end{center}
\end{figure}

\section{Perturbations in a cylinder   $F\times I$}\label{sec2}

The transversality arguments in this article consist of two types.  In part, we make use of general results about the generic structure of the perturbed moduli space, proven in \cite{herald1}.  But we also need additional results  concerning how the traceless conditions and earring anticommutativity condition  in this paper cut this moduli space down, and these are not addressed in \cite{herald1}, so we prove  these additional results here.  For  the latter   arguments, we show that it is sufficient to use perturbation curves in a collar neighborhood of $\partial X$.  To set up these arguments, we present in this section some basic results about the effect of perturbing in a cylinder $F\times I$ where $I=[0,1]$. We begin by examining the effect of one such perturbation.

Suppose that $C\subset F$ is an  embedded  oriented curve.    Let $N_C$ denote a  tubular neighborhood of $C\times\{\tfrac 1 2\}$ in the cylinder $F\times I$, framed so that its longitude $\lambda$ is represented by the push off $C\times\{ \tfrac 1 2 + \epsilon \}$.   Fix $\phi\in \mathcal{X}$ (see Equation (\ref{pertfun})) and consider the perturbation data $\pi_C=(N_C,  \phi)$.  

\begin{prop}  \label{prop1}  \label{stratpreserving}   With $\pi_C$ as above, the restriction map $$\mathcal{M}_{\pi_C}(F\times I)\to \mathcal{M}(F\times\{1\})
 $$
is a homeomorphism, preserving the orbit type stratification, and the same is true for restriction to the other end.  This is a diffeomorphism on each stratum.  \end{prop} 
\begin{proof}  
  Since the statement of the proposition involves representations up to conjugation, its veracity is independent of where we place the basepoint.  For convenience, choose  a basepoint  $x$  in $F$ that is not on $C$.   We fix the base point $(x,0)\in F\times I$ in the cylinder.  Let $\{A_i, D_i\}_{i=1} ^n$ denote the usual set of generators for $\pi_1(F)$.  
 We consider two cases, when $C$ is non-separating and when $C$ is separating in $F$. 

Consider first the case when $C$ is non-separating.  Since homeomorphisms of $F$ induce homeomorphisms of $\mathcal{M}(F)$ which are diffeomorphisms on each stratum,  it is sufficient to consider the case when $C$ is 
 the special perturbation curve $C_{\RN 1}(1)$ of Figure \ref{drawing3afig}.  View  the longitude   $\lambda$  of $C$ as a based loop by connecting it to the base point so that $\lambda$ and $D_1$ are homotopic relative to the base point.  

The Seifert-Van Kampen theorem shows that 
$$\pi_1(F\times I\setminus N_C)=\langle A_i, D_i, m~|~ \textstyle \prod\limits_{i=1}^n[A_i,D_i]=1\rangle=\pi_1(F)*\ZZ\langle m\rangle,$$  where $m$ is a meridian for $C$ connected to the base point the same way as $\lambda$. Thus any $\rho\in \mathcal{M}(F)$ may be extended to $\pi_1(F\times I\setminus N_C)$ by sending $m$ to any element in $SU(2)$.  For such an extension $\tilde \rho$ to satisfy the perturbation condition (\ref{pert}), $\tilde \rho(m)=1$ if $\rho(\lambda)=\pm 1$, and otherwise $\tilde \rho(m)=F(\rho(\lambda))$, where $F(e^{\alpha Q})=e^{\phi(\alpha)Q}$  when $\|Q\|=1$.    Hence the   extension of $\rho$ to $\tilde \rho$ in $\mathcal{M}_{\pi_C}(F\times I)$ is unique.

It is clear that the stabilizers of $\rho$ and $\tilde \rho$  coincide,   since $\tilde \rho(m)$ commutes with the element  $\rho(D_1)$ in the image of the first factor of the free product,  and since $\tilde\rho (m)=1$ if $\rho(D_1)$ is central.    This extension map $\rho\in \mathcal{M} (F) \mapsto \mathcal{M_\pi} (F\times  [0,1]) $ is an inverse for the restriction map sending $\tilde \rho$ to its restriction to $\pi_1(F),$ that is, its restriction to the first factor in the free product.  

 The longitude of $N_C$ can be  expressed as the word $\lambda(A_1,B_1,\dots , A_n, B_n)$. The map $SU(2)^{2n}\to SU(2)^{2n+1}$  given by $$(a_1,b_1,\dots,a_n,b_n)\mapsto (a_1,b_1,\dots,a_n,b_n,F(\lambda(a_1,b_1,\dots,a_n,b_n))$$   is smooth and equivariant with respect to conjugation.  The map $SU(2)^{2n+1}\to SU(2)^{2n}$ which projects onto the first $2n$ factors is also smooth and equivariant. These two maps induce bijections on their subquotients $\mathcal{M}_{\pi_C}(F\times I)\cong \mathcal{M}(F)$ by the previous paragraphs,  and hence they induce inverse homeomorphisms.  These are smooth diffeomorphisms on each stratum, since the orbit type stratification of $\mathcal{M}(F)$ coincides with its stratification as an algebraic variety. 

\medskip

Now consider the case when $C$ is separating. Assume the path component of $F\setminus C$ containing the base point has genus $g$. Up to homeomorphism of $F$, we may assume that $C$ is the curve depicted in Figure \ref{separatingfig}. 
    \begin{figure}[h]
\begin{center}
\def\svgwidth{4.9in}
 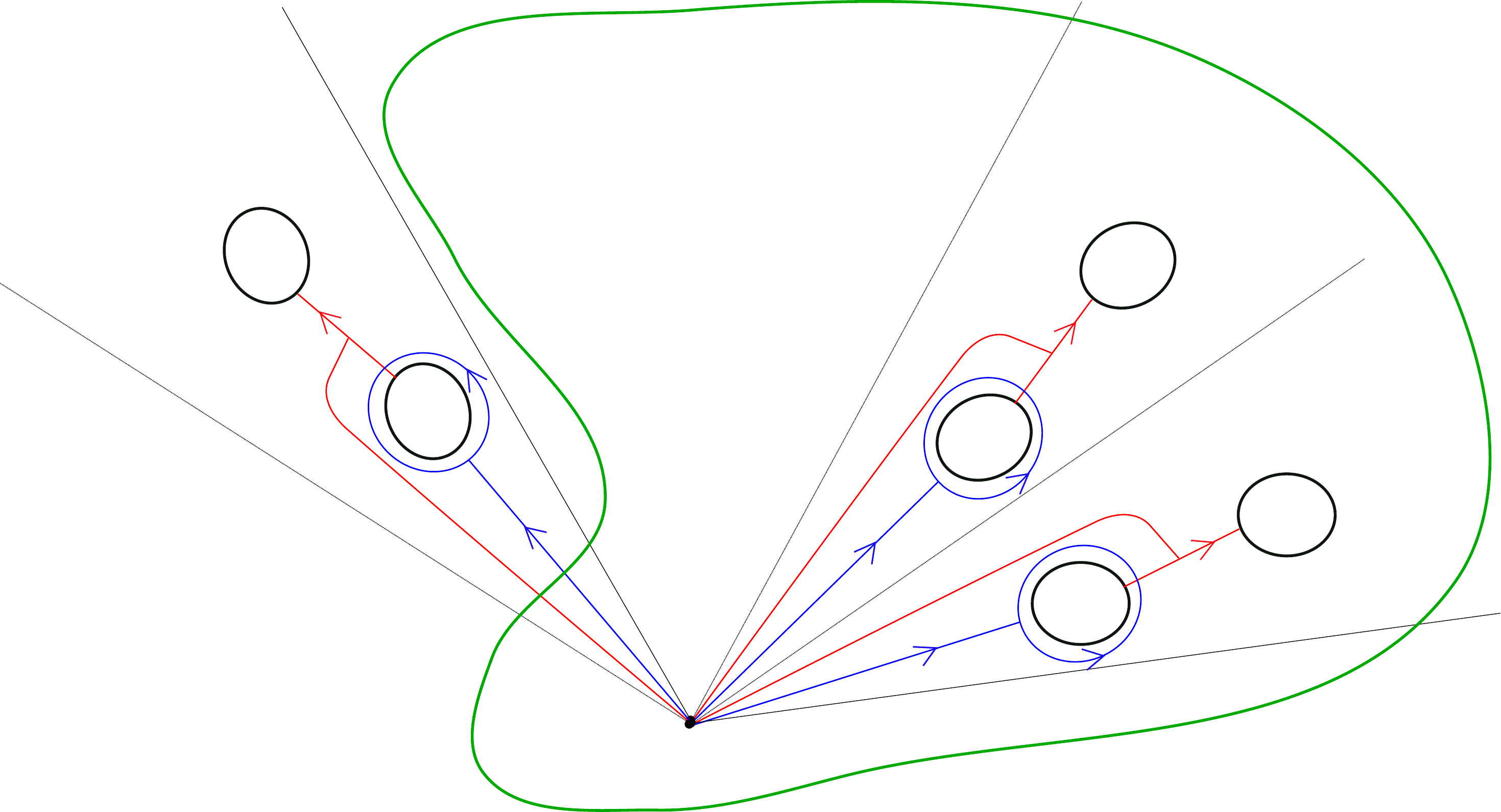
 \caption{ \label{separatingfig} }
\end{center}
\end{figure} 
Thus $C$ misses $A_1,D_1, \dots, A_g,D_g$ and their paths from  the base point, but $C$ intersects the paths from the base point to $A_{g+1},D_{g+1}, \dots, A_n,D_n$.

The Seifert-Van Kampen theorem in this case shows that  
$\pi_1(F\times [0,1]\setminus N_C)$ has a presentation with generators
$$A_i=A_i\times \{0\}, D_i=D_i\times \{0\},A_i'=A_i\times \{1\}, D_i'=D_i\times \{1\}, m$$
where $m$ is the meridian of $N_C$, subject to the relations
 \begin{equation}\label{seprel}
\textstyle \prod\limits_{i=1}^n[A_i,D_i]=1= \textstyle \prod\limits_{i=1}^n[A_i',D_i' ] , ~A_i'=\begin{cases} A_i&\text{ if } i\leq g,\\ mA_im^{-1}&\text{ if } i> g,\end{cases}, ~D_i'=\begin{cases} D_i&\text{ if } i\leq g,\\ mD_im^{-1}&\text{ if } i> g.\end{cases}
\end{equation}

The longitude $\lambda$ represents $\textstyle \prod\limits_{i=1}^g[A_i,D_i]$.  Again, one sees that given $\rho:\pi_1(F)\to SU(2)$, there exists a unique extension of $\rho$ to $\tilde\rho:\pi_1(F\times [0,1]\setminus N_C)\to SU(2)$ satisfying the perturbation condition. In fact,  $\rho$ determines $\tilde\rho(\lambda)$ by  $\tilde\rho(\lambda)=\rho(\textstyle \prod\limits_{i=1}^g[A_i,D_i])$, which in turn determines $\tilde\rho(\mu)$ by the perturbation condition (\ref{pert}).  Then $\rho$ and $\tilde\rho(m)$ determine $\tilde\rho(A_i')$ and $\tilde\rho(D_i')$ by the relations (\ref{seprel}), and the relation $ \tilde\rho(\textstyle \prod\limits_{i=1}^n[A_i',D_i' ])=1$ is a consequence of the fact that $\rho(\textstyle \prod\limits_{i=1}^n[A_i,D_i ])=1$ and $\tilde\rho([m,\lambda])=1$. The rest of the argument is similar to the first case, and we leave the details to the reader.
 \end{proof}

 Denote  by
 $\Phi_{\pi_C}:\mathcal{M}(F)\to \mathcal{M}(F) $ the composite homeomorphism $$\mathcal{M}(F)=\mathcal{M}(F\times\{0\})\leftarrow\mathcal{M}_{\pi_C}(F\times I)\to \mathcal{M}(F\times\{1\})=\mathcal{M}(F).
 $$ 
 Recall that $\lambda_C(E)$ denotes the longitude of $C$ with respect to $E$, from Definition \ref{longitude}, and $\pi_C$ denotes the perturbation data $\pi_C=\{N_C,\phi\}$.

\begin{prop} \label{isoform}
 Let $C$ be an embedded curve in $F$ and $\phi\in \mathcal{X}$, determining perturbation data $\pi_C=\{N_C,\phi\}$. Let $E$ be a  loop in $F$, for example $E\in \{A_i, D_i\}_{i=1}^n$, $\gamma_E$ a path from $E$ to the base point,    and let $\rho\in \mathcal{M}(F)$.    
 
 If $C$ is disjoint from  $E \cup \gamma_E$, then 
 $$\Phi_{\pi_C}(\rho)(E)=\rho(E).$$

If $C$ meets $E$ transversely in a single  point    with oriented   intersection number $E\cdot C=\pm 1$  and $C$ misses its arc $\gamma_E$ from the base point, 
and  if $\rho(\lambda_C(E))=e^{ \alpha  Q}$,  then 
 $$\Phi_{\pi_C}(\rho)(E)=e^{\pm \phi(\alpha) Q}\rho(E).$$
 
If $C$ misses $E$  but intersects  the
  arc $\gamma_E$ 
  transversely in a single point with oriented intersection number $C\cdot \gamma_E = \pm 1$,  and if $\rho(\lambda_C(E))=e^{\alpha Q}$, 
  then 
  $$\Phi_{\pi_C}(\rho)(E)=e^{\pm \phi(\alpha)Q}\rho(E)e^{\mp \phi(\alpha)Q}.$$
\end{prop}

\begin{proof} 
Let $e:I\to F$ be a parameterization of the based loop $\gamma_E*E*\gamma_E^{-1}$. Extend $e$ to a level preserving map of a square 
$$h_E:I\times I \to F\times I, ~(t,s)\mapsto  (e(t),s).$$
Identify $e$ with $e\times \{ 0\}$.   Recall that $x$ denotes the base point of $F$.   Denote by $e'$ the based loop in $F\times I\setminus N_C$ obtained by traveling along $\{x \}\times I$, then following $e\times \{1\}$, then returning to the base point  along $\{x\}\times I$.   The map $\Phi_{\pi_C}:\mathcal{M}(F)\to \mathcal{M}(F)$ takes a representation $\rho:\pi_1(F\times \{0\}) \to SU(2)$ to the restriction to $F\times \{1\}$ of the unique extension $\tilde \rho$ of $\rho$ to $\mathcal{M}_{\pi_C}(F\times I)$. In other words, $$ \Phi_{\pi_C}(\rho(e))= \tilde \rho(e').$$

 If $C$ misses $E$ and $\gamma_E$, then 
$h_E$ gives a based homotopy in $F\times I \setminus N_C $ between $e$ and $e'$, and hence  $\tilde \rho(e')=\tilde \rho(e)$, proving the first assertion.

If $C$ meets $E$ transversely in one point $y$ and $C$ misses $\gamma_E$, then $C\times \{ \frac 1 2 \}$ punctures the square $h_E$ once in  the interior of its embedded middle section, 
and  
\begin{equation}
\label{intersect}\tilde \rho(e')=\tilde \rho(m)\tilde \rho(e),
\end{equation}
where $m$ is a loop which travels  in $F\times \{0\}$ along $\gamma_E$ to $E$, then along $E$   in $F\times \{0\}$ to $(y,0)$, then along $\{y\}\times I$ toward the puncture where $h_E$ meets $C\times \{ \frac 1 2 \}$,  then around the puncture with  the positive orientation (on the square), and then backwards along $\{y\}\times I$ and backwards  in $F\times \{0\}$  along $E$ to the base point.  

Since $\lambda_C(E)$ was defined in this case to be a longitudinal curve connected  to the base point in the same way as the meridian $m$, with the same orientation as $C$,  $m^{E\cdot C}$ is a meridian to the perturbation curve with $m^{E\cdot C} \cdot \lambda_C(E)=1$ in $\partial N_C$.   In particular,  since $\tilde \rho$ satisfies the perturbation condition (\ref{pert}),  
 Equation (\ref{intersect}) implies that 
\begin{equation*}
 \tilde \rho(e')=e^{\pm \phi(\alpha)Q}\tilde \rho(e), \text{ where } \tilde \rho(\lambda_C(E))=e^{\alpha Q},  
\end{equation*}
 with the sign in the exponent equal to   $E\cdot C$. This proves the second assertion.

  A similar argument applies when $C$ misses $E$ but intersects $\gamma_E$ in a point $y$.   In this case,  $h_e$ maps two outer strips of the square onto the same embedded rectangle $\gamma_E \times I$,  but with opposite orientations, and $C\times \{ \frac 1 2 \}$ punctures this embedded rectangle.  Define  $m$ by traveling along $\gamma_E$ to $(y,0)$, down $\{y\}\times I $ and around the puncture, and back up to $(y,0)$ and then back along $\gamma_E$ to the base point. 
Then $e'$ is homotopic rel base point to $mem^{-1}$, and the perturbation condition gives $$ \tilde \rho(e')=e^{\pm \phi(\alpha)Q}\tilde \rho(e)e^{\mp \phi(\alpha)Q}, $$  where $\pm 1=  \gamma_E \cdot C$.  We leave the details to the reader.\end{proof}

\medskip

We next prove that the maps $\Phi_{\pi_C}$ are symplectomorphisms, and in fact that the 1-parameter family of perturbations $\pi_{C,t}=\{N_C,t\phi\}, t\in [0,1]$  yields  a Hamiltonian isotopy $\Phi_{\pi_{C,t}}$ starting at the identity. We make use of Goldman's result \cite{goldman1}, which we recall briefly next.

Let  
 $f:SU(2)\to \RR$ and   $F:SU(2)\to su(2)$ be functions that are conjugation invariant and equivariant, respectively,  and which satisfy
 Goldman's  relation  in \cite{goldman1} with respect to the inner product $\langle u, v\rangle=-\Real(uv)$:
\begin{equation}
\label{goldmanrel}\langle F(g), v\rangle= \tfrac{d}{dt}|_{t=0} f(ge^{tv}) \text{ for all } g\in SU(2), v\in su(2).
\end{equation}
 
  Let $c\in \pi_1(F)$ be the loop obtained by connecting the embedded curve $C$ to the base point by some path.
 Then $c$  defines  maps
 $$f_c:\mathcal{M}(F)\to \RR,~ f_c([\rho])= f(\rho(c))$$
 and
  $$F_c:\Hom(\pi_1(F),SU(2))\to su(2),~ F_c(\rho)= F(\rho(c)).$$
 Note that the function $f_c$ does not depend on the choice of path from the base point to $C$, although $F_c$ does. 
The map $f_c$ defines a Hamiltonian flow $\Xi_t$   on the symplectic manifold $\mathcal{M}(F)^{\ZZ/2}$.

Goldman's result,  \cite[Theorems 4.3, 4.5, 4.7]{goldman1}, calculates the flow $\Xi_t$ explicitly,  assuming the curve $C$ passes through the base point. He defines the  {\em twist flow on  $\Hom(\pi_1(F),SU(2))$ associated to $C$} as follows.  If $C$ is non-separating, then $\pi_1(F)$ is generated by loops which miss $C$, together with one additional loop which intersects $C$ transversally exactly once. Goldman defines a  flow on $\Hom(\pi_1(F),SU(2))$ by
\begin{equation}\label{nonsepgold}\Xi_t(\rho)(\gamma)=\begin{cases} \rho(\gamma)& \text{ if $\gamma$ is a loop missing $C$,} \\ \rho(\gamma)e^{tF_c(\rho)}& \text{ if $\gamma$ intersects  $C$ exactly once transversally with $\gamma\cdot c=1$.}
\end{cases}
\end{equation}

If $C$  separates $F$ into path components $F_1, F_2$, then $\pi_1(F)$ is generated by loops which lie entirely in $F_1$ or entirely in $F_2$. In this case Goldman defines a  flow on $\Hom(\pi_1(F),SU(2))$ by
\begin{equation}\label{sepgold}\Xi_t(\rho)(\gamma)=\begin{cases} \rho(\gamma)& \text{ if $\gamma$ is a loop in $F_1$,} \\ e^{tF_c(\rho)} \rho(\gamma) e^{-tF_c(\rho)}& \text{ if $\gamma$ is a loop in $F_2$.}
\end{cases}
\end{equation}
Goldman's theorem asserts that (\ref{nonsepgold}) and (\ref{sepgold}) uniquely determine  flows on $\Hom (\pi_1(F),SU(2))$, and these  push down to yield the Hamiltonian flow on $\mathcal{M}(F)^{\ZZ/2}$ determined by $f_c$.

\begin{thm}\label{hamiltonian} Let $\pi_{C,t}$ be the 1-parameter family of perturbations $\pi_{C,t}=(N_C, t \phi)$ where $C\subset F$ is an embedded curve  and $\phi\in \mathcal{X}$.  

Then the isotopy $\Phi_{\pi_{C,t}}:\mathcal{M}(F)\to \mathcal{M}(F)$ restricts to a Hamiltonian isotopy on the smooth  stratum  $\mathcal{M}(F)^{\ZZ/2}$. In fact, $\Phi_{\pi_{C,t}}$ is the twist flow associated to $C$,  generated by the function 
$$f_c:\mathcal{M}(F)\to \RR, ~f_c([\rho])=\psi(\cos^{-1}\Real(\rho(C)))$$ for any antiderivative  $\psi$ of $\phi$.

\end{thm}

\begin{proof} Given $\phi\in \mathcal{X}$, choose an antiderivative  $\psi:\RR\to \RR$ for $\phi$. Then  $\psi$ is even and periodic. The pair of functions 
$$f:SU(2)\to \RR,~ f(e^{\alpha Q})=\psi(\alpha) \text{ for  $\|Q\|=1$}$$
and 
$$F:SU(2)\to su(2),~ F(e^{\alpha Q})=\phi(\alpha)Q \text{ for   $\|Q\|=1$}$$
satisfy Goldman's relation (\ref{goldmanrel}).

Recall that $c\in \pi_1(F)$ denotes the based loop obtained by connecting $C$ to the base point. As above, set $f_c([\rho])=f(\rho(c))$ and $F_c(\rho)=F(\rho(c))$.   
 Note that $\psi(\alpha)=\psi(\cos^{-1}(\Real(e^{\alpha Q})))$, so that 
$f_c([\rho])=\psi(\cos^{-1}(\Real(\rho(c)))$. 
Then $e^{tF_c(\rho)}=e^{t\phi(\alpha)Q}$, where $\rho(c)=e^{\alpha Q}$. 
The proof now nearly follows by comparing Goldman's formulae, (\ref{nonsepgold}) and (\ref{sepgold}), to the  formulae of Proposition \ref
{isoform}, replacing the fixed perturbation $\phi$ by the 1-parameter family $t\phi$.  The difference occurs because our set up does not permit us to assume the curve $C$ passes through the base point, whereas Goldman places the base point on $C$. Since $f_c$ is independent of the choice of path from  the base point to $C$, the Hamiltonian flow it induces is also independent of this choice.

Reconciling these two formulae is an exercise in changing base points, using the arc $\gamma_C$ from  the original base point to a point on $C$.       Suppose first that $C$ separates $F$ into $F_1$ and $F_2$ and the base point of $F$ is contained in $F_1$. If  $\beta$ is a loop  in $F_2$ based on $C$, then the longitude of $C$ with respect to the loop $\gamma_C\beta\gamma_C^{-1}$ is just $c=\gamma_C C\gamma_C^{-1}$. Hence Proposition \ref
{isoform} shows that $\Phi_{\pi_{C,t}}(\rho)(\gamma_C\beta\gamma_C^{-1})= e^{t\phi(\alpha)Q}\rho(\gamma_C\beta\gamma_C^{-1})
e^{-t\phi(\alpha)Q}$ where $\rho(c)=e^{\alpha Q}$. Equation (\ref{sepgold}) says $\Xi_t(\rho)(\beta)=e^{t\phi(\alpha)Q}\rho(\beta)
e^{-t\phi(\alpha)Q}$.  If $\beta $ is a loop in $F_1$ based on $C$, then $\gamma_C\beta\gamma_C^{-1}$ can be homotoped off $C$, so that 
 Proposition \ref
{isoform} shows that $\Phi_{\pi_{C,t}}(\rho)(\gamma_C\beta\gamma_C^{-1})=  \rho(\gamma_C\beta\gamma_C^{-1}) $ and  Equation (\ref{sepgold}) says $\Xi_t(\rho)(\beta)=\rho(\beta)
$. These induce the same flow on the quotient $\mathcal{M}(F)$.

The non-separating case is similar.  If $E$ is an embedded curve in $F$ passing through the base point which meets $C$ transversely once and which contains $\gamma_C$, then   Proposition \ref
{isoform} shows that $\Phi_{\pi_{C,t}}(\rho)(\gamma_CE\gamma_C^{-1})=  e^{t\phi(\alpha)Q}\rho(\gamma_CE\gamma_C^{-1}) $, whereas  Equation (\ref{nonsepgold}) says $\Xi_t(\rho)(E)=\rho(E)e^{t\phi(\alpha) Q}$. Now  $\pi_1(F)$ is generated by $E$ and curves of the form $\gamma_C\gamma_1\beta\gamma_1^{-1}\gamma_C^{-1}$, where $\gamma_1$ is a small extension of $\gamma_C$ so that the arc $\gamma_C\gamma_1$  meets $C$ transversely in one point, and $\beta$ is a loop   in $F\setminus C$.  Proposition \ref
{isoform} shows that $\Phi_{\pi_{C,t}}(\rho)(\gamma_C\gamma_1\beta\gamma_1^{-1}\gamma_C^{-1})=e^{t\phi(\alpha)Q}\rho(\gamma_C\gamma_1\beta\gamma_1^{-1}\gamma_C^{-1})e^{-t\phi(\alpha)Q}$ whereas $\Xi_t(\rho)(\gamma_1\beta\gamma_1^{-1})=\rho(\beta).$
Conjugating $\Phi_{\pi_{C,t}}(\rho)$ by $e^{-t\phi(\alpha)Q}$ identifies the flows on $\mathcal{M}(F)$.
\end{proof}

For  the remainder of this article, it will suffice to use  the perturbation functions
 $\phi(\alpha)=t\sin(\alpha)$, so that our perturbation data depends only on the curve $C$ and  the parameter $t$.   It will be convenient to compose  $\Phi_{\pi_{C_i}}$ for various perturbation curves $C_i$, using a different parameter for each perturbation curve.

\begin{df} \label{bigphi} Let ${\bf C}=(C_1,C_2,\dots,C_K)$ be an ordered list of  oriented simple closed curves in $F$.  Extend this to perturbation data in $F\times [0,1],$  parameterized by ${\bf t}=(t_1,\dots,t_k)\in \RR^K$, by pushing $C_i$ into $F\times\{\tfrac{i}{K+1}\}$ and  taking $$\pi_{{\bf t},{\bf C}}=\{ (N_{C_i}, t_i\sin(\alpha))\}.$$
 Then define  {\em the $K$-parameter family of stratum-preserving  isotopies of $\mathcal{M}(F)$ determined by ${\bf C}$}:
 $$\Phi_{\bf C}:\mathcal{M}(F)\times \RR^K\to \mathcal{M}(F),~ \Phi_{\bf C}(\rho,{\bf t})=\Phi_{\pi_{{\bf t},{\bf C}}}(\rho).$$
If we denote $(N_{C_i},t_i\sin(\alpha))$ by $\pi_{t_i,C_i}$, then 
 $$\Phi_{\bf C}(\rho,{\bf t})=\Phi_{\pi_{t_K,C_K}}\circ\Phi_{\pi_{t_{K-1},C_{K-1}}}\circ\dots\circ\Phi_{\pi_{t_{1},C_{1}}}.$$
Proposition \ref{stratpreserving} implies that $\Phi_{\bf C}$ preserves the stabilizer decomposition
of Equation ({\ref{stabde}}).  Theorem \ref{hamiltonian} implies that the restriction to $\mathcal{M}(F)^{\ZZ/2}$ is a composite of symplectomorphisms, each one Hamiltonian isotopic to the identity.
\end{df}

\section{Abundance}\label{abundancesec}

Recall that the irreducible stratum $\mathcal{M}(F)^{\ZZ/2}$ is a smooth manifold of dimension $6n-6$, the abelian stratum $\mathcal{M}(F)^{U(1)}$ is a smooth manifold of dimension $2n-2$, and the central stratum $\mathcal{M}(F)^{SU(2)}$ is a finite set, in fact has  $2^{2n}$ points  \cite{goldman1}.

  The following three propositions form the technical heart of this article.  They will be the basis of stratum-by-stratum general position arguments about the restriction map 
  $\mathcal{M}_\pi(X)\to \mathcal{M}(S_0)$ and the trace conditions  in Sections \ref{local} and \ref{main},  and the trace and anticommutativity conditions in Section \ref{piercedear}.    Let $\mathcal{C}$ denote the set of special perturbation curves  $$ \mathcal{C}=\{ C_{\RN 1}(i),C_{\RN 2}(i),C_{\RN 3}(i)\}_{i=1}^n\cup \{C_{\RN 4}(ij), \dots, C_{\RN {11}}(ij)\}_{i\ne j}$$ defined in Section \ref{snakesonaplane2} 
  and let $ {\bf C} $ denote any ordering of $\mathcal{C}$. For convenience, set $K=3n + 8n(n-1)=8n^2-5n$. 
  
Theorem  \ref{surj} shows that $r^{-1}(R(S^2,\punctures)^{U(1)})\cap \mathcal{M}(F)^{\ZZ/2}$ is a finite union of $(n-1)$-dimensional tori, namely those orbits of the ${\bf T}$ action with non-trivial stabilizer.

\begin{prop}\label{abund1} 
 If $\rho:\pi_1(F)\to SU(2)$ is any representation satisfying
 \begin{enumerate}
\item $\rho$ is irreducible,
\item the restriction of $\rho$ to $\pi_1(S_0)$ is abelian,
 \item $\Real(\rho(A_i))=0$ for $i=1,\dots ,n$,
\end{enumerate}
then the map
$$\Phi_{\bf C}(\rho,-): \RR^{K} \to \mathcal{M}(F)^{\ZZ/2}$$
 has surjective differential at $0$ in $\RR^K$. 
 The set  $r^{-1}(R(S^2,\punctures)^{U(1)})\cap \mathcal{M}(F)^{\ZZ/2}$ of conjugacy classes of   such $\rho$  is compact, being a finite union of $(n-1)$-dimensional tori.  Thus we can find a neighborhood $U_1\subset \RR^K$ of $0$ such that 
$$\Phi_C(\rho,-): U_1  \to \mathcal{M}(F)^{\ZZ/2}$$ is a submersion for  each  $\rho\in r^{-1}(R(S^2,\punctures)^{U(1)})\cap \mathcal{M}(F)^{\ZZ/2}$.
\end{prop}

\begin{prop}\label{abund2} 
 If   $\rho:\pi_1(F)\to SU(2)$ is any representation satisfying
 \begin{enumerate}
\item $\rho$ is irreducible,
\item the restriction of $\rho$ to $\pi_1(S_0)$ is irreducible,
 \item $\Real(\rho(A_i))=0$ for $i=1,\dots ,n$,
\end{enumerate}
then the composite: 
$$T\circ \Phi_C(\rho,-): \RR^K \to \RR^n$$
has surjective differential at 
$0\in \RR^K$,  with $T$ denoting the trace map of Equation {\rm (\ref{Tmap})}.
\end{prop}

\begin{prop}\label{abund3}
If 
 $\rho:\pi_1(F)\to SU(2)$ is any representation satisfying
 \begin{enumerate}
\item $\rho$ is abelian,
 \item $\Real(\rho(A_i))=0$ for $i=1,\dots ,n$,
\end{enumerate}
then the composite:
$$  \RR^{K} \xrightarrow{\Phi_{\bf C}(\rho,-) }\mathcal{M}(F) \xrightarrow{T}  \RR^n$$
has surjective differential at 
$0\in \RR^K$.   
 The set $  \mathcal{M}(F)^{U(1)}  \cap T^{-1} (0)$  of conjugacy classes of such $\rho$  
  is compact, and hence there exists a neighborhood $U_3$ of $0\in \RR^K$ such that    $$ T\circ \Phi_{C}(\rho, -): U_3\to \RR^n  $$   is a submersion  for every $\rho\in  \mathcal{M}(F)^{U(1)} \cap T^{-1} (0).$

 \end{prop}
  Since the proofs of Propositions \ref{abund2} and \ref{abund3} require less machinery, we will provide these before taking up the proof of Proposition \ref{abund1}.   
 
\bigskip 
\noindent{\em Proof of Proposition \ref{abund2}.}  
Let $\rho\in \mathcal{M}(F)^{\ZZ/2}$ be a representation that restricts to an irreducible representation on $\pi_1(S_0)$, and that satisfies $T(\rho)=0$.   For each $j=1,\dots, n$, 
 $\Real(\rho(A_j))=0$ and so  $P_j=\rho(A_j)$
 is a purely imaginary unit quaternion.  Also write $\rho(D_j)=e^{\beta_jR_j}$ for $R_j$ a purely imaginary unit quaternion and $\beta_j\in [0,\pi]$.

 We next choose a sequence of curves   ${\bf C'}=(C_1, C_2,\dots, C_n)$ in $\mathcal{C}$ and compute $$M_{ij}=\left. \frac{d}{dt}\Real(\Phi_{C_i,t}(\rho)(A_j))\right|_{t=0}, ~j=1,\dots, n.$$
The choice of   $C_i$ is made so that this derivative $M_{ij}$ is nonzero for $j=i$ and there is one fixed index $j_0$ such that $M_{ij}=0$ if $j\not\in \{ i, j_0 \}$, which implies that the matrix $\left[ M_{ij} \right]$ has full rank.     In particular,  this shows that the differential of $T\circ \Phi_{\bf C'}(\rho,{\bf t})$ at $\rho$  maps onto $\RR^n$.  
   
   \medskip
\noindent{\bf Case 1.} {\em Suppose that $\Real(P_iR_i)\ne 0$ and $\sin \beta_i\ne 0$.}    
Take $C_i=C_{\RN 1}(i)$, so that $C_i$ intersects $A_i$ once and misses $A_j,~j\ne i$,  Referring to Table \ref{table1} one sees that  $ \rho(\lambda_{C_i}(A_i)) =\rho(A_iD_iA_i^{-1})=P_ie^{\beta_iR_i}P_i^{-1}$, and $A_i\cdot C_i=-1$.   Proposition \ref{isoform} implies that  
$$\Phi_{C_{i,t}}(\rho(A_i))=
P_ie^{-t\sin (\beta_i)R_i}P_i^{-1}\rho(A_i)=P_ie^{-t\sin (\beta_i)R_i}$$
and so
$$\Real(\Phi_{C_{i,t}}(\rho(A_i)))=-\sin (t\sin\beta_i)\Real(P_iR_i) $$
which has non-zero derivative at $0$.  Since $C_i$ misses $A_j$   as well as the path from the base point to $A_j$, for $j\ne i$,  the first   assertion of Proposition \ref{isoform}  implies that $ \Phi_{C_{i,t}}(\rho(A_j)) = \rho(A_j)$.   Hence, $M_{ii}\neq 0$ and $M_{ij}=0$ for $j\neq i$.

 \medskip 

\noindent{\bf Case 2.}  {\em Suppose that $\sin\beta_i=0$ so that $\rho(D_i)=\pm 1$.} Take $C_i= C_{\RN 3}(i)$. Again $C_i$ meets $A_i$ transversely in one point and  misses  $A_j  $   and also misses the paths from the base point to $A_j$, for all $j\ne i$.  
From Table \ref{table1} we read $   \rho(\lambda_{C_i}(A_i))=\rho(A_i D_i) =\pm P_i=e^{\pm \frac\pi 2 P_i}$. Proposition \ref{isoform} implies that $M_{ij}=0$ for $j\neq i$, and that
$$\Real(\Phi_{C_{i,t}}(\rho(A_i)))=\Real(e^{-t\sin (\pm \frac\pi 2)P_i }P_i)=\pm \sin t.
$$
This has non-zero derivative at 0 and so  we have $M_{ii}\neq 0$. 
  \medskip
  
 \noindent{\bf Case 3.} {\em Suppose that $\Real(P_iR_i)= 0$ and $\cos \beta_i\ne 0$. }  Set $C_i=C_{\RN 3}(i)$, as in Case 2.  Then 
 $$ \rho(\lambda_{C_i}(A_i))=\rho(A_i D_i)= P_i e^{\beta_iR_i}.  $$ 

 Since $\Real(P_i)=\Real(P_iR_i)=0$, $\Real(P_ie^{\beta_iR_i})
 =0$. Hence $ \rho(\lambda_{C_i}(A_i))=
 e^{\frac\pi 2 P_i e^{\beta_iR_i}}$ and 
 $$\Real(\Phi_{C_{i,t}}(\rho(A_i)))=\Real(e^{ \mp   t   P_ie^{\beta_iR_i}} P_i)=\pm\sin t\cos\beta_i 
$$
which has non-zero derivative at $0$, so,  
again, we have $M_{ii}\neq 0$ and $M_{ij}=0$ for $j\neq i$.  
  \medskip
 
  \noindent{\bf Case 4.}  {\em Suppose $\Real(P_iR_i)= 0$ and $\cos \beta_i= 0$ and that there exists $j\ne i$ so that $P_j\ne \pm P_i$.}   Since $P_i$, $R_i$, and $P_iR_i$ form an orthonormal basis of $su(2)$, and $P_j\ne \pm P_i$, 
  $$\text{ either }\Real(R_iP_j)=-\langle R_i, P_j\rangle\ne 0\text{ or }\Real(P_iR_iP_j)=-\langle P_iR_i, P_j\rangle\ne 0.$$
  
If $\Real(R_iP_j)\ne 0$, choose $C_i=C_{\RN 6}(ij)$.  Then 
$$ \rho(\lambda_{C_i}(A_i)) =\rho(A_iA_jA_iD_iA_i^{-1})=P_iP_jP_iR_i(-P_i)=-P_iP_jR_i.$$
Choose $\alpha, Q$ so that  $-P_iP_jR_i=e^{\alpha Q}$.   Then 
$$\Real(\Phi_{C_{i,t}}(\rho)(A_i))=\Real(e^{\pm t\sin(\alpha)Q}P_i)=\pm\sin(t\sin(\alpha))\Real(QP_i).
$$
 This has derivative  at $t=0$ equal to 
 $$  \pm \sin(\alpha) \Real(QP_i)=\pm\Real(e^{\alpha Q}P_i)=\pm\Real(  P_iP_jR_iP_i)=\pm\Real(R_iP_j)\ne 0.
$$   

If instead, $\Real(P_iR_iP_j)\ne 0$, choose $C_i=C_{\RN 4}(ji)$.  Then 
$$ \rho(\lambda_{C_i}(A_i))=\rho(D_i^{-1}A_j)=
-R_iP_j .$$
Choose $\alpha', Q'$ so that  $-R_iP_j=e^{\alpha' Q'}$.
$$\Real(\Phi_{C_{i,t}}(\rho)(A_i))=\Real(e^{\pm t\sin(\alpha')Q'}P_i)=\pm\sin(t\sin(\alpha'))\Real(Q'P_i).
$$

 This has derivative  at $t=0$ equal to 
 $$   \sin(\alpha')\Real(Q'P_i)=\Real(e^{\alpha Q'}P_i)= -\Real( R_iP_jP_i)=-
 \Real(P_iR_iP_j)\ne 0.
 $$   
 Either way,  $M_{ii}\neq 0$.  When $j\neq i$,   then $C_i$ does not intersect $A_j$.  If $C_i$ also misses the path from the base point to $A_j$ then $M_{ij}=0$ by the first part of Proposition \ref{isoform}. But even if $C_i$ intersects the path from the base point to $A_j$, the third part of Proposition \ref{isoform} shows that $ \Phi_{C_{i,t}}(\rho)(A_j)$ is a conjugate of $\rho(A_j)$ for all $t$, so that 
 $\Real(\Phi_{C_{i,t}}(\rho)(A_i))$ is independent of $t$ and hence again $M_{ij}=0$.

   \medskip
   
     \noindent{\bf Case 5.}    {\em Suppose $\Real(P_iR_i)= 0$ and $\cos \beta_i= 0$ and that for all $j=1,\dots, n$,  $P_j= \pm P_i$.} 
     For clarity, write $P=P_i$ and $R=R_i$, so $P_j=\pm P$ for all $j$.
     By hypothesis,  the restriction of $\rho$ to $\pi_1(S_0)$ is non-abelian, so  there exists at least one index   $j\in \{ 1, \dots, n\}$ such that $\rho(B_j)\ne \pm P$;  fix $j_0$ to be the lowest such $j$.   Note that Case 5 does not arise for $i=j_0$,  since $\rho(B_i)=\rho(D_iA_iD_i^{-1})=RP(-R)=-P$.  

 The fact that  $\rho(B_{j_0})=\rho(D_{j_0}A_{j_0}D_{j_0}^{-1})\ne \pm P$ implies that  $\rho(D_{j_0})=e^{\beta_{j_0} R_{j_0}}\ne \pm 1$, and hence  $\sin\beta_{j_0}\ne 0$.   It also implies that
 $$\text{ either }\Real(RPR_{j_0})\ne 0 \text{ or }\Real(RR_{j_0})\ne0.$$

 If $\Real(RR_{j_0})\ne 0$, choose $C_i=C_{\RN 7}(j_0i)$. A calculation just like in Case 4  shows that 
 $$ M_{ij}= 
 \left. \frac{d}{dt}  \Real(\Phi_{C_{i,t}}(\rho)(A_j))\right |_{t=0}=
 \begin{cases}
 \pm \sin \beta_{j_0}\Real(RR_{j_0})\ne 0 &j=i \text{ or } j_0\\
 0& j\ne i, {j_0}. \end{cases}$$

If $\Real(RPR_{j_0})\ne 0$, choose $C_i= C_{\RN 5}(j_0i)$. Again one    calculates 
$$  M_{ij}=  \left. \frac{d}{dt} \Real(\Phi_{C_{i,t}}(\rho)(A_j))\right |_{t=0}=
 \begin{cases}
 \pm \sin \beta_{j_0}\Real(RPR_{j_0})\ne 0 &j=i\text{ or } j_0\\
 0& j\ne i, {j_0}. \end{cases}$$

\bigskip 

 It follows that, for the curves $C_i$ constructed according to Cases 1-5,  the matrix $[M_{ij}]$ has nonzero diagonal entries, and all other entries zero except for the $j_0$th column, so it  has nonzero determinant.   Thus, for ${\bf C'} = (C_1, \dots, C_n)$ so constructed, 
this matrix $[M_{ij}]$, which  is the differential at ${\bf t} = {\bf 0} $ of the composite
$$\RR^n\xrightarrow{\Phi_{\bf C'}(\rho, {\bf t})}\mathcal{M}(F)\xrightarrow{T} \RR^n,$$ has  full rank. This is an open condition, so it holds near ${\bf 0}$ as well.

  Reordering the curves in ${\bf C'}$ does not change the rank of this differential, and adding more perturbation curves cannot  decrease the rank, and hence the composite
$$\RR^K\xrightarrow{\Phi_{\bf C}(\rho, {\bf t})}\mathcal{M}(F)\xrightarrow{T} \RR^n$$ has  full rank, completing the proof.
\qed

\bigskip

\noindent{\em Proof of Proposition \ref{abund3}.} Assume that  $\rho:\pi_1F\to SU(2)$ is abelian and $\Real(\rho(A_i))=0$ for all $i=1,\dots , n$. Then we can choose a purely imaginary unit quaternion $P$ and real numbers  $\beta_i\in [0,2\pi )$   so that 
$\rho(A_i)=\pm P$  and $\rho(D_i)=e^{\beta_i P}$.  

For each  $i=1,\dots,n$, one of Cases 1 and 2 in the proof of Proposition \ref{abund2} applies, so one may choose $C_i=C_{\RN 1}(i)$ or $C_i=C_{\RN 3}(i)$ so that
$$M_{ij} = \left. \frac{d}{dt}\Real(\Phi_{C_i,t}(\rho)(A_j))  \right|_{t=0}=
\begin{cases} a_{i}\ne 0 & j=i\\
0& j\ne i\end{cases} $$
As in the proof of Proposition \ref{abund2}, this produces a sequence  ${\bf C'}=(C_1, \dots, C_n)$  such that the differential at ${\bf t}=0$ of the composite 
$$\RR^{n}\xrightarrow{\Phi_{\bf C'}(\rho,{\bf t})}\mathcal{M}(F)\xrightarrow{T} \RR^n$$
is represented by a diagonal matrix with nonzero diagonal entries, and hence is surjective. Again, this  shows that the differential near ${\bf t}=0$ of the composite  $$\RR^K\xrightarrow{\Phi_{\bf C}(\rho, {\bf t})}\mathcal{M}(F)\xrightarrow{T} \RR^n$$ has  full rank, as required.  

 That $  \mathcal{M}(F)^{U(1)}  \cap T^{-1} (0)$ is compact follows from the fact that  $  \mathcal{M}(F)^{SU(2)} \cup   \mathcal{M}(F)^{U(1)} $ is compact and that
$$ \mathcal{M}(F)^{U(1)}  \cap T^{-1} (0)= \left( \mathcal{M}(F)^{SU(2)} \cup \mathcal{M}(F)^{U(1)}\right)  \cap T^{-1} (0).$$ 
\qed

\medskip

  Finally, we return to the proof of Proposition \ref{abund1}.  Because this involves showing a map to $\mathcal{M}(F)^{\ZZ/2}$ is a submersion, and the tangent space of this manifold is identified with first cohomology, this will require identifying cocycles associated to paths of representations.    
  
  Referring back to Section \ref{tangentspaces} and \ref{snakesonaplane},  we fix the standard set $\{A_i,D_i\}_{i=1}^n$ of generators of $\pi_1(F)$,  represented by the embedded curves together with their indicated paths from the base point shown in Figure \ref{drawingfig},  so we can identify     $$C^1(F;su(2)_{\ad{\rho}})=\operatorname{Funct} (\{A_i,D_i\}_{i=1}^n,su(2)).$$
  If $\rho_t, t\in (-\ep,\ep)$ is a path of representations of $\pi_1(F)$ and $z\in\operatorname{Funct} (\{A_i,D_i\}_{i=1}^n,su(2))$  is a function so that 
$$\rho_t(E)=e^{tz(E)}\rho_0(E) \text{ for all } E\in \{A_i,D_i\}, $$ 
or, more generally, so that
$$z(E)=\left. \frac{d}{dt}~\rho_t (E)\right|_{t=0}  \left( \rho_0 (E) \right)^{-1},$$ then $z$ is a 1-cocycle, and the tangent vector to $\rho_t$ at $t=0$ is identified with this cocycle via the identification of Equation (\ref{andreweil}).

\medskip

Suppose that $C$ is a special perturbation curve.  Consider the 1-parameter family of perturbations $\pi_{C_t}=(N_C, t\sin(\alpha))$.  This determines an isotopy 
$t\mapsto \Phi_{\pi_{C_t}}:\mathcal{M}(F)\to \mathcal{M}(F)$ with $\Phi_{\pi_{C_0}}={\rm Id}$.   Given $E\in \{A_i,D_i\}$, let $\gamma_E$ denote its path from the base point. Proposition \ref{prop1} shows that
\begin{equation}
\label{path}\Phi_{\pi_{C_t}}(\rho)(E)=\begin{cases}\rho(E)&\text{if $C\cap E=\emptyset$ and $C\cap \gamma_E=\emptyset$},\\
e^{\pm t\sin\alpha  ~Q}\rho(E) &\text{if $C\cap E\ne\emptyset$ and $C\cap \gamma_E=\emptyset$,}\\
e^{\pm t\sin\alpha  ~Q}\rho(E)e^{\mp t\sin\alpha  ~Q}&\text{if $C\cap E=\emptyset$ but $C\cap\gamma_E\ne \emptyset$,}
\end{cases}
\end{equation}
where   $\lambda_C(E)$ is sent to $e^{\alpha Q}$ by $\rho$.  Thus the 1-cocycle $z_{C}$ corresponding to the path $\Phi_{\pi_{C_t}}(\rho)$ is given by

\begin{equation}
\label{cocycle}
z_C(E)=\begin{cases}0&\text{if $C\cap E=\emptyset$ and $C\cap \gamma_E=\emptyset$},\\
 \pm  \sin\alpha  ~Q  &\text{if $C\cap E\ne\emptyset$ and $C\cap \gamma_E=\emptyset$,}\\
\pm \sin\alpha (Q-\rho(E)Q\rho(E)^{-1})&\text{if $C\cap E=\emptyset$ but $C\cap\gamma_E\ne \emptyset$.}
\end{cases}
\end{equation}

\bigskip
 
  Next, we characterize traceless representations of $\pi_1(F)$ which are abelian on $\pi_1(S_0)$.   
\begin{lem} \label{ir-re}Suppose $\rho:\pi_1(F)\to SU(2)$  satisfies $\Real(\rho(A_i))=0$ for $i=1,\dots ,n$ and  restricts to an  abelian representation on $S_0$.  Let $P=\rho(A_1)$. Fix arbitrarily a purely imaginary unit quaternion $Q$ so that $\Real(PQ)=0$.  Then there exist signs $\ep_i,\delta_i\in \{\pm 1\}, ~ i=1,\dots,n$, with $\ep_1=1$, and $e^{\alpha_i\bbi}\in S^1,  ~ i=1,\dots,n$ so that 
 $$\rho(A_i)=\ep_i P, ~\rho(D_i)=\begin{cases} e^{\alpha_iP}&\text{ if } \epsilon_i=\delta_i,\\
Qe^{\alpha_iP} &\text{ if } \epsilon_i\ne\delta_i.
\end{cases}$$
If $\rho$ is irreducible, at least one index satisfies $\ep_i\ne \delta_i$.
\end{lem}
 \begin{proof}
  Since $\rho$ is abelian on $S_0$, there exist signs $\ep_i  \in \{\pm 1\}$ with $\ep_1=1$, so that 
 $\rho(A_i)=\ep_i P$.  Since $B_i=D_iA_iD_i^{-1}$, it follows that $\Real(\rho(B_i))=0$, and since 
 $\rho$ is abelian on $S_0$, there exist signs $\delta_i  \in \{\pm 1\}$ so that 
 $\rho(B_i)=\delta_i P.$ 
 
The set of  unit quaternions $R$ satisfying $RPR^{-1}=P$ is precisely the circle $\{e^{\alpha P}\}$. Furthermore, the set of unit quaternions $R$ satisfying $RPR^{-1}=-P$ is precisely the circle $\{Qe^{\alpha P}\}$ of purely imaginary unit quaternions orthogonal to $P$.
 Since $B_i=D_iA_iD_i^{-1}$, there exist $\alpha_i\in \RR$ so that
$$\rho(D_i)=\begin{cases} e^{\alpha_iP}&\text{ if } \epsilon_i=\delta_i,\\
Qe^{\alpha_iP} &\text{ if } \epsilon_i\ne\delta_i.
\end{cases}$$
If  $\epsilon_i=\delta_i$ for all $i$, then $\rho$ is not irreducible.
\end{proof}

\noindent{\em Proof of Proposition \ref{abund1}.}
 Assume that $\rho$ is an irreducible  representation of the form described in Lemma \ref{ir-re}.  For convenience, since   some $\ep_i\ne\delta_i$, reindex  so that $\ep_1=1$ and $\delta_1=-1$ and  conjugate $\rho$ if necessary so that $P=\bbi$.  Take $Q=\bbj$.   Thus
$$\rho(A_i)=\ep_i\bbi, ~\rho(B_i)=\delta_i\bbi,~ \rho(D_i)= \begin{cases} e^{\alpha_i\bbi}&\text{ if } \epsilon_i=\delta_i,\\
\bbj e^{\alpha_i\bbi} &\text{ if } \epsilon_i\ne\delta_i.
\end{cases}
$$

Recall from Section \ref{tangentspaces} that the tangent space of $\mathcal{M}(F)$ at $\rho$ is identified with the $(6n-6)$-dimensional cohomology $H^1(F;su(2)_{\ad \rho})$. The space $Z^1(F;su(2)_{\ad \rho})$ of 1-cocycles is a $(6n-3)$-dimensional subspace of  $$C^1(F;su(2)_{\ad \rho})=\operatorname{Funct} (\{A_i,D_i\}_{i=1}^n,su(2))\cong su(2)^{2n}.$$ The space of 1-coboundaries is 3-dimensional since $\rho$ is irreducible. 

We first find perturbation isotopies which produce a collection of cocycles 
projecting to a $(6n-6)$-dimensional subspace of 
 $\operatorname{Funct} (\{A_i,D_i\}_{i=2}^n,su(2))$ (note the indexing). We then find three more perturbation isotopies producing cocycles spanning a 3-dimensional subspace of $$\{z\in C^1(F;su(2)_{\ad \rho}))~|~ z(A_i)=0=z(D_i) \text{ for } i>1\}.$$
 Their union therefore span $Z^1(F;su(2)_{\ad \rho})$.  It follows that (the differentials of) these perturbation isotopies, acting on representations modulo conjugation,  span all of $H^1(F;su(2)_{\ad \rho})$.
 
 \medskip

 Take ${\bf C'}$   to be some ordering of the $9n-6$ perturbation curves:
 $$\{C_{\rm I}(j)\}_{j=1}^n \cup \{C_{\rm II}(j)\}_{j=1}^n\cup \{C_{\rm III}(j)\}_{j=1}^n\cup \{C_{\RN 8}(1j)\}_{j=2}^n\cup \{C_{\RN 8}(j1)\}_{j=2}^n
\hskip.5in $$
 $$\hskip.5in \cup \{C_{\RN 9}(1j)\}_{j=2}^n\cup \{C_{\RN {10}}(1j)\}_{j=2}^n\cup \{C_{\RN {11}}(1j)\}_{j=2}^n\cup\{C_{\RN {11}}(j1)\}_{j=2}^n
 $$
from among those  listed in Table \ref{table1} and illustrated in  Figures \ref{drawing3afig}, \ref{drawing3a1fig}, \ref{drawing3a2fig},   \ref{drawing3a3fig} and  \ref{drawing3a4fig}.
  For each $C\in   {\bf C'} $, we compute the value of the cocycle  $z_{C}$ on  $E\in \{A_i,D_i\}_{i=2}^n$.  
Recall that  $\lambda_{C}(E)$ denotes the longitude of $C$ with respect to $E$.

\medskip

Consider first $C_{\RN 1}(j)\in {\bf C'}$. We compute $z_{C_{\RN 1}(j)}(A_j)$.  Table \ref{table1} gives $\lambda_{C_{\RN 1}(j)}(A_j)=A_jD_jA_j^{-1}$,   from which one computes 
$$\rho(\lambda_{C_{\RN 1}(j)}(A_j))=  \begin{cases} e^{\alpha_j\bbi}&\text{ if } \epsilon_j=\delta_j,
\\
-\bbj e^{\alpha_j\bbi} = e^{-\frac\pi 2 \bbj e^{\alpha_j\bbi}}&\text{ if } \epsilon_j\ne\delta_j.
\end{cases}
$$
Since $A_j\cdot C_{\RN 1}(j)=-1$, Equation (\ref{cocycle}) yields
$$z_{C_{\RN 1}(j)}(A_j)=  \begin{cases} -\sin \alpha_j\bbi&\text{ if } \epsilon_j=\delta_j,
\\
  \bbj e^{\alpha_j\bbi}&\text{ if } \epsilon_j\ne\delta_j.
\end{cases}
$$
 Figure \ref{drawing3afig}  shows that the perturbation curve $C_{\RN 1}(j)$ misses $A_i, \gamma_{A_i}, D_{i},$ and $\gamma_{D_i}$ when $i\ne j$. It also misses $D_j$ and $\gamma_{D_j}$, and hence  $z_{D_j}(A_i)=0, i\ne j$, and $z_{C_{\RN 1}(j)}(D_i)=0$ for all $i$.

Next, consider $C_{\RN 2}(j)\in {\bf C'}$.  Since, from Table 1, $\rho(\lambda_{C_{\RN 2}(j)}(D_j))=\rho(D_jA_jD_j^{-1})=\delta_j\bbi$, hence $z_{C_{\RN 2}(j)}(D_j)=\pm\bbi.$  Henceforth we ignore signs, since it turns out they do not matter.  As in the previous case,  $z_{C_{\RN 2}(j)}$ vanishes on $D_i, i\ne j$ and on  $A_i $ for all $i$.  

As a third example, consider $C_{\RN 3}(j)\in {\bf C'}$. This time there are two intersecting curves, $D_j$ and $A_j$. For $D_j$, $ \lambda_{C_{\RN 3}(j)}(D_j)=D_jA_j$. When $\ep_j=\delta_j$, this is sent to $e^{\alpha_j\bbi}\ep_j\bbi= e^{(\alpha_j+\ep_j \frac\pi 2)\bbi}$ so that $z_{C_{\RN 3}(j)}(D_j)=\pm \sin (\alpha_j+\ep_j \frac\pi 2)\bbi$ which equals  $\cos(\alpha_j)\bbi$ up to sign.    When $\ep_j\ne \delta_j$, $ \rho(\lambda_{C_{\RN 3}(j)}(D_j))=\pm \bbk e^{\alpha_j\bbi}$, so that in this case $z_{C_{\RN 3}(j)}(D_j)=\pm \bbk  e^{\alpha_j\bbi}$.
Since $\rho(\lambda_{C_{\RN 3}(j)}(A_j))=\rho(A_jD_j)$, we obtain $z_{C_{\RN 3}(j)}(A_j)=\pm\cos(\alpha_j)\bbi$ or $\pm \bbk e^{\alpha_j\bbi}$ according to whether $\ep_j=\delta_j$ or $\ep_j\ne\delta_j$.   Again, $z_{C_{\RN 3}(j)}(E)=0$ for  $E\ne D_j, A_j$.

The calculation of $z_C$    for   perturbation curves $C$ indexed by both $i$ and $j$ is more involved. Consider, for example,  $C_{\RN {10}}(1j)$. First note that $z_{C_{\RN {10}}(1j)}(D_i)=0$ for all $i$  since $C_{\RN {10}}(1j)$ misses $D_i$ and $\gamma_{D_i}$. Similarly $z_{C_{\RN {10}}(1j)}(A_i)=0$ for all $i\ne 1,j$.  
Table \ref{table1} gives 
$\lambda_{z_{C_{\RN {10}}(1j)}}(A_j)$. The expression is complicated, but, upon applying $\rho$, it simplifies greatly because $\rho(A_iB_i^{-1})=\pm 1$. Hence (recalling that $\ep_1=-\delta_1$)
$$\rho(\lambda_{z_{C_{\RN {10}}(1j)}}(A_j))=\pm \rho(D_1D_j)=
\pm \begin{cases}\bbj e^{(\alpha_1+\alpha_j)\bbi}&\text{ if } \epsilon_j=\delta_j,
\\
e^{(\alpha_j-\alpha_1)\bbi}&\text{ if } \epsilon_j\ne\delta_j.
\end{cases}
$$
Therefore 
$$z_{C_{\RN {10}}(1j)}(A_j)=
\pm \begin{cases}\bbj e^{(\alpha_j+\alpha_1)\bbi}&\text{ if } \epsilon_j=\delta_j,
\\
\sin(\alpha_j-\alpha_1)\bbi&\text{ if } \epsilon_j\ne\delta_j.
\end{cases}
$$

The rest of the calculations, left to the reader, are carried out in the same manner.  The results are tabulated (with signs ignored) in Table \ref{table2}. The reader should keep in mind that $\epsilon_1=-\delta_1$.  Only the calculations we require are given;  other entries are left blank.

\begin{table}
\begin{center}
  \begin{tabular}{|c|c|c|c|}
  \hline
 Pert. curve $C$ & Int. curve $E$&$z_C(E)$ when $\ep_j=\delta_j$&$z_C(E)$ when $\ep_j\ne\delta_j$\\
   \hline \hline
   $C_{\RN 1}(j)$& $A_j$&$\sin\alpha_j\bbi $& $ \bbj e^{\alpha_j\bbi}$\\ \hline
    $C_{\RN 2}(j)$& $D_j $&$ \bbi $& $\bbi $\\ \hline
 $ C_{\RN 3}(j)$ & $D_j$ &  $ \cos\alpha_j\bbi $ &   \\ \cline{2-4}
   & $A_j$ & $ \cos\alpha_j\bbi $ & \\ 
  \hline
$ C_{\RN{8}}(1j)$, $j>1$ & $A_j$ & & $\bbk e^{\alpha_j\bbi}$\\ \hline
$ C_{\RN{10}}(1j)$, $j>1$ 
& $A_j$ &$\bbj e^{(\alpha_j+\alpha_1)\bbi}$ & $\sin(\alpha_j-\alpha_1)\bbi$\\ \hline
$ C_{\RN{9}}(1j)$, $j>1$  
& $A_j$ &$\bbk e^{(\alpha_j+\alpha_1)\bbi} $ & $\cos(\alpha_j-\alpha_1)\bbi$\\ \hline

$C_{\RN{8}}(j1)$, $j>1$&$D_j$&$\bbk e^{\alpha_1\bbi}$ & $\bbk e^{\alpha_1\bbi}$\\ \hline

$C_{\RN {11}}(1j)$, $j>1$&$D_j$&$ \bbj e^{\alpha_1\bbi}$& $ \bbj e^{\alpha_1\bbi}$\\
\hline

\end{tabular}
\medskip
 \caption {Values  of selected cocycles $z_C$ on selected intersecting curves, used in the proof of Proposition \ref{abund1}.}  \label{table2} 
\end{center}
\end{table}
 
 To complete the proof of Proposition \ref{abund1}, we next make several observations regarding linear independence of certain of these cocycles.  Fix $j>1$.    Consider the three cocycles
$$d_1(j)=z_{C_{\RN 2}(j)}, ~d_2(j)=z_{C_{\RN {11}}(1j)}, ~d_3(j)=z_{C_{\RN 8}(j1)}.$$ 
These take $D_j$ respectively to:
$$\bbi, ~\bbj e^{\alpha_1\bbi}, ~\bbk e^{\alpha_1\bbi},$$
a basis of $su(2)$. The three perturbation curves 
$C_{\RN 2}(j), ~C_{\RN {11}}(1j),$ and $C_{\RN 8}(j1)$ miss $\gamma_{A_i}A_i \gamma_{A_i}^{-1}$ for $i>1$ and miss $\gamma_{D_i}D_i\gamma_{D_i}^{-1}$ for $i\ne 1,j$ and therefore $d_1(j), d_2(j)$ and $d_3(j)$ vanish on $\{A_i,D_i\}_{i=2}^n\setminus \{D_j\} $.

Suppose that   $\ep_j=\delta_j$. Then define  
$$a_1(j)= \begin{cases} \tfrac{1}{\sin\alpha_j}z_{C_{\RN 1}(j)}&\text{ if } \sin\alpha_j\ne 0,\\ z_{C_{\RN  3}(j)}\pm z_{C_{\RN 2}(j)}&\text{ if } \sin\alpha_j= 0,\end{cases}$$
with the sign chosen, in the second case, so that $a_1(j)(D_j)=0$. 
Define $$ a_2(j)= z_{C_{\RN {10}}(1j)}, ~a_3(j)=z_{C_{\RN 9}(1j)}.$$ The 3-cocycles   map $A_j$ respectively to (up to signs)
$$\bbi, ~ \bbj e^{(\alpha_j+\alpha_1)\bbi}, ~\bbk e^{(\alpha_j+\alpha_1)\bbi}.$$
The perturbation curves $C_{\RN 1}(j), C_{\RN 2}(j), C_{\RN 3}(j),C_{\RN {10}}(1j),$ and $  C_{\RN 9}(1j)$ miss $\gamma_{A_i}A_i \gamma_{A_i}^{-1}$ for $i\ne 1, j$ and  $C_{\RN 1}(j), C_{\RN {10}}(1j),$ and $  C_{\RN 9}(1j)$  miss $\gamma_{D_i}D_i\gamma_{D_i}^{-1}$ for $i>1$ and therefore the three cocycles  $a_1(j), a_2(j), a_3(j)$ vanish on $\{A_i,D_i\}_{i=2}^n\setminus \{A_j\} $.

Suppose instead that  $\ep_j=-\delta_j$. 
Define 
$$a_1(j)= \begin{cases} z_{C_{\RN {10}}(1j)}&\text{ if } \sin(\alpha_j-\alpha_1)\ne 0,\\ z_{C_{\RN  9}(1j)} &\text{ if } \sin(\alpha_j-\alpha_1)= 0.\end{cases}$$
Then define 
$$ a_2(j)= z_{C_{\RN {1}}(j)}, ~a_3(j)=z_{C_{\RN 8}(1j)}.$$ 
Then a similar calculation to that used  in the first case shows that 
$a_1(j), a_2(j), a_2(j)$ map $A_j$ to a basis of $su(2)$ and vanish on 
$\{A_i,D_i\}_{i=2}^n\setminus \{A_j\} $.

 Hence the $(6n-6)$ cocycles $\{a_1(j), a_2(j), a_3(j), d_1(j), d_2(j), d_3(j)\}_{j=2}^{n}$  are linearly independent in 
$$\operatorname{Funct} (\{A_i,D_i\}_{i=2} ^n,su(2)) .$$

  We now add three more cocycles to our to our collection: $$z_1=z_{C_{\RN 1}(1)}, z_2=z_{C_{\RN 2}(1)},\text{ and}z_3=z_{C_{\RN 3}(1)}.$$ Since $\ep_1=1$ and $\delta_1=-1$, we see from Table \ref{table2} that these three cocycles map $\{A_1,D_1\}$  onto a 3-dimensional subspace of $$\{z\in \operatorname{Funct} (\{A_i,D_i\}_{i=1} ^n,su(2)) \mid z(A_j)=z(D_j)=0 \text{ for } j>1\}.$$
  
 Hence, the set of $6n-3$ cocycles
 $$\{z_1,z_2,z_3, a_1(j),a_2(j), a_3(j), d_1(j), d_2(j), d_3(j), ~2\leq j\le n\}$$
span a subspace of 
$$Z^1(F;su(2)_{\ad{\rho}})\subset C^1(F;su(2)_{\ad{\rho}})=\operatorname{Funct} (\{A_i,D_i\}_{i=1} ^n,su(2))$$
of dimension $6n-3$.  Each cocycle in this set is a linear combination of the $9n-6$ cocycles in the set $\{z_C~|~ c\in {\bf C'}\}$ and hence the span of $\{z_C~|~ c\in {\bf C'}\}$  is at least $(6n-3)$-dimensional.   As explained at the start of the proof, this shows that 
the cohomology classes of these $9n-6$ cocycles  span $H^1(F;su(2)_{\ad{\rho}})=T_\rho\mathcal{M}(F)$. 

That $\Phi_{{\bf C'} }(\rho,-):\RR^{9n-6}\to \mathcal{M}(F)$ is submersive near zero now follows from the fact that, by Equation (\ref{cocycle}), 
$$
\left. \frac{\partial}{\partial t_i} \Phi_{{\bf C'} }(\rho,(t_1, \dots, t_n)\right|_{\bf{t}=\bf {0}} =z_{C_i}.$$ 

 As in the proof of Proposition  \ref{abund2}, expanding ${\bf C'}$ to  ${\bf C}$ cannot decrease the rank of the differential, so that 
$$\Phi_{\bf C}(\rho,-):\RR^{K}\to \mathcal{M}(F)$$ is submersive near zero.    
 \qed

\section{Transversality near the singular points}\label{local}

Now let $Y$ be a 3-dimensional $\ZZ$-homology ball,  
equipped  with an identification of its boundary with $S^2$.  Suppose that  $L\subset Y$ a properly embedded $1$-manifold with $n$  components, and denote the boundary points of the $i$th arc by    $a_i, b_i$. We identify $\partial Y\setminus nbd(\partial L)$  with $S_0$.

Let $X$ denote the 3-manifold obtained by removing a tubular neighborhood of $L$ from $Y$
$$X=Y\setminus nbd(L).$$
Thus $\partial X$ is identified with $F$.

The proof of our main result, Theorem \ref{thm1}, involves a stratum-by-stratum transversality argument, using the abundance of perturbations established in Section \ref{sec2}, combined with some symplectic  arguments.  In this section we establish   the first of these transversality arguments, identifying the structure  of the  traceless character variety of $(Y,L)$ near the finitely many singular points, namely the abelian representations, 
for all small perturbations. 

\medskip

 We first recall a transversality result concerning the perturbed flat moduli space for a 3-manifold $X$ with genus $n$ boundary $F$.  Let $\mathcal{X}$ denote the vector space of perturbation functions defined in Equation (\ref{pertfun}).

\begin{thm}{\cite[Theorem 15]{herald1}}\label{herald} There exists a disjoint union of embeddings
$N_i:D^2\times S^1\to X,$ $i=1,\dots,  p,$ and a neighborhood $\mathcal{U}_1$ of $0\in \mathcal{X}^p$, such that for $\pi $ in   a residual  subset of $\mathcal{U}_1$,   if $n\geq 3$, then
 \begin{equation}
\label{strat1}
 \mathcal{M}_\pi(X)= 
  \mathcal{M}_\pi(X)^{\ZZ/2,\ZZ/2}\sqcup \mathcal{M}_\pi(X)^{U(1),U(1)} \sqcup\mathcal{M}_\pi(X)^{SU(2),SU(2)} ,
\end{equation}
and, if $n=2$, 
 \begin{equation}
\label{strat2}
 \mathcal{M}_\pi(X)=  \mathcal{M}_\pi(X)^{\ZZ/2,\ZZ/2}\sqcup\mathcal{M}_\pi(X)^{\ZZ/2,U(1)} \sqcup\mathcal{M}_\pi(X)^{U(1),U(1)}\sqcup\mathcal{M}_\pi(X)^{SU(2),SU(2)}. 
\end{equation}
Moreover, for such perturbations $\pi$, $\mathcal{M}_\pi(X)^{\ZZ/2,\ZZ/2}$ is a smooth manifold of dimension $3n-3$, 
$\mathcal{M}_\pi(X)^{U(1),U(1)} $ is a smooth manifold of dimension $n$,  $\mathcal{M}_\pi(X)^{SU(2),SU(2)}$ is a finite set, and, when $n=2$, 
$\mathcal{M}_\pi(X)^{\ZZ/2,U(1)}$ is a finite set.    $\mathcal{M}_\pi(X)$ is compact, as is $\mathcal{M}_\pi(X)^{\ZZ/2,U(1)} \sqcup\mathcal{M}_\pi(X)^{U(1),U(1)}\sqcup\mathcal{M}_\pi(X)^{SU(2),SU(2)}$,  and  $\mathcal{M}_\pi(X)^{U(1),U(1)}$ has a cone bundle neighborhood in   $\mathcal{M}_\pi(X)^{\ZZ/2,\ZZ/2} \cup  \mathcal{M}_\pi(X)^{U(1),U(1)} $.  

  The restriction map  
\begin{equation}\label{incluson}
j:\mathcal{M}_\pi(X)\to \mathcal{M}(F)
\end{equation}
Lagrangian  immerses $\mathcal{M}_\pi(X)^{\ZZ/2,\ZZ/2}$ into the   $(6n-6)$-dimensional symplectic manifold $\mathcal{M}(F)^{\ZZ/2}$. 

\qed
 
\end{thm}
 
\medskip

  Proposition \ref{lin} says that  the real-algebraic variety $R(S^2,\punctures)$ consists of two strata: a finite set $R(S^2,\punctures)^{U(1)}$  and a smooth $(4n-6)$-manifold $R(S^2,\punctures)^{\ZZ/2}$.  It follows from general results about algebraic varieties   that each singular point has a neighborhood homeomorphic to a cone on a $(4n-7)$-dimensional manifold \cite{King}. Multiplying a representation by a central character $\pi_1(S^2\setminus\punctures)\to \{\pm 1\}$ defines an action  on $R(S^2,\punctures)$  which is transitive on the singular points.  Hence there exists a   manifold $M(n)^{4n-7}$ such that each singular point $\rho\in R(S^2,\punctures)^{U(1)}$
has a neighborhood homeomorphic to $\text{cone}(M(n))$.  For example, when $n=2$, $R(S^2,\{a_1,b_1,a_2,b_2\})$ is a pillowcase and $M(2)=S^1$.

\medskip

The following proposition  is a generalization of Proposition 8.1 in our previous article \cite{HHK2}, which treated the case $n=2$ by a different method, passing to a branched cover and $SO(3)$ and making use of the fact that the only traceless representations near an abelian one are binary dihedral. There are non-binary dihedral representations near the abelian representations when $n>2$, and hence a different proof is required.    In the last part of the proposition, we use the term {\em radial embedding} to denote a map between cones that is an embedding away from the cone point, and sends the cone point to the cone point.

\begin{prop}\label{near singular} Let $(Y,L)$ be an $n$-tangle in a $\ZZ$-homology ball $Y$.   Given any disjoint union of embeddings $N_i:D^2\times S^1\to X, i=1,\dots,  p$, there exists a  neighborhood $\mathcal{U}_2$ of $0$ in $\mathcal{X}^p$ so that given any $(f_1,\dots,f_p)$ in $\mathcal{U}_2$, determining perturbation data $\pi=\{N_i, f_i\}_{i=1}^p$,  
\begin{enumerate} 
\item  the abelian stratum $R_\pi(Y,L)^{U(1),U(1)}$ is a finite set of $2^{n-1}$ points, and  
\item each point in $\rho \in R_\pi(Y,L)^{U(1),U(1)}$ has neighborhood in $R_\pi(Y,L)$    homeomorphic to a cone on $\CC P^{n-2}$ on which the restriction map $R_\pi(Y,L)\to R(S^2,\punctures)$ is a   radial embedding  
 $$\text{\rm cone}(\CC P^{n-2} )\to \text{\rm cone}(M(n))$$
 near the cone point.  
 \end{enumerate}
  The second claim implies, in particular,  that a small  deleted  neighborhood of $\rho$  in $R_\pi(Y,L)\setminus\{\rho\}$ contains only $(\ZZ/2,\ZZ/2)$ points.  In other words,  every traceless $\pi$-perturbed representation of $\pi_1(Y\setminus L)$ near $\rho$ is irreducible and restricts irreducibly to $\pi_1(S_0)$.  
 \end{prop}
\begin{proof} 
 Since $Y$ is a $\ZZ$-homology ball and $X=Y\setminus nbhd( L)$, $H_1(X)$ is isomorphic to $\ZZ^n$, generated by the meridians $A_1, \dots , A_n$.  Thus any $\rho\in R(Y,L)^{U(1),U(1)}$ may be uniquely conjugated so that $\rho(A_1)=\bbi$, and $\rho(A_i)=\epsilon_i \bbi$ for some  of signs $(\epsilon_2, \dots, \epsilon_n)$.   Hence there are $2^{n-1}$  points in  $\rho\in R(Y,L)^{U(1),U(1)}$.   The structure of $R(Y,L)$ near these points can be determined by computing the first and second cohomology, and using the Kuranishi model.  In order to show that the number of points and the local structure nearby persist for small perturbations, however, it is more convenient to first establish the correspondence between the points in $R(Y,L)^{U(1),U(1)}$ and $R_\pi(Y,L)^{U(1),U(1)}$, so we can show that the cohomology calculations work for small perturbation as well.

Fix  a collection of embeddings $N_i:D^2\times S^1\to X, ~i=1,\dots ,p$. Then each choice of a $p$-tuple $(f_1,\dots ,f_p)\in \mathcal{X}^p$ of perturbation functions determines the perturbation data $$\pi=\{(N_i, f_i)\}_{i=1}^p.$$

The group $H_1(X\setminus \sqcup_i N_i)$
is isomorphic to $\ZZ^{n+p}$, generated by the  $A_1, \dots , A_n$ and the $p$   
 meridians $\mu_i$ of the perturbation solid tori $N_i$.  
The longitude $\lambda_i$ of $N_i$ can be expressed in $H_1(X\setminus \sqcup_i N_i)$ as a linear combination
 $$\lambda_i=\sum_j a_{i,j} A_j + \sum_k b_{i,k} \mu_k$$  for some integers $a_{i,j}, b_{i,k}$.  
 
 We omit the proof of the following simple lemma, which describes the perturbation condition in the context of abelian representations.
\begin{lem}\label{U(1)reps}  Let perturbation data $\pi=\{(N_i, f_i)\}_{i=1}^p$ be given.
 Any  $\rho \in\mathcal{M}_\pi(X)^{U(1)}$ satisfying $\rho(A_i)\ne \pm 1, ~i=1,\dots,n$ may be conjugated so that 
 \begin{equation}\label{form1}\rho(A_i)=e^{\ep_i\alpha_i\bbi},~ i=1,\dots,n,~ \rho(\lambda_i)=e^{\theta_i\bbi},~ i=1,\dots,p,\end{equation}
where 
 \begin{equation}\label{form2}
 0<\alpha_i<\pi,  \ep_1=1, \ep_i\in \{\pm 1\},\text{~and~} e^{\theta_i \bbi}\text{~satisfy ~}  f_i\big( \sum_{j=1}^n a_{i,j}\ep_i \alpha_i + \sum_{k=1}^p b_{i,k} \theta_k \big) =\theta_i ~{\rm mod} ~2\pi.
\end{equation}
 Conversely,   any choice of $\alpha_i\in(0,\pi), \ep_i\in\{\pm 1\}, \ep_1=1,e^{\theta_i\bbi}\in U(1)$ satisfying {\rm (\ref{form2})} and   determines, via  {\rm (\ref{form1})}, a unique conjugacy class in $\mathcal{M}_\pi(X)^{U(1)}$ satisfying $\rho(A_i)\ne \pm 1$.
 Such representations are traceless, and hence $\rho\in R_\pi(Y,L)^{U(1)}$, exactly when $\alpha_i=\tfrac\pi 2,~ i=1,\dots, n$. \qed
\end{lem}

 Consider the   $C^k$  map
 $$C: \{\pm 1\}^{n-1}\times \mathcal{X}^p\times (0,\pi)^n\times U(1)^p\to U(1)^p$$
whose $i$th coordinate is given by
$$C_i(\ep, f,\alpha, \theta)= f_i\big(a_{i,1}\alpha_1  + \sum_{j=2}^n a_{i,j} \ep_j\alpha_j + \sum_{k=1}^p b_{i,k} \theta_k \big) -\theta_i ~\text{ mod }2\pi,$$  
where $\ep=(\ep_2,\dots,\ep_n),~ f=(f_1\dots, f_p),~\alpha=(\alpha_1,\dots,\alpha_n), ~$ and $\theta=(\theta_1,\dots, \theta_p)$.

Lemma \ref{U(1)reps} shows that  the fiber over $(1,\dots ,1)$ of the restriction    
\begin{equation}\label{les1}
C:  \{\pm 1\}^{n-1}\times \{(0, \dots,0)\}\times \{(\tfrac\pi 2,\dots,\tfrac\pi 2)\}\times U(1)^p\to U(1)^p
\end{equation}
  is identified with $R(Y,L)^{U(1),U(1)}$.  More generally, the fiber over $(1,\dots ,1)$ of the restriction  
\begin{equation}\label{les2}
C: \{\pm 1\}^{n-1}\times \{(f_1, \dots,f_p)\}\times \{(\tfrac\pi 2,\dots,\tfrac\pi 2)\}\times U(1)^p\to U(1)^p
\end{equation}  
is identified with $R_\pi(Y,L)^{U(1),U(1)}$, where $\pi$ is the perturbation associated to the $p$-tuple 
$(f_1,\dots,f_p)$. Finally, for a fixed $\alpha=(\alpha_1,\dots,\alpha_n)$, 
 the fiber over $(1,\dots ,1)$ of the restriction  
 \begin{equation}\label{les3}
 C: \{\pm 1\}\times \{(f_1, \dots,f_p)\}\times \{(\alpha_1,\dots,\alpha_n)\}^{n-1}\times U(1)^p\to U(1)^p 
\end{equation}  
 is identified with the subset of $\mathcal{M}_\pi(X)^{U(1)}$ consisting of those representations satisfying (\ref{form1}).

The restriction   (\ref{les1}) is obviously a trivial $(n-1)$-fold covering map.  Hence  so is the restriction  (\ref{les3}) whenever    
$(f_1,\dots,f_p)\in \mathcal{X}^p$ is sufficiently small  with respect to the $C^k$ metric on each factor and $(\alpha_1,\dots,\alpha_n)$  and each $\alpha_i$ is sufficiently close enough to $\frac\pi 2$.   Thus we can choose a neighborhood $\mathcal{U}_2$ of $0$ in $\mathcal{X}^p$ and a neighborhood $\mathcal{V}$ of $(\tfrac\pi 2,\dots,\tfrac\pi 2)$ in $(0,\pi)^n$ so that 
each choice $( \ep,f, \alpha)$ in $\{\pm 1\}^{n-1}\times\mathcal{U}_2\times \mathcal{V} $ determines a unique $\rho_{\ep,\pi,\alpha}\in\mathcal{M}_\pi(X)^{U(1)}$.  
When $\alpha=(\tfrac\pi 2,\dots,\tfrac\pi 2)$ we write  $$\rho_{\ep, \pi}\in R_\pi(Y,L)^{U(1),U(1)}.$$ In particular,  $R_\pi(Y,L)^{U(1),U(1)}$ contains $2^{n-1}$ points for each $(f_1,\dots,f_p)\in \mathcal{U}_2$, indexed by $\ep\in\{\pm 1\}^n$. Near each $\rho_{\ep,\alpha}$, $\mathcal{M}_\pi(X)^{U(1)}$ is homeomorphic to a neighborhood of $(\tfrac\pi 2,\dots,\tfrac\pi 2)$ in $(0,\pi)^n$, equivalently to an open set in $\RR^n$.

\medskip

We turn now to a description of  the local structure of $R_{\pi} (Y,L)$ and $\mathcal{M}_\pi(X)$ near $\rho_{\ep,\pi} $ for $\ep=(\ep_2,\dots, \ep_n)$ in $\{\pm 1\}^n$ and $\pi=\{N_i, f_i\}_{i=1}^p$ with $(f_1,\dots, f_p)$ in   $ \mathcal{U}_2$, using the Kuranishi model.  We begin by summarizing some results about the relevant cohomology groups.

 Since $\rho_{\ep, \pi}$ is reducible,   the conjugation action of $\rho_{\ep,\pi}  $ on $su(2)$ splits as 
  $su(2)_{ad~\rho_{\ep,\pi}}=\RR\oplus \CC_{\ep,\pi}$
and for any path connected subspace $Z$ of $X\setminus \sqcup_i N_i$,
\begin{equation}
\label{split11}H^i_\pi(Z;su(2)_{ad~\rho_{\ep,\pi}  })=H^i_\pi(Z;\RR)\oplus H^i_\pi(Z;\CC_{\ep,\pi} )
\end{equation}
(see Section \ref{tangentspaces} and  esp. Equation (\ref{cohosplit})).
We calculate $H^i_\pi(Z;su(2)_{ad~\rho_{\ep,\pi}  })$  for $Z=X, F,$ and $ S_0$ by first calculating the terms in the splitting (\ref{split11}) in the case when the perturbation is trivial, that is, for $\rho_{\ep,0}$.  We then use the upper semicontinuity property to show this remains stable for small perturbations $\pi$. The following lemma treats the unperturbed case (all $f_i=0$).

\begin{lem} \label{dimcomp}
\hfill \begin{enumerate} 
 \item $H^1(X;su(2)_{\ad{ \rho_{\ep,0}}})=H^1(X;\RR)\oplus H^1(X;\CC_{\ep,0})=\RR^n\oplus \CC^{n-1}$, 
 \item $H^2(X;su(2)_{\ad {\rho_{\ep,0}}})=H^2(X;\RR)\oplus H^2(X;\CC_{\ep,0})=0,$  
 \item $H^2(S_0;su(2)_{\ad{\rho_{\ep,0}}})= 0,$ $H^1(S_0;su(2)_{\ad{\rho_{\ep,0}}})=
 H^1(S_0;\RR)\oplus H^1(S_0;\CC_{\ep,\pi})=\RR^{2n-1}\oplus \CC^{2n-2}, $ and 
 \item the restriction map $H^1 (X; su(2) _{\ad{\rho_{\ep, 0}}} ) \to H^1 (S_0; su(2)_{\ad{\rho_{\ep, 0}}})$ is injective.  \end{enumerate} \end{lem}   
 \begin{proof} 
 
 Since $X$ is the complement of an $n$-tangle in  a $\ZZ$-homology 3-ball, $H^0(X;\RR)=\RR, ~H^1(X;\RR)=\RR^{n} , $ and $H^i(X;\RR)=0, $ for $i>1.$

The reducible representation $\rho_{\ep, 0}$ has image in the 4-element subgroup $\{\pm 1,\pm\bbi\}$ since $f_i=0$ implies that $\rho(\mu_i)=0$, and hence $ \rho_{\ep, 0}:H_1(X\setminus \sqcup_iN_i)\to U(1)\subset SU(2)$ factors through $H_1(X\setminus \sqcup_iN_i)\to H_1(X)$, spanned by the $A_i$, which are sent to $\pm \bbi$.

The $ad\rho_{\ep, 0}$ action on $\CC$ factors through a $\{\pm 1\}$ action. In fact,  $-1$ acts  trivially and $\pm \bbi$ as multiplication by $-1$ via the weight 2 representation on $\CC$. This gives an identification of cohomology groups
$$H^i(X;  \CC_{\ep, 0}  ) \cong H^i(X;  \CC_{\tau}  ) $$
where $\tau:\pi_1(X)\to \{\pm 1\}$ is the unique representation taking each $A_i$ to $-1$. 
   The group ring $\CC[\ZZ/2]$ splits as a $\ZZ/2=\{\pm 1\}$ module into the sum of $\CC_{\tau}$, spanned by $1-t$, and a trivial factor spanned by $1+t$, i.e., 
$\CC[\ZZ/2]=\CC_{\tau}\oplus\CC$.  Thus $$H^i(X;\CC[\ZZ/2])=H^i(X;\CC_{\tau})\oplus H^i(X;\CC),$$ and Shapiro's lemma \cite{brown} shows that $H^i(X;\CC[\ZZ/2])=H^i(X_2;\CC)$, where $X_2\to X$ is the induced 2-fold cover.

\medskip
  
Next, we claim that
$$H^0(X;\CC_{\tau})=0, ~H^1(X;\CC_{\tau})=\CC^{n-1}, ~H^i(X;\CC_{\tau})=0 ~i>1.$$
To see this, first recall  that $$H^0(X;\CC_{\tau})=\{z\in \CC~|~ (1-\tau(\gamma))z=0 \text{ for all }\gamma\in \pi_1(X)\},$$ which is  $0$ since $\tau(A_1)=-1$.
The 3-manifold with nonempty boundary $X$ has Euler characteristic $n-1$, and hence the assertion that $H^1(X;\CC_{\tau})=\CC^{n-1}$ follows once we show that $H^2(X;\CC_{\tau})=0.$ Since $H^2(X;\CC)=H^2(Z;\RR)\otimes\CC=0$, it follows that $H^2(X;\CC_{\tau})=H^2(X;\CC_{\tau})\oplus H^2(X;\CC)=H^2(X;\CC[\ZZ/2])=H^2(X_2;\CC)$.

Let $Y_2$ be the branched cover of $Y$, branched along $T$,  corresponding to the 2-fold cover $X_2\to X$. Consider the two groups adjacent to $H^2 (X_2; \CC)$  in the cohomology sequence for the pair $(Y_2,X_2)$: 
\begin{equation} 
\dots \to H^2(Y_2;\CC) \to H^2(X_2;\CC) \to H^3(Y_2,X_2;\CC)\to\dots \label{leq_pair} 
\end{equation}
 The excision isomorphism shows $H^3(Y_2,X_2;\CC)=H^3(\sqcup_{i=1}^n I\times D^2,I\times S^1;\CC)=0$.
The fact that $H^2(Y_2;\CC)=0$ follows from the fact that $Y_2$ is the 2-fold branched cover of a homology ball, and hence a $\ZZ/2$-homology ball.
In more detail, one can take a trivial $n$-tangle in a 3-ball, $(B^3,T)$, and attach it to $(Y,L)$ to obtain a knot $K=L\cup T$ in a homology sphere $\Sigma=Y\cup B^3$.  The 2-fold branched cover of $\Sigma$ branched over $K$ is a $\ZZ/2$-homology sphere (\cite{CG}), and, since the 2-fold branched over of $(B^3,T)$ is a handlebody, $Y_2$ is the complement of a handlebody in a $\ZZ/2$-homology sphere, hence $H^2(Y_2;\CC)=0$.  This shows that the exact sequence (\ref{leq_pair}) 
traps $H^2(X_2;\CC)$ between two vanishing groups. 
Thus $0=H^2(X_2;\CC)=H^2(X;\CC_{\tau}).$   By Euler characteristic considerations,  it follows that $H^1 (X; \CC_\tau)=\CC^{n-1}$, as asserted.  
Hence $$H^1(X;su(2)_{\ad{\rho_{\ep,0}}})=H^1(X;\RR)\oplus H^1(X;\CC_\tau)=\RR^n\oplus \CC^{n-1}$$ and
$$H^2(X;su(2)_{\ad{ \rho_{\ep,0}}})=H^2(X;\RR)\oplus H^2(X;\CC_\tau)=0\oplus 0=0.$$     This proves claims (i) and (ii).

Recall that $ F=\partial X$ is a surface of genus $n$ and $S_0\subset F$ is a $2n$-punctured 2-sphere.  
The Euler characteristic shows that 
$$H^0(F;\RR)=\RR=H^2(F;\RR), ~H^1(F;\RR)=\RR^{2n},
$$
$$H^0(F;\CC_\tau)=0=H^2(F;\CC_\tau), ~H^1(F;\CC_\tau)=\CC^{2n-2},
$$
and 
$$H^0(S_0;\RR)=\RR, ~H^1(S_0;\RR)=\RR^{2n-1},~H^2(S_0;\RR)=0,
$$
$$H^0(S_0;\CC_\tau)=0=H^2(S_0;\CC_\tau), ~H^1(S_0;\CC_\tau)=\CC^{2n-2}.
$$
 This proves claim (iii). 

Since $H_1(X)$ is freely generated by $A_1,\dots, A_n$ and $H_1(S_0)$ by $A_1,\dots, A_n, B_1,\dots, B_{n-1},$  the restriction
$$H^1(X;\RR)\to H^1(S_0;\RR)$$ is injective.
 The inclusion 
$$H^1(X;\CC_\tau)\to H^1(F;\CC_\tau)$$  is injective since, by the usual Poincar\'e duality argument, its image is a Lagrangian (in particular,  half-dimensional) subspace, and hence the kernel is trivial.  The inclusion 
$$H^1(F;\CC_\tau)\to H^1(S_0;\CC_\tau)$$
 is also injective; this follows from the long exact sequence for the pair $(F,S_0)$, excision, and the calculation $H^*(S^1\times I, S^1\times\{0,1\});\CC_\tau)=0$.  Hence the composite
  $H^1(X;\CC_\tau)\to H^1(F;\CC_\tau)\to H^1(S_0;\CC_\tau)$ is injective. 
 From the splitting of Equation (\ref{split11}) with $Z=S_0$, we conclude that
 \begin{equation}
\label{inj}H^1(X;su(2)_{\ad{\rho_{\ep,0}}})\to H^1(S_0;su(2)_{\ad{\rho_{\ep,0}}})
\end{equation}
 is injective, establishing claim (iv). \end{proof} 
 
   We now return to the proof of Proposition \ref{near singular}.  
We begin by analyzing  $H^i_\pi(X;su(2)_{\ad{\rho_{\ep,\pi}}})$  for small perturbations $\pi$. Since $\rho_{\ep,\pi}$ takes values in the diagonal $U(1)$ but is non-central, using Equation (\ref{h0pi}) one computes $$H^0_\pi(X;su(2)_{\ad{\rho_{\ep,\pi}}})=H^0_\pi(X;\RR)\oplus H^0_\pi(X;\CC_{\ep,\pi})=\RR\oplus 0=\RR.$$   

Next, $H^i_\pi(X;su(2)_{\ad {\rho_{\ep,\pi}}})=0$ for $i>3$, since $X$ is a 3-manifold. Also $H^3 (X;su(2)_{\ad {\rho_{\ep,0}}})=0$ since $X$ is homotopy equivalent to a 2-complex.    The upper semicontinuity property of $\dim H^3_\pi (X;su(2)_{\ad {\rho_{\ep,\pi}}})$ with respect to varying $\pi$ shows that after perhaps shrinking the neighborhood $\mathcal{U}_2$ of $0$ in $\mathcal{X}^p$, $H^3_\pi (X;su(2)_{\ad {\rho_{\ep,\pi}}})=0$ for all $\pi\in \mathcal{U}_2$.  

The upper semicontinuity property and Lemma \ref{dimcomp} shows that, after  shrinking the neighborhood $\mathcal{U}_2$ if needed, 
$H^2_\pi (X;su(2)_{\ad{ \rho_{\ep,\pi}}})=0$ for all $\pi\in \mathcal{U}_2$.  An Euler characteristic argument, Lemma \ref{dimcomp}, and Equation (\ref{split11}) now imply that for $\pi\in \mathcal{U}_2$ 
 \begin{equation}
\label{split12}
H^1_\pi (X;su(2)_{\ad {\rho_{\ep,\pi}}})=H^1_\pi (X;\RR)\oplus H^1_\pi (X;\CC_{\rho_{\ep,\pi}})=\RR^n\oplus \CC^{n-1}.
\end{equation}

 Now we utilize the identification of the perturbed character variety with the perturbed flat moduli space in gauge theory.  Since $H^2 _\pi (X; su(2)_{\ad{\rho_{\ep,\pi}}})=0$, the    Kuranishi model  for the perturbed flat moduli space (see \cite{herald1})  implies that 
 a neighborhood of $\rho_{\ep,\pi}$ in $\mathcal{M}_\pi(X)$ is identified  with the quotient of a slice manifold, parameterized by  a neighborhood of $0$ in $H^1(X;su(2)_{\ad{\rho_0}})\cong \RR^n\oplus \CC^{n-1}$,   by    the   action of the stabilizer    $U(1)=\Stab(\rho_{\ep, \pi}),$ 
which   acts trivially on $\RR^n$ and  with weight 2 on $\CC^{n-1}$.

    In terms of representations, this means the following.   Every $\rho\in \mathcal{M}_\pi(X)$ near $\rho_{\ep, \pi}$ can be conjugated so that $\rho:\pi_1(X\setminus\sqcup_iN_i)\to SU(2)$ satisfies $\rho(A_1)=\bbi e^{\bbi t}$ for some $t$ close to 0.
Choose   an invariant metric $d$  on  $\Hom(\pi_1(X\setminus \sqcup_i N_i), SU(2)) = SU(2)^ q $ and define
 \begin{equation}\label{OX}
\mathcal{O}_{\rho_{\ep, \pi}}(X)=\{\rho\in  \Hom(\pi_1(X\setminus \sqcup_i N_i),SU(2))~|~ \rho(A_1)\in U(1) , d(\rho,\rho_{\ep, \pi})<  \delta_X  , \rho(\mu_i)=F_i(\rho(\lambda_i)\}.
\end{equation}
  Each point in $\mathcal{O}_{\rho_{\ep, \pi}}(X)$ has stabilizer contained in $U(1)$, and the only $SU(2)$ elements mapping points in this subset back into this subset are $U(1)$ elements, so a neighborhood of $\rho_{\ep, \pi}$ in $\mathcal{M}_\pi(X)$ may be described by $\mathcal{O}_{\rho_{\ep, \pi}}(X)/U(1)$.    From the above Kuranishi argument,   we can choose $\delta_X$ sufficiently small so that $\mathcal{O}_{\rho_{\ep, \pi}}(X)$ is an open ball of dimension $3n-2$, and the tangent space $T_{\rho_{\ep,\pi}}\mathcal{O}_{\rho_{\ep, \pi}}(X)$ is canonically identified with $H^1_\pi(X;su(2)_{\ad{\rho_{\ep,\pi}}})$.  It follows that a  neighborhood of $\rho_{\ep, \pi}$ in $\mathcal{M}_\pi(X)$ is homeomorphic to $\RR^n\times \text{cone}(\CC P^{n-2})$.

We next focus our attention on the model for a neighborhood of the restriction of $\rho_{\ep, \pi}$ to $\pi_1(S_0)$.  Notice that this restriction  is independent of $\pi$, since   $\rho_{\ep,\pi}(B_i)=\rho_{\ep,\pi}(A_i)=\ep_i\bbi=\rho_{\ep,0}(A_i)$.  Denote this restriction by $\rho_\ep$.

Define 
$$\mathcal{O}_{\rho_\ep}(S_0)=\{\rho\in  \Hom(\pi_1(S_0),SU(2))~|~ \rho(A_1)\in U(1) , d(\rho,\rho_\ep)<\delta_{S_0} \},
$$
using an invariant metric $d$ on $\Hom(\pi_1 (S_0), SU(2))$.   Then, after perhaps adjusting the size of   $\delta_X$, we obtain a $U(1)$ equivariant restriction map  $\mathcal{O}_{\rho_{\ep,\pi}}(X)\to \mathcal{O}_{\rho_\ep}(S_0)$. 

 As is the case for  $\mathcal{O}_{\rho_{\ep, \pi}}(X)$,  $\mathcal{O}_{\rho_\ep}(S_0)$ gives a slice of the conjugation action of $SU(2)$ on $\Hom(\pi_1 (S_0), SU(2))$ near $\rho_\ep$.   The tangent space of $\mathcal{O}_{\rho_\ep}(S_0)$ at $\rho_\ep$ is canonically identified with $H^1 (S_0;su(2)_{\ad{\rho_{\ep,\pi}}})=\RR^{2n-1} \oplus \CC^{2n-2}$.  
 One can see these facts explicitly as follows.  
Recall that $\pi_1(S_0)$ is free on the $2n-1$ generators $A_1, A_2,B_2,\dots, A_{n}, B_{n}$.  Consider the map ${\bf r}:\RR^{2n-1} \oplus \CC^{2n-2}\to \Hom(\pi_1(S_0),SU(2))$ which sends
$(t_1,\dots, t_{2n-1}, z_2,\dots, z_{2n-1})$ to the unique homomorphism satisfying
 \begin{equation}
\label{mapr}A_1\mapsto \bbi e^{t_1\bbi}, \text{ and }A_i\mapsto \ep_i\bbi e^{t_i\bbi + z_i\bbj},~
B_i\mapsto \ep_i\bbi e^{t_{n-1+i}\bbi + z_{n-1+i}\bbj} \text{ for } i=2,\dots, n.
\end{equation}

 Then ${\bf r}$ sends a neighborhood of $0$ diffeomorphically and $U(1)$ equivariantly to a neighborhood of $\rho_\ep $ in $\mathcal{O}_{\rho_\ep}(S_0)$, since every representation near $\rho_\ep$ can be conjugated to this form, uniquely up to the $U(1)$ action.

 Using this  slice for the $SU(2)$ action on $\Hom(\pi_1(S_0), SU(2))$ we conclude  that  $\rho_\ep$ has  a neighborhood  in $\mathcal{M}(S_0)$ is homeomorphic to   $\RR^{2n-1}\times \text{cone}(\CC P^{2n-3})$.  
   Furthermore, we have arranged the two slices in such a way that the  restriction map sends one to the other     $U(1)$ equivariantly.    The restriction map $\mathcal{O}_{\rho_{\ep,\pi}}(X)\to \mathcal{O}_{\rho_\ep}(S_0)$ has injective differential when $\pi=0$, because 
$$H^1(X;su(2)_{\ad{\rho_{\ep,0}}})\to H^1(S_0;su(2)_{\ad{\rho_\ep}})$$ is injective, by 
Lemma \ref{dimcomp}.  

It follows that, by choosing $\delta_X$ smaller if needed, we may assume that $\mathcal{O}_{\rho_{\ep,\pi}}(X)\to \mathcal{O}_{\rho_\ep}(S_0)$ is an $U(1)$  {\em  equivariant embedding}.  In particular, we conclude that  every $\rho$ in $\mathcal{O}_{\rho_{\ep,\pi}}(X)$ on which $U(1)$ acts non-trivially  restricts to a point in $\mathcal{O}_{\rho_\ep}(S_0)$ on which $U(1)$ acts non-trivially.  In other words, all   $\rho\in \mathcal{M}_\pi(X)^{\ZZ/2}$ close enough to $\rho_{\ep,\pi}$  restrict  to irreducible representations of $\pi_1(S_0)$.

\medskip

  The restriction of the map
  $$T:\Hom(\pi_1(X),SU(2))\to \RR^n,~ T(\rho)=(\Real(\rho(A_1),\dots,
  \Real(\rho(A_n)))$$
  to $  \mathcal{O}_{\rho_{\ep,0}}(X)$ is submersive near $\rho_{\ep,0}$ since, defining 
  \begin{equation}
\label{mapm}
m:\RR^n\to \Hom(H_1(X),SU(2))\subset \Hom(\pi_1(X), SU(2)),~
m(t_1,\dots, t_n)(A_j)=\ep_i \bbi e^{t_j\bbi}, 
\end{equation}
  the 
 composite 
 $$\RR^n\xrightarrow{m} \Hom(H_1(X),SU(2))\xrightarrow{T}\RR^n$$  is the submersion $(t_1,\dots, t_n)\mapsto (\sin(t_1),\dots, \sin(t_n))$.
 Again shrinking $\mathcal{U}_2$ if necessary, we may assume that 
  \begin{equation}\label{submersive at ep}
\left. T  \right|_{\mathcal{O}_{\rho_{\ep,\pi}}(X)}   :\mathcal{O}_{\rho_{\ep,\pi}}(X)\to\RR^n,~ T(\rho)=(\Real(\rho(A_1),\dots,
  \Real(\rho(A_n))
\end{equation}
is submersive for all $\pi\in \mathcal{U}_2$. This map is clearly invariant with respect to the $U(1)$ action, and therefore its restriction to the fixed set of the $U(1)$ action is submersive.  The fixed set is locally equivariantly homeomorphic to the fixed set of the $U(1)$ action $H^1_\pi(X;su(2)_{\ad{\rho_{\ep,\pi}}})$, which is precisely the summand $H^1_\pi(X;\RR)=\RR^n$ in Equation (\ref{split12}).  

It follows that  $T^{-1}(0)\cap  \mathcal{O}_{\rho_{\ep,\pi}}(X)$ is a smooth $(2n-2)$-dimensional submanifold passing through $\rho_{\ep,\pi}$, invariant under the $U(1)$ action and transverse to 
the fixed point set of the $U(1)$ action.   Hence every point in $T^{-1}(0)\cap   \mathcal{O}_{\rho_{\ep,\pi}}(X)$ except $\rho_{\ep,\pi}$ has trivial stabilizer, and is therefore irreducible.   Moreover, by making $\delta_X$ even smaller if needed, we may assume that $T^{-1}(0)\cap  \mathcal{O}_{\rho_{\ep,\pi}}(X)$ is $U(1)$ equivariantly diffeomorphic to $\CC^{n-1}$ with the weight 2 action.   
 
  Since $T^{-1}(0)\cap \mathcal{O}_{\rho_{\ep,\pi}}(X)$ consists precisely of those traceless representations near $\rho_{\ep, \pi}$, gauge fixed so that $A_1$ is sent to $\bbi$, 
$$\big(T^{-1}(0)\cap  \mathcal{O}_{\rho_{\ep,\pi}}(X)\big)/ U(1)$$   
 is, on the one hand, homeomorphic to a neighborhood of $\rho_{\ep,\pi}$ in  $R_\pi(Y,L)$, and, on the other hand, 
 homeomorphic to ${\rm cone}(\CC P^{n-2})$, as asserted.

\medskip
 
  Similarly, the map $\widetilde T:\Hom(\pi_1(S_0),SU(2))\to \RR^{2n-1}$ given by $$\widetilde T(\rho)=(\Real(\rho(A_1)),\dots, \Real(\rho(A_n)),\Real(\rho(B_2)),\dots,\Real(\rho(B_{n})))$$  is $U(1)$ equivariant and submersive,  as is its restriction to the $(6n-5)$-ball $\mathcal{O}_{\rho_\ep}(S_0)$.   
One sees this by precomposing with the map   ${\bf r}$ defined  in Equation  (\ref{mapr})  (note that we have excluded $\Real(\rho(B_1))$).  

Thus the $U(1)$ equivariant restriction map
 \begin{equation}
\label{local1}T^{-1}(0)\cap  \mathcal{O}_{\rho_{\ep,\pi}}(X) \to \widetilde T^{-1}(0)\cap  \mathcal{O}_{\rho_\ep}(S_0)
\end{equation}
has injective differential, since the restriction  $H^1(X;\CC_{\ep,\pi})=\CC^{n-1}\to H^1(S_0;\CC_\ep)=\CC^{2n-2}$ is injective.  
Hence (after perhaps making the neighborhoods smaller)   the map (\ref{local1}) is an embedding which descends to a radial embedding
$${\rm cone}(\CC P^{n-2})\to {\rm cone}(\CC P^{2n-3}).$$
 
The $(4n-4)$-ball $\widetilde T^{-1}(0)\cap  \mathcal{O}_{\rho_\ep}(S_0)$ contains the subset of all  traceless representations of $\pi_1(S_0)$ near $\rho_\ep$ which send  $A_1$ to $\bbi$.
Indeed, this subset is the preimage of $0$ for  the map
 \begin{equation}
\label{Tb1} T_{B_1}:\widetilde T^{-1}(0)\cap  \mathcal{O}_{\rho_\ep}(S_0)\to \RR,~ T_{B_1}(\rho)=\Real(\rho(B_1)). 
\end{equation}

  As mentioned before, is known (and elementary to show) that the map (\ref{Tb1})  is   singular at $\rho_\ep$ and a submersion  at all other points of $T_{B_1} ^{-1}(0)$.  Denote by $\widetilde M(n)$ the link of this singularity,  obtained by intersecting $T_{B_1}^{-1}(0)$ with a small transverse $(4n-5)$-sphere in $\widetilde T^{-1}(0)\cap  \mathcal{O}_{\rho_\ep}(S_0)$.  The image of the embedding (\ref{local1}) lies in $T_{B_1}^{-1}(0)$,
and hence $\widetilde{M}(n)$ is non-empty, and therefore is a smooth manifold of dimension $4n-6$. 

Since the map (\ref{Tb1}) is $U(1)$ invariant, $U(1)$ acts on $T_{B_1}^{-1}(0)\cap \widetilde T^{-1}(0)\cap  \mathcal{O}_{\rho_\ep}(S_0)$, freely away from $\rho_\ep$,   with orbit space a neighborhood of $\rho_\ep$ in 
$R(S^2,\punctures)$.  In particular, this neighborhood is   a cone on the manifold $M(n)\subset \CC P^{2n-3}$ obtained as the orbit space $\widetilde{M}(n)$ by the  free $U(1)$ action. 

Since the image of the map (\ref{local1}) lies in $T_{B_1}^{-1}(0)$, we conclude that near $\rho_{\ep, \pi}$, the restriction 
$R_\pi(Y,L)\to R(S^2,\punctures)$ is a radial embedding
$${\rm cone}(\CC P^{n-2})\to {\rm cone}(M(n)),$$
completing the proof. 
  \end{proof}

 \section{The main result}\label{main} 
In this section, we combine the earlier transversality results in the paper to obtain a global description of the structure of $R_\pi (Y,L)$ for generic small perturbations $\pi$.

\begin{thm}\label{thm1}  Assume $Y$ is a $\ZZ$-homology 3-ball containing an $n$-stranded tangle, with $n\geq 2$.     
There exist arbitrarily small perturbations $\pi$ such that
  $R_\pi(Y,L)$  is compact and  is the union of two strata
$$  R_\pi(Y,L)= 
  R_\pi(Y,L)^{\ZZ/2,\ZZ/2}\sqcup R_\pi(Y,L)^{U(1),U(1)}, $$ 
with the following properties:  
\begin{itemize} 
\item 
 $  R_\pi(Y,L)^{\ZZ/2,\ZZ/2}$  a smooth manifold of dimension $2n-3$, and $R_\pi(Y,L)^{U(1),U(1)}$ a finite set of $2^{n-1}$ points.

\item The restriction map  $  R_\pi(Y,L) \to R(S^2,\punctures)$  takes  the 0-manifold $R_\pi(Y,L)^{U(1),U(1)}$  into the 0-manifold $R(S^2,\punctures)^{U(1)}$,  and Lagrangian immerses the $(2n-3)$-manifold $R_\pi(Y,L)^{\ZZ/2,\ZZ/2}$ in the symplectic $(4n-6)$-manifold $R(S^2,\punctures)^{\ZZ/2}$.

 \item Each point in $R_\pi(Y,L)^{U(1),U(1)}$ has a neighborhood in $R_\pi(Y,L)$ homeomorphic to a cone on $\CC P^{n-2}$. Each point in $R(S^2,\punctures)^{U(1),U(1)}$ has a neighborhood in $R(S^2,\punctures)$ homeomorphic to a cone on a $(4n-5)$ manifold $M(n)$. The restriction 
$R_\pi(Y,L)\to R(S^2,\punctures)$ is a radial embedding
 ${\rm cone}(\CC P^{n-2})\to {\rm cone}(M(n)).$ \end{itemize} 
 
\end{thm}

\begin{proof}  The fact that $R_\pi(Y,L)$ is compact for any perturbation data $\pi$ is  explained in Section \ref{tangentspaces}, following  Equation (\ref{compact}).

  Choose a collection of perturbation curves $\sqcup_{i=1}^p N_i$ in $X$ as in Theorem \ref{herald}. 
 Next,  choose a small collar neighborhood of $F=\partial X$   in $X$ which misses these $p$ perturbation curves  $\sqcup_{i=1}^p N_i$.
 Let ${\bf C}=(C_1,\dots, C_K)$    denote an ordering  of  the special perturbation curves  $\mathcal{C}$, as described in Section  \ref{abundancesec}. Use the collar variable  to push them to disjoint curves into this collar. As in  Section \ref{sec2},  this determines a family of isotopies $$ \Phi_{\bf C}:\RR^K\to \text{Homeo}(\mathcal{M}(F),\mathcal{M}(F)), $$ with $\Phi_{\bf C}(0)=\text{Id}$,
 associated to these perturbation curves and perturbation functions $t_i\sin(\alpha),$ where  ${\bf t}=(t_1, \dots, t_K) \in \RR^K$.
More explicitly,   ${\bf t}\in \RR^K$ determines perturbation data $$\pi_{\bf t}=\{C_i, t_i\sin(\alpha)\}_{i=1} ^K$$ so that for any $f\in  \mathcal{X}^p$ determining perturbation data $\pi=\{N_i,f_i\}$, 
$$ \mathcal{M}_{\pi}(X)\cong \mathcal{M}_{\pi\cup \pi_{\bf t}}(X)$$
by a homeomorphism  such that the two maps 
  $$
   \Phi_{\bf C}({\bf t})\circ j: \mathcal{M}_{\pi}(X)\to \mathcal{M}(F)
  \text{ and } j_{\bf t}:\mathcal{M}_{\pi\cup \pi_{\bf t}}(X)\to \mathcal{M}(F)  $$
coincide,   where both $j:\mathcal{M}_\pi(X) \to \mathcal {M}(F)$ and $j_t:\mathcal{M}_{\pi \cup \pi_t } (X) \to \mathcal{M}(F)$ represent restriction to the boundary.

  \medskip
  
 Let $\mathcal{U}\subset \mathcal{X}^p$ and $U\subset \RR^K$ be neighborhoods of zero such the conclusion of Theorem \ref{herald} holds for $\pi$ in a residual $\mathcal R\subset  \mathcal{U}$ and such that the conclusion of Proposition \ref{near singular} holds for $R_{\pi\cup \pi_{\bf t}}(Y,L)$, for any $(\pi, \pi_{\bf t})\in \mathcal{U}\times U$.  
  Assume, furthermore,  that $U$ lies in the intersection $U_1\cap  U_3$ of the neighborhoods given in Propositions \ref{abund1} and  \ref{abund3}.  
 Fix any $\pi \in \mathcal {R}$.

 All further perturbations required to complete the proof are of type $\pi_{\bf t}$. In particular, these additional perturbations change neither the homeomorphism type of $\mathcal{M}_\pi(X)$ nor, by Proposition \ref{stratpreserving}, the decompositions of Equations (\ref{strat1}) and (\ref{strat2}).

\bigskip

 The space $R_{\pi\cup \pi_{\bf t}}(Y,L)$ is a subspace of $\mathcal{M}_{\pi\cup \pi_{\bf t}}(X)$, namely, 
$$R_{\pi\cup \pi_{\bf t}}(Y,L)=(T\circ j_{\bf t})^{-1}(0)$$
where $T:\mathcal{M}(F)\to \RR^n$ is the map defined in Equation (\ref{Tmap}).     It will be convenient to define the map 
$$J :\mathcal{M}_\pi(X)\times U \to\mathcal{M}(F), ~J(\rho,{\bf t})= j_{\bf t}(\rho).$$  
\medskip 

The case when $n=2$ requires a special treatment since $\mathcal{M}_\pi(X)^{\ZZ/2,U(1)}$ need not be empty. 
\begin{lem}\label{n=2case} 
When $n=2$, there exists a residual subset    $V'\subset U$ such  that, for each  ${\bf t}$ in this subset, 
$$T^{-1}(0) \cap \left( \mathcal{M}_{\pi\cup \pi_{\bf t}}(X)^{\ZZ/2, U(1)} \right) =\emptyset .$$
\end{lem}
\begin{proof} The set $\mathcal{M}_{\pi\cup \pi_{\bf t}}(X)^{\ZZ/2, U(1)}$ is finite by Theorem \ref{herald}.

Since $U\subset U_3$, Proposition \ref{abund3} implies that 
 the restriction 
 $$\left. T\circ J\right| _{ \mathcal{M}_{\pi}(X)^{\ZZ/2, U(1)}\times U } :\mathcal{M}_{\pi}(X)^{\ZZ/2, U(1)}\times U\to \RR^2$$ has $0$ as a regular value.   
 Thus 
 $$\Sigma_1=(T\circ J)^{-1}(0)\cap (\mathcal{M}_{\pi}(X)^{\ZZ/2,U(1)}\times U)$$ is a smooth submanifold of $\mathcal{M}_{\pi}(X)^{\ZZ/2,U(1)}\times U$ of codimension $2$.  Hence the  set $V'\subset U$ of regular values of the projection map $\Sigma_1 \to U$ is a residual set, by Sard's Theorem.  But since $\Sigma_1$ has smaller dimension than $U$, the preimage of any regular value ${\bf t} \in V'$ in  $\Sigma_1$ is empty, so  $j_{\bf t}(\mathcal{M}_{\pi}(X)^{\ZZ/2,U(1)})$ misses $T^{-1}(0)$.  This is equivalent to saying that 
 $$T:\mathcal{M}_{\pi\cup \pi_{\bf t}}(X)^{\ZZ/2,U(1)}\to \RR^2$$
 misses zero. 
\end{proof}

For convenience, denote $$E=r^{-1}(R(S^2,\punctures)^{U(1)})\cap \mathcal{M}(F)^{\ZZ/2}.$$ Theorem  \ref{surj} shows that $E$ is a finite union of $(n-1)$-dimensional tori, namely, those orbits of the ${\bf T}$ action with non-trivial stabilizer.  

\begin{lem}\label{missE} For any $n\geq 2$, there exists a residual set $V''\subset U$ such that if ${\bf t}\in V''$, $\mathcal{M}_{\pi \cup \pi_{\bf t}} (X) ^{\ZZ/2, \ZZ/2}$ misses $E$ and so $R_{\pi\cup \pi_{\bf t}}(Y,L)^{\ZZ/2,U(1)}$ is empty.
\end{lem}
\begin{proof}
 By the transversality assumptions for the perturbation $\pi$, $\mathcal{M}_{\pi}(X)^{\ZZ/2,\ZZ/2}$ is  a smooth $3n-3$ manifold.
Recall that  $\mathcal{M}(F)^{\ZZ/2}$  is a smooth manifold of dimension $6n-6$, so $E$ has codimension $5n-5$. 

Since,  by assumption, $U$ is contained in the neighborhood $U_1$ provided in Proposition \ref{abund1}, we have that   
$$\left. \Phi_C\right|_{E\times U} : E\times U \to \mathcal{M}(F)^{\ZZ/2}$$ 
is a submersion.  In particular, it is transverse to $E$.   
Thus the preimage 
  $$\Sigma_2= J^{-1}(E)\cap (\mathcal{M}_{\pi}(X)^{\ZZ/2,\ZZ/2}\times U)$$
  is a submanifold of codimension $5n-5$, and hence of dimension $K+3n-3-(5n-5)=K-2n+2<K$.       The set $V'' $ of regular values for the projection $\Sigma_2\subset \mathcal{M}_{\pi}(X)^{\ZZ/2,\ZZ/2}\times U\to U$ is residual,  by Sard's theorem.    Since the dimension of $\Sigma_2 $ is less than that of $U$, the preimage in $\Sigma_1$ of any  regular value ${\bf t}\in V''$ is empty.  This means that   $ j_{\bf t}(\mathcal{M}_\pi(X)^{\ZZ/2, \ZZ/2})$ misses $E$, or, equivalently, that $j_0(\mathcal{M} _{\pi \cup \pi_{\bf t}} (X) ^{\ZZ/2, \ZZ/2})$ misses $E$.   It follows that, for ${\bf t}\in V''$, $R_{\pi \cup \pi_{\bf t} } (Y,L)^{\ZZ/2, U(1)} $ is empty. \end{proof}

To complete the proof of Theorem \ref{thm1}, we now fix a ${\bf t}_1 \in V'\cap V''$, or, if $n>2$, simply ${\bf t}_1  \in V''$.

\begin{lem}\label{yellowsub}  There exists a neighborhood $U_2\subset \RR^K$ of ${\bf t}_1 $ and a residual subset $V_2\subset U_2$ such that, for ${\bf t}\in V_2$, the conclusion of Theorem \ref{thm1} applies to $R_{\pi \cup \pi_{\bf t}} (Y,L).$    

\end{lem}  
\begin{proof}
From the assumptions on $\mathcal U$ and $U$  at the beginning of the proof of Theorem \ref{thm1},  
$\mathcal{M}_{\pi \cup \pi_{{\bf t}_1}}(X)^{U(1),U(1)} $ consists of $2^{n-1}$  points, each with a neighborhood in $\mathcal{M}_{\pi \cup \pi_{{\bf t}_1}}(X)$ which is a quotient of a slice    $\mathcal{O}_{\rho_{\pi\cup \pi_{{\bf t}_1} ,\ep}}$.  
 
For each $\ep \in \{ \pm 1 \}^{n-1} $, let $\phi_\ep: \RR^n \oplus \CC^{n-1} \to \mathcal{O}_{\rho_{\pi \cup \pi_{{\bf t}_1 } }, \ep}$ be a $U(1)$ equivariant parameterization.  For ease of notation, we define $\mathcal{O} =\sqcup_{\ep} \mathcal{O}_{\rho_{\pi \cup \pi_{{\bf t}_1 } }, \ep}$ and $W_{\mathcal{O}}=\mathcal{O}/U(1)$.  

Let $$\mathcal{P}:\{ \pm 1\} ^{n-1} \times \left( \RR^n  \oplus \CC^{n-1} \right) \times \RR^K \to \RR^n$$ by 
$$\mathcal{P}\left( \ep, ({\bf v}, {\bf z}), {\bf t} \right) = T\circ j_{\bf t} \left( \phi_\ep ({\bf v}, {\bf z}) \right).$$  Because $(\pi, \pi_{{\bf t}_t} ) \in \mathcal{U}\times U_2$, the differential $d\mathcal{P}_{(\ep, (0,0), {\bf t}_1 )} $ maps the $\RR^n\oplus \{0\} \subset \RR^n \oplus \CC^{n-1}$ onto $\RR^n$.  Since surjectivity of the differential is an open condition, we can find neighborhoods  $U_0 \subset \RR^n\oplus \CC^{n-1}$ of $0$ and $U_2 '\subset \RR^K$ of ${\bf t}_1$ so that $\mathcal{P}$ restricted to $\{\pm 1\}^{n-1} \times U_0 \times U_2 '$ is a submersion.  Rather than complicate the notation with restriction to $U_0$, change the parameterizations $\phi_\ep$ by precomposition with a diffeomorphism from $\RR^n \oplus \CC^{n-1}$ onto a small open ball, so that we can assume $\mathcal{P}$ is a submersion on $\{\pm 1\}^{n-1} \times (\RR^n\oplus \CC^{n-1} )  \times U_2 '$, and $\mathcal{O}$ is the union of the images after these changes.  

Let $\mathcal{K}$ denote the compact set 
$$\mathcal{K} = \left( (T\circ j_0)^{-1} (0) \cap \mathcal{M}_{\pi \cup \pi_{{\bf t}_1}} (X) \right) \setminus W_{\mathcal{O}}.$$  Proposition \ref{abund2} implies that the map 
$$
\mathcal{M}_{\pi \cup \pi_{{\bf t}_1}} (X) ^{\ZZ/2, \ZZ/2} \times \RR^K \to \RR^n$$ obtained by restricting $T\circ J$ is a submersion along $\mathcal{K} \times \{ {\bf t}_1 \}$.  Therefore there is a neighborhood $W_{\mathcal{K}} \times U_2 ''$ on which this map is a submersion.  Set $U_2=U_2 ' \cap U_2 ''$.  Then 
$$\Sigma_2 ' = (T\circ J)^{-1} (0) \cap W_{\mathcal K} \times U_2 \mbox{ and } \Sigma_2 '' = 
\mathcal{P}^{-1} (0) \cap \left( \{ \pm 1\}^{n-1} \times (\RR^n \oplus \CC^{n-1}) \times  U_2\right) .   $$ 
 are submanifolds of dimensions  $2n-3+K$ and $2n-2+K$,    respectively, and $\Sigma_2 ''$ is $U(1)$ invariant, where we view the $U(1)$ action on the second factor as trivial.

Let $V_2 '\subset U_2$ be the set of regular values for the projection $\Sigma_2 '\to U_2$, and $V_2 '' \subset U_2$ the set of regular values for the projection $\Sigma_2 ''\to U_2$.  The intersection $V_2=V_2' \cap V_2 ''$ is a residual subset of $U_2$.  For ${\bf t}\in V_2$, 
$$
(T\circ j_{\bf t} )^{-1} (0) \cap \mathcal{M}_{\pi \cup \pi_{\bf t}} (X) \cap \left( W_{\mathcal{O} } \cup W_{\mathcal{K}} \right) $$ is a manifold, except at the cone points  where the $U(1)$ action on $\mathcal{O}$ creates $\operatorname{cone}({\CC P^{n-2} })$ singularities in the quotient.  

Next, we claim that by shrinking $U_2$, we can insure that whenever ${\bf t}\in U_2$, $((T\circ j_{\bf t} )^{-1} (0) \cap \mathcal{M}_{\pi \cup \pi_{\bf t}} (X)) \subset W_{\mathcal{K}} \cup W_{\mathcal{O}}. $  To see this, suppose that there exists a sequence ${\bf s}_k$  in $\RR^K$ converging to ${\bf t}_1$, and $\rho_k \in \mathcal{M}_{\pi}(X) \not\in  W_{\mathcal{K}} \cup W_{\mathcal{O}} $ such that 
$T \left( j_{{\bf s}_k} (\rho_k)\right) =0.$  By compactness of $\mathcal{M}_{\pi}(X)$, we can choose replace the sequence by a subsequence (which we reindex to avoid multiple subscripts) such that $\rho_k \to \rho_\infty \in \mathcal{M}_{\pi} (X) $.  By continuity, $0=\lim_{k\to \infty} T\left( \Phi_C (j(\rho_k) , {\bf s}_k)\right) =T\left( \Phi_C(\rho_\infty, {\bf t}_1 ) \right) = 0$  so $\rho_\infty \in (T\circ j_{t_1} ) ^{-1}(0) \cap \mathcal{M}_{\pi \cap \pi_{{\bf t}_1}} (X) $.  But this contracts the fact that none of the $\rho_k$ are in the neighborhood $W_{\mathcal{K}} \cup W_{\mathcal{O}} $ of this limit point. Therefore, there cannot be such a sequence $({\bf s}_k, \rho_k)$, which proves our claim about shrinking $U_2$. \end{proof}

  The only  assertion in Theorem \ref{thm1} remaining to be proven   is that $R_{\pi\cup \pi_{\bf t}}(Y,L)^{\ZZ/2,\ZZ/2}\to R(S^2,\{a_i,b_i\}_{i=1}^n)^{\ZZ/2}$ is a Lagrangian immersion.
  
  Theorem \ref{herald} asserts that $j_0:\mathcal{M}_{\pi\cup \pi_{\bf t}}(X)^{\ZZ/2,\ZZ/2}\to \mathcal{F}^{\ZZ/2}$ is a Lagrangian immersion, as is its further restriction to the neighborhood $p(\mathcal{N})$ of $R_{\pi\cup {\bf t}}(Y,L)$.  Since $T$ and the moment map $\mu$ have the same level sets, 
\begin{equation}
\label{op1}j:p(\mathcal{N})^{\ZZ/2,\ZZ/2}\to \mathcal{M}(F)^{\ZZ/2}
\end{equation}
 is a Lagrangian immersion transverse to the level set $\mu^{-1}(0)$ of the moment map for the torus action, and therefore the composite 
$$p(\mathcal{N})^{\ZZ/2,\ZZ/2}\cap (\mu\circ j)^{-1}(0)=R_\pi(Y,L)^{\ZZ/2,\ZZ/2}\to R(S^2,\{a_i,b_i\}_{i=1}^n)^{\ZZ/2}$$ is the symplectic reduction of the Lagrangian immersion (\ref{op1}) with respect to $\mu$ and the torus action, and hence is again a  Lagrangian immersion by Corollary \ref{lagreg}.\end{proof}

\section{Adding an earring}\label{piercedear}

As before, let $L=L_1\cup\dots\cup L_n$ be an $n$-tangle in a $\ZZ$-homology ball $Y$. Following \cite{KM2}, one can associate a variant $R^\nat_\pi(Y,L)$ of $R_\pi(Y,L)$ as follows  (for more details see \cite[Section 4.3]{HHK1}). Let $D\subset Y$ be a small normal disk to the $n$th arc $L_n$ in the interior of $Y$,  obtained by pushing disk neighborhood of the endpoint $a_{n}$,   in $\partial Y$ slightly into the interior of $Y$.   Let $H=\partial D$ (we call $H$ an {\em earring}) and let $W$ be an arc on $D$ joining $L_{n}$ to $H$.  Let $w:S^1\subset Y\setminus (L\cup H\cup W)$  be a small meridian circle to $W$.  The notation is illustrated in Figure \ref{earringfig}. One can think of $H$ as a slight push in of $A_n$ and $W$ as a push in of an arc in $\partial Y$ from $a_n$ to $A_n$.

     \begin{figure}[h]
\begin{center}
\def\svgwidth{2.7in}
 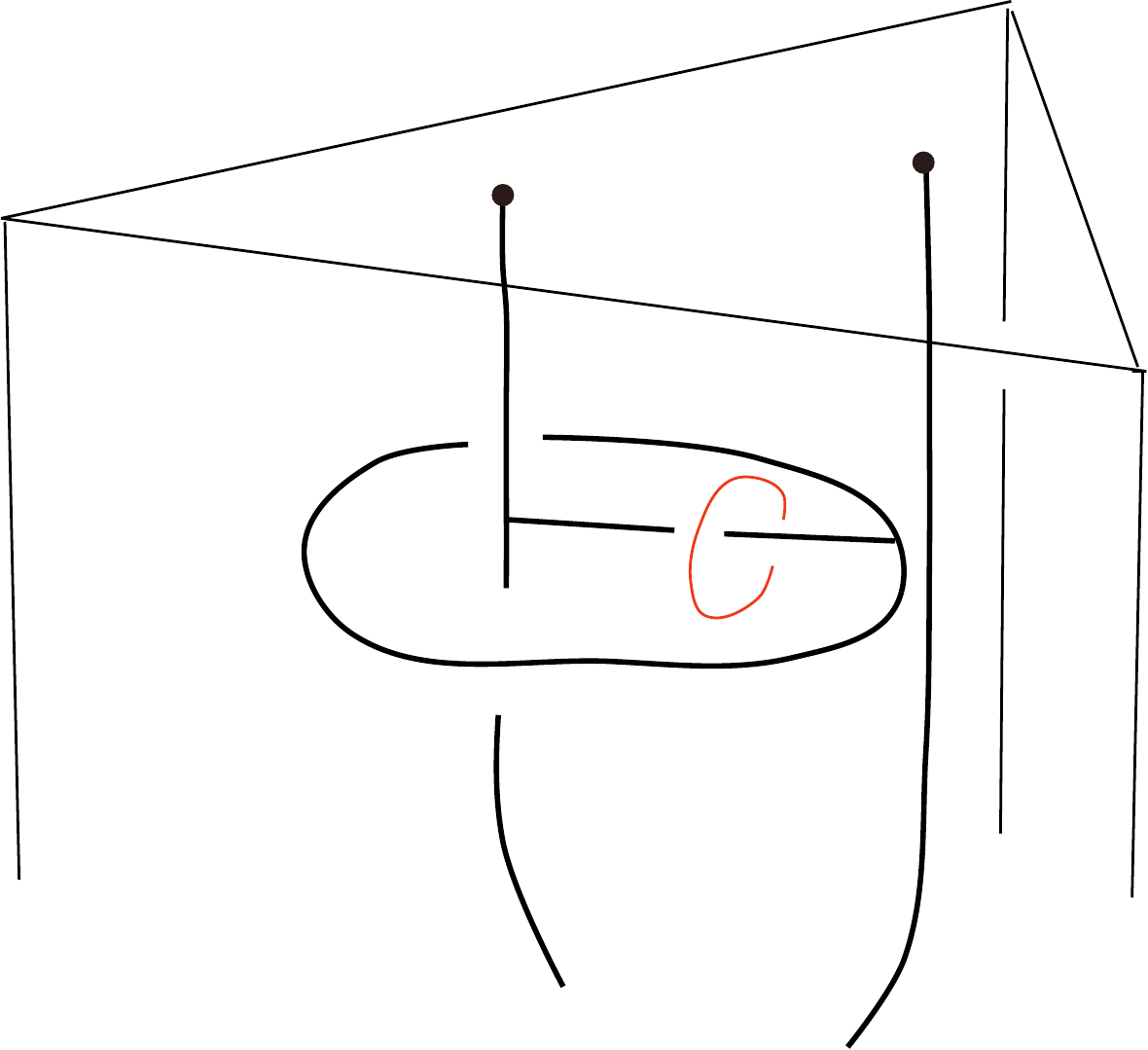
 \caption{\label{earringfig} }
\end{center}
\end{figure}

One then defines 
$R^\nat(Y,L)$ to be the space of conjugacy classes of representations
$$\rho:\pi_1(Y\setminus(L\cup H\cup W))\to SU(2)$$ satisfying
\begin{equation}
\label{ear}
\Real(\rho(m_{L_i}))=0, ~~ \Real(\rho(m_H))=0,  ~~\rho(w)=-1.
\end{equation}
More generally, given a perturbation $\pi=\{N_i,f_i\}$,  we can assume the earring misses the solid tori $N_i$ and define $R_\pi^\nat(Y,L)$  to be those $SU(2)$ representations of $\pi_1(Y\setminus(L\cup H\cup W\cup_iN_i))\to SU(2)$ satisfying the perturbation conditions (\ref{pert}) in addition to (\ref{ear}).

\begin{thm}\label{earthm} Assume $Y$ is a $\ZZ$-homology ball containing an $n$-strand tangle $L$, with $n\geq 2$.   There exist arbitrarily small perturbations $\pi$ so that  the space $R_\pi^\nat(Y,L)$ is a closed manifold and the restriction map 
$R_\pi^\nat(Y,L)\to R(S^2,\punctures)$ is an immersion into the smooth stratum $R(S^2,\punctures)^{\ZZ/2}$, and  hence misses the singular points.
 
\end{thm}

\begin{proof}
We will make use of the main result of \cite{HHK1}, Theorem 7.1, which  treats the case when $(Y,L)$ the trivial 2-tangle in the 3-ball.  Figure \ref{earring3fig} illustrates a tangle $U$ in a 3-ball $B$ with four strands, an earring $H\cup W$, and a perturbation curve $N_1$ with meridian labelled $p$ is indicated. Seven based loops in $B\setminus(U\cup H\cup W\cup_iN_i)$, $a,b,c,d,h,p,w$ are indicated.   Pick an $\ep>0$ as small as desired, and let $\pi_1$ denote the perturbation data $(N_1, \ep\sin(x))$.
 \begin{figure}[h]
\begin{center}
\def\svgwidth{4.3in}
 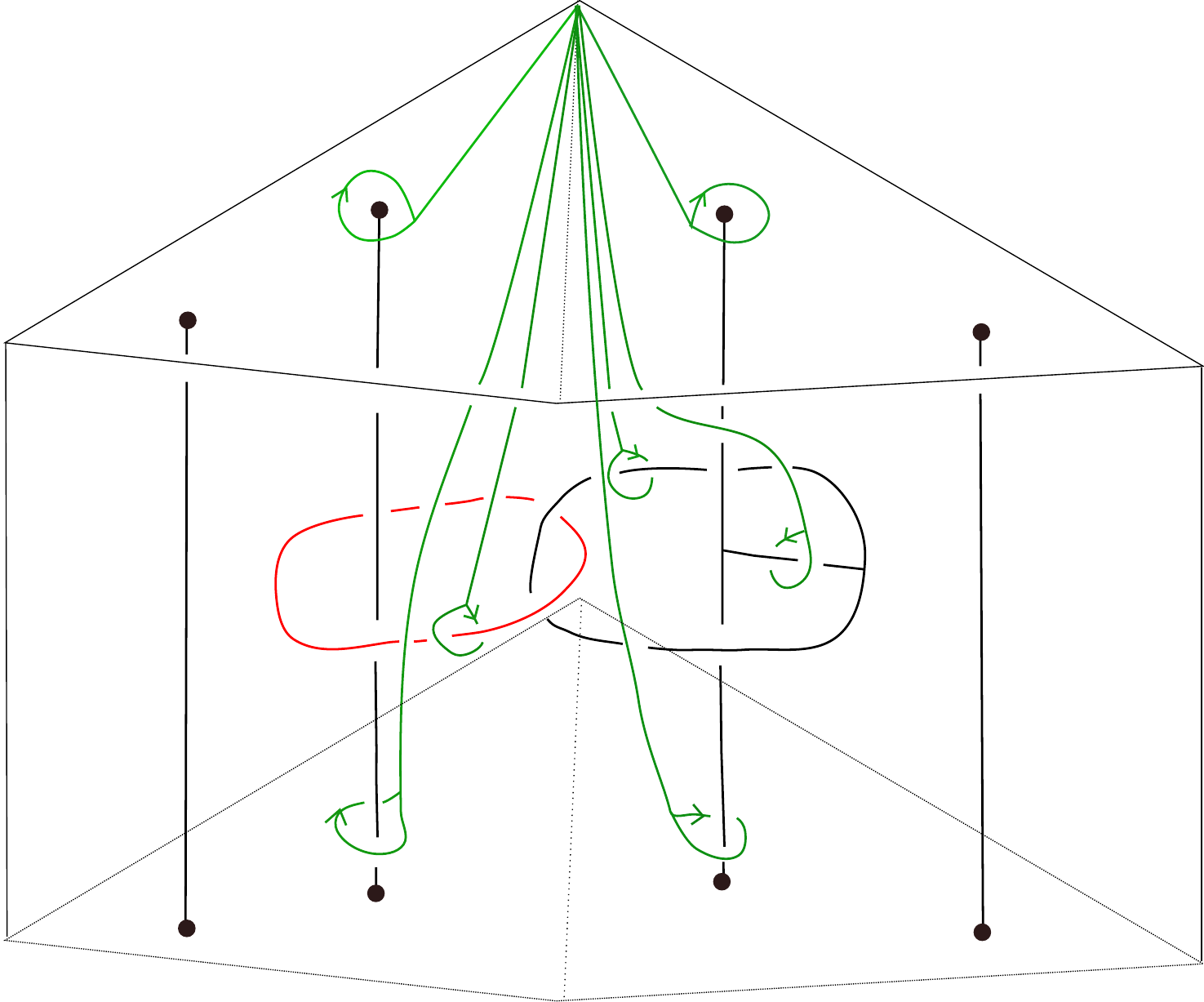
 \caption{\label{earring3fig} }
\end{center}
\end{figure} 

Denote by $(B,U)^\nat$ the tangle, with all the data as illustrated, of Figure \ref{earring3fig}.  Consider the new tangle
$$(Y',L')=(Y,L)\cup_{S_{n-1}\cup S_n}(B,U)^\nat$$
where $S_{n-1}$ and $S_n$ denote the $n$th and $(n+1)$st sectors, as defined in Section \ref{snakesonaplane} and Figure \ref{drawingfig}.
Ignoring the earring and perturbation curve, $(Y',L')$ and $(Y, L)$ are homeomorphic tangles. 

Denote by $S_0'$ the $2n$-punctured sphere in the boundary of $(Y',L')$.  Denote the corresponding generators  of $\pi_1(S_0')$ (analogously to Figure \ref{drawingfig}) by $A_1',B_1', \dots, A_n', B_n'.$  Then (using a vertical arc to connect the base points)
 \begin{equation}\label{eq11.1}
A_1'=A_1, ~B_1'=B_1,\dots, ~A_{n-2}'=A_{n-2}, ~B_{n-2}'=B_{n-2}, \mbox{ and } B_n=B_n'.
\end{equation}
With our labeling of the $A_i$ curves in the sectors,  $$A_{n-1}'=b^{-1}\text{ and } A_n'=a^{-1}.$$
Moreover, it is elementary to check that 
 \begin{equation}\label{eq11.2}
A_{n-1}=(pb)^{-1} A_{n-1} ' pb,
~ B_{n-1}=[b^{-1},p^{-1}] B_{n-1}'[b^{-1},p^{-1}]^{-1},~ A_n=[b^{-1},p^{-1}]A_n'.
\end{equation}

Let $\pi$ denote any perturbation data for $(Y,L)$ satisfying the conclusion of Theorem \ref{thm1}. The union $\pi\cup \pi_1$ is perturbation data for $(Y',L')$  so that the only perturbation curve in $(B,U)$ is $N_1$.

Theorem 7.1 of \cite{HHK1}   implies that given any $\rho\in R_{\pi\cup \pi_1}^\nat(Y',L')$, there exists a $\beta\in [0,2\pi)$,  so that, setting $\nu=\ep\sin(\beta)$ for convenience, $\rho$ has a unique representative in its conjugacy class which takes the form:
 \begin{equation}
\label{values}
\rho: a\mapsto \bbi, ~b\mapsto e^{(\beta+\nu)\bbk}\bbj,~ c\mapsto e^{(\beta-\nu)\bbk}\bbj,~d\mapsto e^{-2\nu\bbk}\bbi,~ h\mapsto-\bbj e^{-\nu\bbk},~p\mapsto e^{\nu\bbk}, w\mapsto-1.
\end{equation}
This immediately implies that $R^\nat_{\pi\cup\pi_1}(Y',L')^{U(1)}$ is empty for any $\pi$ as above.

Suppose that $\rho\in R^\nat_{\pi\cup\pi_1}(Y',L')^{\ZZ/2,U(1)}$. Hence $\rho(A'_j)=\pm \bbi$ and $\rho(B_j')=\pm \bbi$ for all $j$.   In particular
$\beta+\nu=\pm \frac\pi 2$. Since $\nu=\ep\sin \beta$, one can check that there are two such $\beta$ in $S^1$, one near each of $\pm \frac \pi 2$.

Then $\rho(A_j)=\pm \bbi$ and $\rho(B_j)=\pm \bbi$ for all $j<n-1$ and,  restricting $\rho$ to $(Y,L)$,  Equations (\ref{eq11.1}), (\ref{eq11.2}),  and (\ref{values}) imply that 
\begin{equation}
\label{eq11.3}
\rho:A_{n-1}\mapsto \pm e^{-2\nu\bbk}\bbi,~ 
B_{n-1}\mapsto \pm e^{-4\nu\bbk}\bbi,~ 
A_{n}\mapsto \pm e^{-2\nu\bbk}\bbi,~ 
B_{n}\mapsto \pm  \bbi,~ A_j,B_j\mapsto \pm \bbi, j<n-2
\end{equation}
The set $\mathcal{F}$ of all $\rho\in R(S^2,\punctures)$ satisfying (\ref{eq11.3}) is finite (in fact contains $2^{2n}$ points).  Note that $\mathcal{F}\subset R(S^2,\punctures)^{\ZZ/2}$.

\begin{prop}\label{noreducibles} There exists a residual set of perturbations $\pi_2$ in a neighborhood of the trivial perturbation, supported in a collar of $S_0$ in $(Y,L)$, so that the restriction to the boundary $R_{\pi\cup \pi_2}(Y,L)^{\ZZ/2} \to R(S^2,\punctures)^{\ZZ/2}$ misses $\mathcal{F}$. Hence $R^\nat_{\pi\cup\pi_1\cup \pi_2}(Y',L')^{\ZZ/2, U(1)}$ is empty and so $R^\nat_{\pi\cup\pi_1\cup \pi_2}(Y',L')=R^\nat_{\pi\cup\pi_1\cup \pi_2}(Y',L')^{\ZZ/2, \ZZ/2}$.
\end{prop}
\noindent{\em Sketch of proof.}  We spare the reader the description of a collection ${\bf C}$ of $K$ perturbation curves in $S_0$ and unenlightening calculations    required to show that there is a family of isotopies 
$\Phi_{\bf C}:\RR^K\times  R(S^2,\punctures)\to R(S^2,\punctures)$
so that for each $\rho\in \mathcal{F}$, the map $\Phi_{\bf C}(\rho,-):\RR^K\to R(S^2,\punctures)$ is a submersion. (This can be gleaned easily from Theorem 9.2 of \cite{HHK2} when $n=2$.  The general case uses the same approach as the proof of Proposition \ref{abund1}.)

Thus the (by now familiar) map $R_{\pi}(Y,T)^{\ZZ/2}\times \RR^K\to  R(S^2,\punctures)^{\ZZ/2}$ is transverse to $\mathcal{F}$. Since the dimensions of $R_{\pi}(Y,T)^{\ZZ/2}$ and $R(S^2,\punctures)^{\ZZ/2}$  are  $2n-3$ and $4n-6$, respectively, Sard's theorem gives a residual set of ${\bf t}\in \RR^k$  near $0$ so that $R_{\pi}(Y,T)^{\ZZ/2}\times 
\{{\bf t }\}\to  R(S^2,\punctures)^{\ZZ/2}$ misses $\mathcal{F}$. Let $\pi_2$ denote the  perturbation corresponding to such a ${\bf t}$.
\qed
\medskip

 Choose $\pi_2$ as in Proposition \ref{noreducibles}.
To simplify the notation, we henceforth use $\pi$ to again denote $\pi\cup \pi_2$. Thus, in the new notation,  
$R^\nat_{\pi_1\cup \pi}(Y',L')=R^\nat_{\pi\cup \pi_1}(Y',L')^{\ZZ/2,\ZZ/2}$, and
 $R^\nat_{\pi_1\cup \pi}(Y',L')$ is compact.
It remains to show that $R^\nat_{\pi\cup \pi_1}(Y',L')$ can be further perturbed to make it a smooth $2n-3$-manifold which Lagrangian immerses into $R((S^2)' ,\punctures)$.

\medskip

Denote by  $X'$ the complement 
$$X'=Y'\setminus (L'\cup H\cup W)=\big( (Y\setminus L)\cup_{D^2\setminus
\{a_{n-1}, b_{n-1}, a_n , b_n\}}(B\setminus (U\cup H\cup W))\big).$$ Thus $X'$ has boundary a surface $F'$ of genus $n+1$ and 
$F'$ contains the $2n$-punctured sphere $S_0$.  The perturbation $\pi$ equals $\{N_i, f_i\}_{i=2}^\ell$ for some solid tori $N_i$ in $Y\setminus L$ and $\pi_1=(N_1, \ep\sin(x))$ for $N_1\subset B\setminus  (U\cup H\cup W)$.

The relation $[ap^{-1},h]=(ha)w(ha)^{-1}$ holds in $\pi_1(X'\setminus(\cup_iN_i))$.  Hence a homomorphism $\rho:\pi_1(X'\setminus(\cup_iN_i))\to SU(2)$ satisfies $\rho(w)=-1$ if and only if $\rho([ap^{-1},h])=-1$. Since a pair $x,y$ of unit quaternions satisfy $[x,y]=-1$ if and only if $\Real(x)=\Real(y)=\Real(xy)=0$, one concludes that $\rho([ap^{-1},h])=-1$ if and only if 
$$\Real(\rho(h))=\Real(\rho(ap^{-1}))=\Real(\rho(ap^{-1}h))=0.$$
Figure \ref{tube1fig} shows that the three curves $ap^{-1}$, $ap^{-1}h$, and $h$ can be homotoped rel base point in $X'$ to simple closed curves which lie on $F'$. We denote the curves $ap^{-1}$ and $ap^{-1}h$ on $F'$ by $Y_1$ and $Y_2$, respectively.

 Define a map
$$S:\mathcal{M}(F')^{\ZZ/2}\to \RR^{n+3}$$
by 
\begin{equation}
\label{defS}
S(\rho)=\big(\Real(\rho(A_1)),\dots,\Real(\rho(A_n)),\Real(\rho(h)),\Real(\rho(ap^{-1})),\Real(\rho(ap^{-1}h))\big).
\end{equation}
 Then  $S^{-1}(0)$ is compact, since it is closed in the compact space $\mathcal{M}(F')$, and it lies in $\mathcal{M}(F')^{\ZZ/2}.$  As before, let $j:\mathcal{M}_{\pi_1\cup\pi}(X')\to \mathcal{M}(F')$ denote the restriction map. Then it follows from the definitions that $$R^\nat_{\pi_1\cup\pi}(Y',L')=(S\circ j)^{-1}({\bf 0}).$$

   \begin{figure}[h]
\begin{center}
\def\svgwidth{6in}
 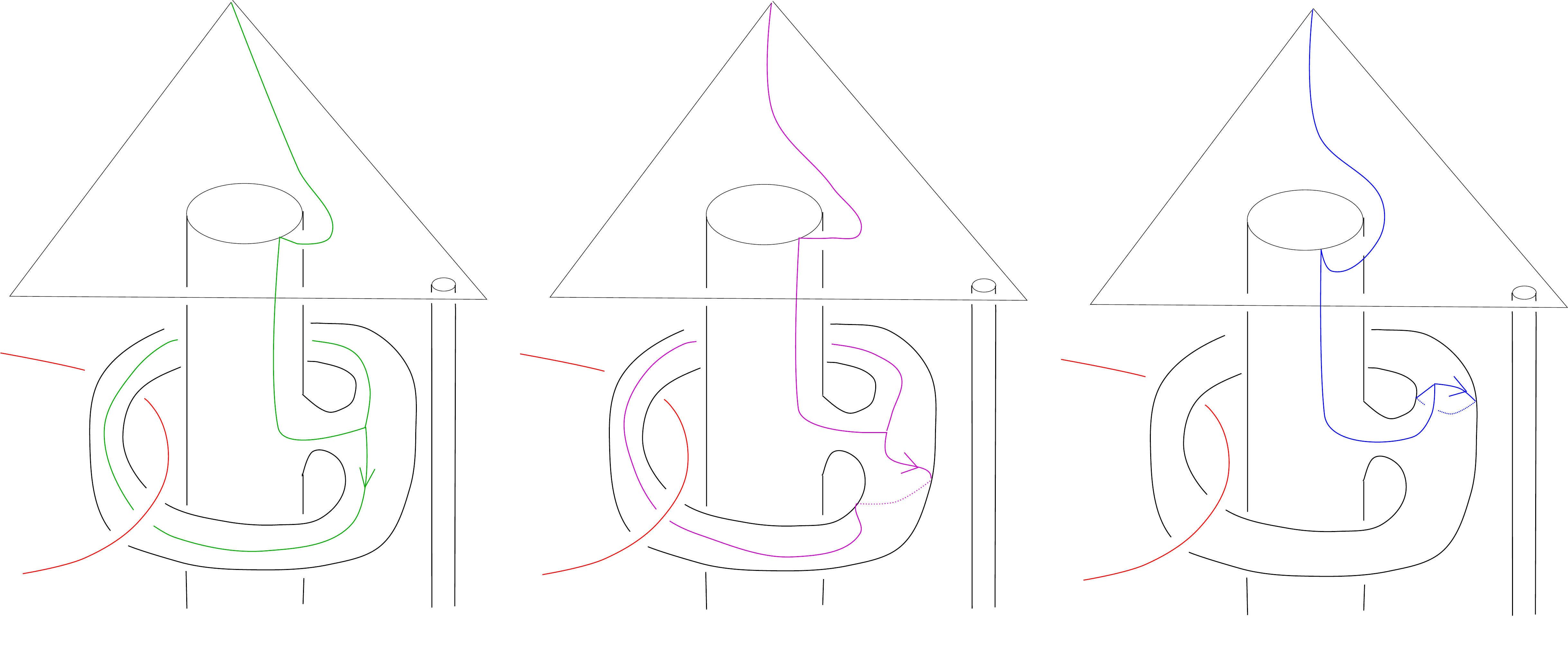
 \caption{\label{tube1fig} }
\end{center}
\end{figure}

\begin{prop}\label{abund4} There exist an integer $K$ and a collection ${\bf C}$ of $K$ perturbation curves on $F'$ inducing a $K$ parameter family of isotopies $\Phi_{\bf C,t},~t\in \RR^K$ of $\mathcal{M}(F')$  so that for each $\rho\in S^{-1}(0)$, the 
map 
$$S\circ \Phi_{\bf C,-}(\rho):\RR^K\to \RR^{n+3}$$
is a submersion near zero. 

Since $S^{-1}(0)$ is compact, there exists a neighborhood $U$ of $0$ in $\RR^K$ and a neighborhood of $S^{-1}(0)$ in $\mathcal{M}(F')$ so that $S\circ \Phi_{\bf C,-}(\rho)$ is a submersion for ${\bf t}, \rho$  in these neighborhoods.
 
\end{prop}
\noindent{\em Sketch of proof.}  That there exists a collection of perturbations such that the rank of the composite is (at least) $n+1$ is exactly Proposition \ref{abund2}, applies to $F'$ and the $n+1$ curves $A_1,\dots, A_n, h$.  Therefore, we simply need to demonstrate the existence of  extra perturbation curves  to boost the rank to $n+3$.  To  increase the rank, one considers perturbation curves that are disjoint from $A_1,\dots, A_n, h$ (although they may intersect the paths connecting them to the base point), and intersect $Y_1$ and $Y_2$ transversely in one point each.

Calculations similar to those in the proof of Proposition \ref{abund2} ensure that sufficiently many such perturbation curves can be found, provided the restriction of $\rho$ to $\pi_1(S_0)$ is irreducible. This is guaranteed by the fact that $R^\nat_{\pi_1\cup\pi}(Y',L')=R^\nat_{\pi_1\cup\pi}(Y',L')^{\ZZ/2, \ZZ/2}$.\qed

 \medskip
We can now complete the proof of Theorem \ref{earthm}.   Let $\pi_3$ be a small  perturbation such that $\mathcal{M}_{\pi_1\cup \pi\cup \pi_2}(X)$ satisfies the conclusion of Theorem \ref{herald}. Then
$\mathcal{M}_{\pi_1\cup \pi\cup \pi_2}(X)^{\ZZ/2,\ZZ/2}\to \mathcal{M}(F)^{\ZZ/2}$ is a Lagrangian immersion, and the composite
$$\mathcal{M}_{\pi_1\cup \pi\cup \pi_2}(X)^{\ZZ/2,\ZZ/2}\times \RR^K\xrightarrow{j} \mathcal{M}(F)^{\ZZ/2}\times \RR^K\xrightarrow{S\circ\Phi_{\bf C}}\RR^{n+3}$$ 
has zero as a regular value.   Then, just as in the proof of Theorem \ref{thm1}, there exists an arbitrarily small ${\bf t}\in \RR^K$ so that 
$$R^\nat_{\pi_1\cup\pi\cup\pi_2\cup\pi_{\bf t} }(Y',T') =
(S\circ \Phi_{\bf C,-}\circ j)^{-1}(0)\cap (\mathcal{M}_{\pi_1\cup \pi\cup \pi_2}(X)^{\ZZ/2,\ZZ/2}\times  \{{\bf t}\})$$
is a smooth compact  manifold (with empty boundary) which  Lagrangian immerses into $R((S^2)',\punctures)^{\ZZ/2}$.   Since $(Y,T)=(Y',T')$ as $n$-tangles, the theorem is proved. 
\end{proof}

\section{Epilogue}
We finish this article by explaining the context which motivates it. We make a few definitions, conjectures, and calculations.

Fix $n$ and let $P_n=R(S^2,\punctures)$. Consider a pair $(j_1:R_1\to P_n,j_2:R_2\to P_n)$ where $R_1$ is a smooth compact $2n-3$ manifold,  $R_2$ is a singular $2n-3$ manifold whose finitely many isolated singularities have neighborhoods which are cones on $\CC P^{n-2}$.  Assume that $j_1, j_2$ are Lagrangian immersions and that $j_2$ preserves the singular stratum and the radial direction near the singular points.

\medskip

\noindent{\bf Conjecture 1.} There is a well-defined Lagrangian-Floer theory associated to $(j_1:R_1\to P_n,j_2:R_2\to P_n)$, with homology 
$FH(j_1:R_1\to P_n,j_2:R_2\to P_n)$ depending on $j_i$ only up to Hamiltonian isotopy.

\medskip

A precise statement of Conjecture 1 would require additional restrictions, for example that $j_1, j_2$ are assumed to be {\em unobstructed} in the sense of \cite{FOOO}.  For the problems we discuss below we    want some additional structure such as a choice of Lagrangian subbundle $\Lambda\subset R(S^2,\punctures)^{\ZZ/2}$ and require $j_1$ and $j_2$ to be $\ZZ/4$ graded (in the sense of \cite{Seidel}, see \cite{HHK2}) with respect to $\Lambda$, so that $FH(R_1,R_2 )$ is a $\ZZ/4$ graded abelian group. For the special case of $n=2$, this conjecture is proved in detail in \cite{HHK2}, based on work of Abouzaid \cite{abouzaid}.

The identification of the spaces $R(S^2,\punctures)$ is unknown to the authors when $n>2$. When $n=2$ it is the pillowcase \cite{HHK1}.
\medskip
Now suppose that one is given a link $L$ in a closed oriented 3-manifold $Y$.  Assume further that a distinguished point is placed on one of the components of $L$ and that a separating 2-sphere $S^2\subset Y$ meets $L$ transversally in $2n$ points, missing the distinguished point. Call the component of $Y\setminus S^2$ containing this point $Y_1$, so that we have a decomposition 
\begin{equation}
\label{decomp}
(Y,L)=(Y_1,L_1)\cup_{(S^2, \punctures)}(Y_2,L_2).
\end{equation}
Then the results of this article show that, after placing an earring near the distinguished point and choosing appropriate small perturbations $\pi_1, \pi_2$, one obtains a pair of Lagrangian immersions.
$$R^\nat_{\pi_1}(Y_1,L_1)\to P_n\leftarrow R_{\pi_2}(Y_2,L_2)$$
\medskip

\noindent{\bf Conjecture 2.} These Lagrangian immersions are unobstructed and $\ZZ/4$ graded with respect to a certain Lagrangian field on $T_*P_n^{\ZZ/2}$, and the resulting Lagrangian-Floer homology 
$FH(R^\nat_{\pi_1}(Y_1,L_1),R_{\pi_2}(Y_2,L_2))$ is independent of $\pi_1$ and $\pi_2$ for small enough $\pi_1,\pi_2$.

\medskip

Theorem 7.1 of \cite{HHK1} shows that when $(Y_1,L_1)$ is a trivial 2-tangle, then with $\pi_1$ the perturbation curve of Figure \ref{earring3fig}, $R^\nat_{\pi_1}(Y_1,L_1)$ is a circle and the immersion 
$R^\nat_{\pi_1}(Y_1,L_1)\to P_2$ is unobstructed.  The circle is parameterized in Equation (\ref{values}) above, with a single double point, corresponding to  $\beta=0,\pi$.   It is straightforward to check that adding a single unknotted, unlinked strand to $L_1$ replaces $R^\nat_{\pi_1}(Y_1,L_1)$ by its product with $S^2$, corresponding to which traceless element the new meridian is sent to.  Hence, for a   trivial $n$-stranded tangle $(Y_1,L_1)=(D^2, n \text{ points})\times I$,
$$R^\nat_{\pi_1}(Y_1,L_1)=S^1\times (S^2)^{n-2}.$$
The restriction $R^\nat_{\pi_1}(Y_1,L_1)\to P_n$ is not an embedding.  The pair of points with the same image when $n=2$ becomes a pair of $(S^2)^{n-1}$ which map to the same $ (S^2)^{n-1}$ in $P_n$.   However, it seems likely that this immersion is unobstructed, and in fact we expect that the methods of this article may be used to find an arbitrarily small further perturbation $\pi$ so that $R^\nat_{\pi_1\cup \pi}(Y_1,L_1)\to P_n$ is an embedding.

In any case, by restricting to decompositions (\ref{decomp}) with $(Y_1,L_1)$ unknotted, one may define an invariant of $n$-tangles $(Y_2, L_2)$ taking $(Y_2,L_2)$ to $FH(R^\nat_{\pi_1}(Y_1,L_1),R_{\pi_2}(Y_2,L_2))$.  Conjecture 2 then implies the following.

\medskip

\noindent{\bf Conjecture 3.} The assignment 
$(Y_2,L_2)\mapsto FH(R^\nat_{\pi_1}(Y_1,L_1),R_{\pi_2}(Y_2,L_2))$ is a well-defined tangle invariant.
\medskip

Ample evidence of Conjecture 2 is provided in \cite{HHK2}, which builds on calculations in \cite{HHK1,FKP}, for   certain 2-tangle decompositions of  2-bridge knots and many torus knots, with $(Y_1,L_1)$ the trivial tangle.  For torus knots,  the unperturbed space $R (Y_2,L_2)$ is singular, and its singularities can be resolved by perturbations in topologically distinct ways, but in all computed examples, all sufficiently small perturbations result in isomorphic Lagrangian-Floer homology.

\medskip

The intersection of $R^\nat_{\pi_1}(Y_1,L_1)$ and $R_{\pi_2}(Y_2,L_2)$ (when transverse) is the finite, regular flat moduli space $R_\pi^\nat(Y,L)$.  This space is precisely the set of critical points of a certain $\pi$-perturbed Chern-Simons function $cs_\pi:\mathcal{A}\to S^1$, defined by Kronheimer-Mrowka \cite{KM1}, whose $\ZZ/4$ graded Morse homology they denote by $I^\nat(Y,L)$. Kronheimer-Mrowka  prove $I^\nat(Y,L)$ is an invariant of $(Y,L)$, independent of the choice of (sufficiently small) perturbation $\pi$.

It is not true in general that $FH(R^\nat_{\pi_1}(Y_1,L_1),R_{\pi_2}(Y_2,L_2))$ is independent of the  decomposition (\ref{decomp}) of  $(Y,L)$. It seems likely, however,  based on calculations,  that it depends only on $(Y,L)$  provided the fundamental groups of $Y_i\setminus L_i,~i=1,2$ are free. This happens, for example, if both $(Y_1,L_1)$ and $(Y_2,L_2)$ are trivial tangles (that is, (\ref{decomp}) is a $n$-bridge presentation of a link $L$ in $S^3$). Moreover, one can, in certain special cases, establish exact triangles for $FH(R^\nat_{\pi_1}(Y_1,L_1),R_{\pi_2}(Y_2,L_2))$ associated to skein triples $L_+, L_-, L_0$, and so it seems reasonable to conjecture:   
 \medskip

\noindent{\bf Conjecture 4.} For some class of nice decompositions  (\ref{decomp}), including   $n$-bridge decompositions   of  links in $S^3$, the assignment of the $\ZZ/4$ graded homology
$$(Y,L)\mapsto FH(R^\nat_{\pi_1}(Y_1,L_1),R_{\pi_2}(Y_2,L_2))$$ is a well-defined topological invariant of $(Y,L)$.

\medskip

When $(Y_2,L_2)$ is the trivial $n$-tangle, $R(Y_2,L_2)$ satisfies the conclusion of Theorem \ref{thm1} without the need for perturbations. In fact, $R(Y_2,L_2)=(S^2)^{n-1}/S^1$, where $S^1$ acts diagonally and by rotation on each factor. The $2^{n-1}$ singular points correspond to the representations which send each $A_i$ to $\pm\bbi$. The restriction map $R(Y^2,S^2)\to R(S^2,\punctures)$ is an embedding.

\medskip

Since all the perturbations considered in this article are holonomy perturbations, which apply to the Chern-Simons function as well,  the following  {\em Atiyah-Floer}-type  conjecture implies Conjecture 4, and would provide a link between the analytical approach of \cite{KM1,KM2} and the algebraic approach of the current article and \cite{HHK1,HHK2}.

\medskip

\noindent{\bf Conjecture 5.} For some class of nice decompositions  (\ref{decomp}), including  $n$-bridge  decompositions of  links in $S^3$, the assignment of the $\ZZ/4$ graded homology
$$(Y_2,L_2)\mapsto FH(R^\nat_{\pi_1}(Y_1,L_1),R_{\pi_2}(Y_2,L_2))$$ is naturally isomorphic to  $I^\nat(Y,L)$.

\medskip

Conjecture 5 is true for the 2-bridge knots and certain decompositions of torus knots, as described in \cite{HHK2}. We refer the reader to that article for examples, as well as for concrete illustrations of the   theorems in the present article. 

 In this context, we mention the article  \cite{JR} by Jacobsson and Rubinsztein  which takes a  different approach  towards the construction of a Lagrangian-Floer theory for knots and links in $S^3$  using traceless $SU(2)$  representation varieties.  In that article they work in the {\em extended traceless moduli space} of a $2n$-punctured 2-sphere, defined in \cite{GHLW} as a subset of the character variety  of the surface obtained by adding one more  puncture. 
To an element of the braid group, Jacobsson-Rubinsztein associate the graph of its induced action on this extended moduli space. This graph is Lagrangian, and by pairing it with the graph associated to the trivial braid, they produce a pair of Lagrangians of the
extended moduli space associated to a braid presentation of a knot or link.


\end{document}